\documentclass[12pt]{book}

\usepackage{amsmath,amsthm}
\usepackage{amssymb}
\usepackage{amsfonts}
\usepackage{amscd}

\usepackage{amsmath,amsthm}
\usepackage{amssymb}
\usepackage{amsfonts}
\usepackage{amscd}

\usepackage{color}
\usepackage[normalem]{ulem}

\usepackage{cancel}

\newtheorem{thm}{Theorem}[section]
\newtheorem{theorem}[thm]{Theorem}

\newtheorem{corollary}[thm]{Corollary}
\newtheorem{lemma}[thm]{Lemma}
\newtheorem{proposition}[thm]{Proposition}

\theoremstyle{definition}
\newtheorem{definition}[thm]{Definition}
\newtheorem{example}[thm]{Example}
\newtheorem{examples}[thm]{Examples}
\newtheorem{remark}[thm]{Remark}
\newtheorem{remarks}[thm]{Remarks}
\newtheorem{observation}[thm]{Observation}
\newtheorem{notation}[thm]{Notation}

\newcommand{\N} {\mathbb{N}}
\newcommand{\Z} {\mathbb{Z}}
\newcommand{\R} {\mathbb{R}}
\newcommand{\C} {\mathbb{C}}
\newcommand{\KK} {\mathbb{K}}
\newcommand{\bk} {\Bbbk}

\newcommand{\bt} {\mathbf{t}}
\newcommand{\bu} {\mathbf{u}}
\newcommand{\bG} {\mathbf{G}}

\newcommand{\bT} {\mathbf{T}}

\newcommand{\M} {\mathcal{M}}

\newcommand{\fg} {\mathfrak{g}}
\newcommand{\fh} {\mathfrak{h}}
\newcommand{\fk} {\mathfrak{k}}

\newcommand{\kzg} {K_{\Z}(\fg)} 
\newcommand{\kzgz} {K_\Z(\fg_0)}

\newcommand{\salg} {\mathrm{(salg)}}
\newcommand{\smflds} {\mathrm{(smflds)}}
\newcommand{\sets}{{(\mathrm{sets})}}
\newcommand{\grps} {\mathrm{(groups)}}
\newcommand{\spec}{{\hbox{\sl Spec}\,}}
\newcommand{\uspec}{\underline{\hbox{\sl Spec}\,}}

\newcommand{\Ad}{\mathrm{Ad}}
\newcommand{\ad}{\mathrm{ad}}
\newcommand{\Hom}{\mathrm{Hom}}

\newcommand{\lra} {\longrightarrow}
\newcommand{\ra}{\rightarrow}
\newcommand{\End}{\mathrm{End}}

\newcommand{\rGL}{\mathrm{GL}}

\newcommand{\Lie}{\mathrm{Lie}}
\newcommand{\str}{\mathrm{str}}
\newcommand{\tr}{\mathrm{tr}}
\newcommand{\e}{\mathrm{e}}
\newcommand{\rosp}{\mathfrak{osp}}
\newcommand{\rsl}{\mathfrak{sl}}
\newcommand{\rgl}{\mathfrak{gl}}

\newcommand{\cO}{\mathcal{O}}
\newcommand{\cC}{\mathcal{C}}

\newcommand{\aschemes} {\mathrm{(aschemes)}}

\newcommand{\tF}{\widetilde{F}}
\newcommand{\cF}{{\mathcal F}}

\newcommand{\wiDelta}{\widetilde{\Delta}}

\newcommand{\wA}{{A}}

\newcommand{\walpha}{{\alpha}}
\newcommand{\wbeta}{{\beta}}

\newcommand{\wdelta}{{\delta}}
\newcommand{\wDelta}{{\Delta}}
\newcommand{\lie} {\mathrm{(Lie)}}

\begin{document}


\centerline{\Large \bf CHEVALLEY SUPERGROUPS}

\bigskip

\centerline{R. Fioresi$^\flat$, F. Gavarini$^\#$}

\bigskip

\centerline{\it $^\flat$ Dipartimento di Matematica, Universit\`a
di Bologna } \centerline{\it piazza di Porta San Donato, 5  ---
I-40127 Bologna, Italy} \centerline{{\footnotesize e-mail:
fioresi@dm.unibo.it}}

\bigskip

\centerline{\it $^\#$ Dipartimento di Matematica, Universit\`a di
Roma ``Tor Vergata'' } \centerline{\it via della ricerca
scientifica 1  --- I-00133 Roma, Italy}

\centerline{{\footnotesize e-mail: gavarini@mat.uniroma2.it}}




\bigskip

{\small {\bf Abstract.}    In the framework of algebraic supergeometry, we give a construction of the scheme-theoretic supergeometric analogue of split reductive algebraic group-schemes, namely affine algebraic supergroups associated to simple Lie superalgebras of classical type.  In particular, all Lie superalgebras of both  {\sl basic\/}  and  {\sl strange\/}  types are considered.  This provides a unified approach to most of the algebraic supergroups considered so far in literature, and an effective method to construct new ones.
                                                       \par
   Our method follows the pattern of a suitable scheme-theoretic revisitation of Chevalley's construction of semisimple algebraic groups, adapted to the reductive case.  As an intermediate step, we prove an existence theorem for Chevalley bases of simple classical Lie superalgebras and a PBW-like theorem for their associated Kostant superalgebras.}



\bigskip

\tableofcontents
%
%

\bigskip

{\bf Acknowledgments.}
   The authors thank V.~S.~Varadarajan for suggesting the problem.  Also, they thank M.~Duflo, L.~Frappat, C.~Gasbarri, K.~Iohara, Y.~Koga, A.~Sciarrino, V.~Serga\-nova, P.~Sorba and R.~Steinberg for their valuable hints and suggestions.

\vskip5pt

   Finally, they are deeply grateful to the referee for many helpful comments and remarks.

\bigskip

%
%

 %
 %

\chapter{Introduction}

\bigskip

   In his work of 1955, Chevalley provided a combinatorial construction of all simple algebraic groups over any field. In particular, his method led to a proof of the existence theorem for simple algebraic groups and to new examples of finite simple groups which had escaped the attention of specialists in group theory. The groups that Chevalley constructed are now known as  {\it Chevalley groups}.  Furthermore, Chevalley's construction provided a description of all simple algebraic groups as group schemes over  $ \Z \, $.

\medskip

   In this work we adapt this philosophy to the setup of supergeometry, so to give an explicit construction of algebraic supergroups whose Lie superalgebra is of classical type over an arbitrary ring.  Our construction provides at one stroke the supergroups corresponding to the families  $ A(m,n) $,  $ B(m,n) $,  $ C(n) $,  $ D(m,n) $  of basic Lie superalgebras and to the families of strange Lie superalgebras  $ P(n) $,  $ Q(n) $,  as well as to the exceptional basic Lie superalgebras  $ F(4) $,  $ G(3) $,  $ D(2,1;a) $   --- for  $ \, a \in \Z \, $;  cf.~\cite{ga}  for the general case.  To our knowledge, supergroups corresponding to the exceptional Lie superalgebras have not previously appeared in the literature.

\medskip

   To explain our work, we first revisit the whole classical construction.

\smallskip

   Let  $ \fg $  be a finite dimensional simple (or semisimple)
Lie algebra over an algebraically closed field  $ \KK $
(e.g.~$ \KK = \C \, $).  Fix in  $ \fg $  a Cartan subalgebra  $ \fh \, $;
then a root system is defined, and  $ \fg $  splits into weight spaces indexed by the roots.  Also,  $ \fg $  has a special basis, called  {\sl Chevalley\/}  basis, for which the structure constants are integers, satisfying special conditions in terms of the root system.  This defines an integral form of  $ \fg \, $,  called  {\sl Chevalley Lie algebra}.
                                                          \par
   In the universal enveloping algebra of  $ \fg \, $,  there is a  $ \Z $--integral  form, called  {\sl Kostant algebra},  with a special ``PBW-like'' basis of ordered monomials, whose factors are divided powers of root vectors and binomial coefficients of Cartan generators, corresponding to elements of the Chevalley basis of  $ \fg \, $.
                                                          \par
   If  $ V $  is a faithful  $ \fg $--module, there is a  $ \Z $--lattice  $ \, M \subseteq V \, $,  which is stable under the action of the Kostant algebra.  Hence the Kostant algebra acts on the vector space  $ \, V_\bk := \bk \otimes_\Z M \, $  for any field  $ \bk \, $.  Moreover there exists an integral form  $ \fg_V $  of  $ \fg $  leaving the lattice invariant and depending only on the representation  $ V $  and not on the choice of the lattice.
                                                       \par
   Let  $ B $  be the Chevalley basis we fixed in  $ \fg \, $:  its elements are either root vectors   --- denoted  $ X_\alpha \, $  ---   or ``toral'' ones (i.e., belonging to  $ \fh \, $), denoted  $ H_i \, $.  For any root vector  $ X_\alpha \, $,  we take the exponential  $ \, \exp(t X_\alpha) \in \rGL(V_\bk) \, $   --- such an  $ X_\alpha $  acts nilpotently, so the exponential makes sense ---   for  $ \, t \in \bk \, $,  which altogether form an ``(additive) one-parameter subgroup'' of  $ \rGL(V_\bk) \, $.  The subgroup of  $ \rGL(V_\bk) $  generated by all the  $ \exp(t X_\alpha) $'s,  for all roots
and all  $ t \, $,  is the  {\sl Chevalley group\/}  $ G_V(\bk) $,  as
introduced by Chevalley.  This defines  $ G_V(\bk) $  set-theoretically,
as an abstract group; some extra work is required to show it is an algebraic
group and to construct its functor of points.  We refer the reader to  \cite{st},
\cite{borel},  \cite{hu}  for a comprehensive treatment of all of these aspects;
nevertheless, we briefly outline hereafter one possible strategy to cast this
construction in scheme-theoretical terms.

\smallskip

   The construction of the vector spaces  $ \, V_\bk := \bk \otimes_\Z M \, $,  of the linear groups  $ \rGL(V_\bk) $  and of the corresponding subgroups  $ G_V(\bk) $  generated by the exponentials
$ \, \exp(t X_\alpha) \in \rGL(V_\bk) \, $,  for  $ \, t \in \bk \, $,  can be directly extended, verbatim, by taking any ground (commutative, unital) ring  $ A $  instead of a field  $ \bk \, $.  This yields a group--valued functor  $ \, \mathrm{(alg)} \longrightarrow \grps \, $,  where  $ \mathrm{(alg)} $  stands for the category of commutative unital rings and  $ \grps $  is the category of groups: however, this functor happens to be not ``the right one''.  Thus one enlarges the set of generators of  $ G_V(A) \, $,  taking for  $ G_V(A) $  the subgroup of  $ \, \rGL(V_A) := \rGL(A \otimes_\Z M) \, $  generated by all the ``one-parameter subgroups'' given by the  $ \, \exp(t X_\alpha)  \, $   --- for all root vectors  $ \, X_\alpha $  and all coefficients  $ \, t \in A \, $  ---   {\sl and\/}  by a  ``maximal torus''  $ T_V(A) $  generated by ``multiplicative one-parameter subgroups''   ---  of  $ \rGL(V_A) $  ---   attached to the toral elements  $ H_i $  in the Chevalley basis  $ B \, $.  This defines a new functor  $ \, G_V : \mathrm{(alg)} \longrightarrow \grps \, $.
                                                       \par
   Now, the category  $ \mathrm{(alg)} $  is naturally antiequivalent
--- taking spectra ---   to the category  $ \mathrm{(aschemes)} $
of affine schemes: via this, the functor  $ G_V $  above defines a
functor $ \, G_V : \mathrm{(aschemes)}^\circ \longrightarrow \grps \, $,
i.e.~a ``presheaf'' on  $ \mathrm{(aschemes)} \, $.
%
%
 Then  $ \mathrm{(aschemes)} $  with the Zariski topology is a site, and
there exists a unique sheaf  $ \, \bG_V : \mathrm{(aschemes)}^\circ
\longrightarrow \grps \, $  suitably associated to  $ G_V \, $,
namely the sheafification of  $ G_V \, $.  Then
%
%
%
%
%
one has that  $ \bG_V $  is (the functor of points of) a group-scheme over  $ \Z \, $,  namely the unique one which is reductive and split associated to the Cartan datum encoded by  $ \fg $  and its representation  $ V \, $.

\medskip

   In the present work, following the strategy outlined above, we extend Chevalley's construction to the supergeometric setting.

\medskip

   In supergeometry the best way to introduce supergroups is via their functor of points.  Unlikely the classical setting, the points over a field of a supergroup tell us very little about the supergroup itself.  In fact such points miss the odd coordinates and describe only the classical part of the supergroup.  In other words, over a field we cannot see anything beyond classical geometry.  Thus we cannot generalize Chevalley's recipe as it was originally designed, but we rather need to suitably and subtly modify it introducing the functor of points language right at the beginning   --- reversing the order in which the classical treatment was developed ---   along the pattern we sketched above.  In particular, the functor of points approach  realizes an affine supergroup as a representable functor from the category of commutative superalgebras  $ \salg $  to  $ \grps \, $.  In this work, we shall first construct a functor from  $ \salg $  to  $ \grps \, $,  via the strategy we just mentioned; then we shall prove that this functor is indeed representable.

\medskip

   Our initial datum is a simple Lie superalgebra of classical type
(or a direct sum of finitely many of them, if one prefers), say  $ \fg \, $:  in our construction it plays the role of the simple (or semisimple) Lie algebra in Chevalley's setting.  We start by proving some basic results on  $ \fg $  (previously known only partially, cf.~\cite{ik}  and  \cite{sw})  like the existence of  {\sl Chevalley bases}, and a PBW-like theorem for the Kostant $ \Z $--form of the universal enveloping superalgebra.
                                                          \par
   Next we take a faithful  $ \fg $--module $ V \, $,  and we show that there exists a lattice $ M $  in  $ V $  fixed by the Kostant superalgebra and also by a certain integral form  $ \fg_V $  of  $ \fg \, $,  which again depends on $ V $  only.  We then define a functor  $ G_V \, $  from  $ \salg $  to  $ \grps $  as follows.  For any commutative superalgebra $ A $,  we define  $ G_V(A) $  to be the subgroup of  $ \rGL(V(A)) $   --- with  $ \, V(A) := A \otimes_\Z M \, $  ---   generated by the homogeneous one-parameter unipotent subgroups associated to the root vectors, together with the multiplicative one-parameter subgroups corresponding to elements in the Cartan subalgebra.  In this supergeometric setting, one must carefully define the homogeneous one-parameter subgroups, which may have three possible superdimensions:  $ \, 1|0 \, $,  $ \, 0|1 \, $  and  $ \, 1|1 \, $.  This also will be discussed.

\smallskip

   As a group-theoretical counterpart of the  $ \Z_2 $--splitting  $ \, \fg = \fg_0 \oplus \fg_1 \, $,  we find a factorization  $ \; G_V(A) = G_0(A) \, G_1^<(A) \cong G_0(A) \times G_1^<(A) \, $.  Here  $ G_0(A) $  is (roughly) a classical Chevalley-like group attached to  $ \fg_0 $  and  $ V $,  while  $ G_1^<(A) $  may be euristically thought of as exponential of  $ A_1 \otimes \fg_1 \, $.  In fact, we show that actually the functor  $ \, G_1^< : A \mapsto G_1(A) \, $  is representable and isomorphic to  $ \mathbb{A}_\bk^{0|\text{\it dim}(\fg_1)} \, $.
                                                                       \par
   Indeed, our result is more precise: in fact,  $ \fg_1 $  in turn splits into  $ \, \fg_1 \! = \fg_1^+ \oplus \, \fg_1^- \, $  according to the splitting of odd roots into positive and negative ones, and so at the group level we have  $ \, G_1^<(A) \cong G_1^{-,<}(A) \times G_1^{+,<}(A) \, $  and  $ \, G(A) \cong G_1^{-,<}(A) \times G_0(A) \times G_1^{+,<}(A) \, $,  resembling the classical ``big cell'' decomposition, which however in this context holds  {\sl globally}.

\smallskip

   The key ``drawback'' of the functor  $ G_V $  is that it is not representable, hence it is not (the functor of points of) an algebraic supergroup.  This occurs already at the classical level: one-parameter subgroups and maximal tori, defined via their functor of points, are not enough to generate split reductive algebraic groups (as group-schemes over  $ \Z \, $).  Hence again we need to pass to sheaves, namely taking the sheafification  $ \bG_V $  of the functor  $ G_V \, $,  which coincides with  $ G_V $  on local superalgebras (we provide at the end an appendix with a brief treatment of sheafification of functors).  In particular,  $ \bG_V $  inherits the factorization  $ \; \bG_V = \bG_0 \, \bG_1 \cong \bG_0 \times \bG_1 \, $,  with  $ \, \bG_1 = G_1 \, $  and  $ \bG_0 $  being a classical (split reductive) Chevalley-like group-scheme associated to  $ \fg_0 $  and  $ V \, $. More in detail, we find the finer factorization  $ \, \bG_V(A) = \bG_0(A) \times \bG_1^{-,<}(A) \times \bG_1^{+,<}(A) \, $  with  $ \, \bG_1(A) = \bG_1^{-,<}(A) \times \bG_1^{+,<}(A) \, $  and  $ \, \bG_1^{\pm,<}(A) = G_1^{\pm,<}(A) \, $.  As  $ \, \bG_1 = G_1 \, $  and  $ \bG_0 $  are representable, the above factorization
implies that  $ \bG_V $  is representable too, and so it is an algebraic supergroup.
We then take it to be, by definition, our ``Chevalley supergroup''.
                                                          \par
   In the end, we prove the functoriality in  $ V $  of our construction, and that, over any field  $ \bk \, $,  the Lie superalgebra  $ \Lie(\bG_V) $  is just  $ \, \bk \otimes \fg_V \, $.  In other words, we prove that a suitable super-version of Lie's Third Theorem does hold in the present case.

\bigskip

   We conclude with a description of the paper, which is organised as follows.

\medskip

%
%
%

   Chapter 2 collects a series of definition and results which we need
throughout the paper.  In particular, we set the framework of
``supergeometry'' briefly explaining how our functorial approach
stands in comparison with other approaches appearing in the literature.
We then recall the notion of Lie superalgebra and the description and classification of classical Lie superalgebras.  In the end we introduce ``homogeneous one-parameter supersubgroups'',
which are a natural generalization of the one-parameter subgroups in the ordinary Lie
theory and will be instrumental for our definition of Chevalley supergroups.

\medskip

   In chapter 3, we first recall the machinery, attached to any classical
Lie superalgebra  $ \fg \, $,  of Cartan subalgebras, root systems,
root vectors, root spaces, etc.  We then introduce the notion
of  ``Chevalley basis'' for  $ \fg \, $,  roughly a basis of  $ \fg $
made of Cartan elements and root vectors with some integral
conditions on their brackets.  The main result in this chapter
is the existence theorem for Chevalley bases for classical Lie superalgebras.

\medskip

   In chapter 4 we define the ``Kostant superalgebra''  $ \kzg $  attached
to  $ \fg $  and to a fixed Chevalley basis  $ B \, $,  as the unital
subalgebra of  $ U(\fg) $  generated over $\Z$ by the odd root vectors,
the divided powers in the even root vectors and the binomial coefficients
in the Cartan elements (in  $ B \, $).  We describe the ``commutation rules''
among the generators of  $ \kzg $,  and we use these rules
to prove a ``PBW-like Theorem'' for  $ \kzg \, $.

\medskip

   Chapter 5 is devoted to the construction of ``Chevalley supergroups''.
First we show that any rational faithful semisimple $ \fg $--module  $ V $,
under mild assumptions, contains a  $ \Z $--integral  form  $ M $,
which is  $ \kzg $--stable.  This allows to define a functorial recipe
yielding an  $ (A \otimes \kzg) $--representation  on  $ \, A \otimes M \, $
by scalar extension to any commutative superalgebra  $ A \, $.
Through this, we take all the one-parameter supersubgroups   ---
inside the supergroup functor  $ \rGL(V) $  associated to  $ V $  and  $ M $
---   corresponding to the root vectors and the Cartan elements in
the Chevalley basis, and the subgroup they generate.
This defines a supergroup functor  $ G_V \, $,  whose sheafification
$ \bG_V $ we define {\sl Chevalley supergroup}.
                                                 \par
%
%
 The rest of the chapter is devoted to prove that the supergroup functor
$ \bG_V $  is representable, in other words, it is the functor of
points of an affine algebraic supergroup scheme.  This is a very subtle
point that we settle with a careful analysis of the structure of $\bG_V(A)$.
%
%
                                                 \par
   In the last part of the chapter, we show that the supergroup  $ \bG_V $
actually depends only on $ V $  and not on the many choices made along with the
construction, and moreover has good functoriality properties with respect
to  $ V \, $.  In particular, we explain the relation between different
Chevalley supergroups built upon the same classical Lie superalgebra
$ \fg $  and different  $ \fg $--modules  $ V \, $.
Also, we prove Lie's Third Theorem for  $ \bG_V \, $,
that is to say  $ \, \Lie(\bG_V) = \fg \, $.

\medskip

   Finally, in chapter 6 we treat apart the ``singular'' cases, namely those when  $ \fg $  is of type  $ A(1,1) $,  $ P(3) $  or  $ Q(n) \, $.  These were discarded before, because in these cases some roots have non-trivial multiplicity (i.e., root spaces have multiplicity greater than one), so additional care is needed.  Indeed, almost nothing changes when dealing with  $ A(1,1) $  and  $ P(3) \, $:  roughly, it is enough to adopt a careful notation.  Instead, the definition of Chevalley basis must be carefully rethought in case  $ Q(n) \, $:  the construction and description of  $ \kzg $  change accordingly.  The rest of our construction of Chevalley supergroups then goes through essentially unchanged, up to minor details.
%
%

%


  %
  %

\chapter{Preliminaries}  \label{preliminaries}

 {\it
   In this chapter, we collect a series of definitions and
results which we shall need throughout all the paper.
In particular, we establish the terminology and the framework
of ``supergeometry'', we are going to work in. Supergeometry is quite a
huge domain: here we bound ourselves to mention only those particular
notions which we shall need later on in our work. We are going
to give only few   glimpses to the other parts of the theory
just to make clear what is the link between our point of view
and some of the many other different approaches found in the literature.
For instance, we spend just a few lines on the key notions of superspaces,
supermanifolds and superschemes, and then we quickly get to
explain their relationship with supergroup functors, through the notion
of ``functor of points'', which is our approach to the construction
of Chevalley supergroups.  Similarly, we introduce the notion of
{\sl classical\/}  Lie superalgebras, but we refer the reader to
the literature for further details.
%
%

\smallskip

   Results here are generally quoted without proofs.
For supergeometry, one can refer to the vast literature on the subject;
nonetheless, we prefer to mostly quote the monograph  \cite{ccfd},
when we mention precise statements.  For Lie superalgebras
our main reference will be Kac's works, such as his foundational
work \cite{ka}, but occasionally may quote additional ones.

\smallskip

   At the end of the chapter, we have a section on
``homogeneous one-parameter supersubgroups''.  These are supersubgroups
in a given supergroup  $ \bG $,  which ``integrate''
 a Lie supersubalgebra
$ \frak{k} $  of the tangent Lie superalgebra  $ \, \fg := \Lie(\bG) \, $  of
$ \, \bG $  generated, as a Lie superalgebra, by a single, homogeneous (even or odd)
vector in  $ \fg \, $.  This notion is modelled on the classical one-parameter
group definition in the standard Lie theory,
yet the situation that actually shows up is somewhat richer.
In fact, (super)dimensions  $ 1|0 \, $,  $ 0|1 \, $,  $ 1|1 $  make their appearance, thus
making the terminology  ``one-parameter supersubgroups''
a bit awkward.  We nevertheless prefer to keep it, to mark the striking
analogies with the classical theory.}

%
%

\bigskip

 \section{Superalgebras, superspaces, supergroups}  \label{first_preliminaries}

\medskip

   Let  $ \bk $  be a unital, commutative ring.

\vskip7pt

   We call  {\it  $ \bk $--superalgebra\/}  any associative, unital  $ \bk $--algebra  $ A $  which is  $ \Z_2 $--graded;  that is,  $ A $  is a  $ \bk $--algebra  graded by the two-element group  $ \Z_2 \, $.  Thus  $ A $  splits as  $ \; A = A_0 \oplus A_1 \; $,  and  $ \; A_a \, A_b \subseteq A_{a+b} \; $.  The  $ \bk $--submodule  $ A_0 $  and its elements are called  {\it even},  while  $ A_1 $  and its elements  {\it odd}.  By  $ p(x) $  we denote the  {\sl parity\/}  of any homogeneous element  $ \, x \in A_{p(x)} \, $.  Clearly,
$ \bk $--superalgebras  form a category, whose morphisms are all those in the category of algebras which preserve the unit and the  $ \Z_2 $--grading.  At last, for any  $ \, n \in \N \, $  we call  $ A_1^{\,n} $  the  $ A_0 $--submodule  of  $ A $  spanned by all products  $ \, \vartheta_1 \cdots \vartheta_n \, $  with  $ \, \vartheta_i \in A_1 \, $  for all  $ i \, $,  and  $ A_1^{(n)} $  the unital subalgebra of  $ A $  generated by  $ A_1^{\,n} \, $.
                                                   \par
   A superalgebra  $ A $  is said to be  {\it commutative\/}  iff  $ \; x y = (-1)^{p(x)p(y)} y x \; $  for all homogeneous  $ \, x $,  $ y \in A \; $.  We denote  by $ \salg $  the category of commutative superalgebras; if necessary, we shall stress the role of  $ \bk $  by writing  $ \salg_\bk \, $.

\bigskip

\begin{definition}
  A  {\it superspace}  $ \, S = \big( |S|, \cO_S \big) \, $  is a topological space  $ |S| $  endowed with a sheaf of commutative superalgebras  $ \cO_S $  such that the stalk  $ \cO_{S,x} $  is a local superalgebra for all  $ \, x \in |S| \, $.
                                        \par
   A  {\it morphism}  $ \; \phi: S \lra T \; $  of superspaces consists of a pair $ \; \phi = \big( |\phi|, \phi^* \big) \, $,  where  $ \; \phi : |S| \lra |T| \; $  is a morphism of topological spaces and  $ \; \phi^* : \cO_T \lra \phi_* \cO_S \; $  is a sheaf morphism such that  $ \; \phi_x^* \big( {\mathbf{m}}_{|\phi|(x)} \big)
= {\mathbf{m}}_x \; $  where  $ {\mathbf{m}}_{|\phi|(x)} $ and  $
{\mathbf{m}}_{x} $  are the maximal ideals in the stalks $ \cO_{T,
\, |\phi|(x)} $  and  $ \cO_{S,x} $  respectively and $\phi_x^*$
is the morphism induced by $\phi^*$ on the stalks.
  Here as usual  $ \phi_* \cO_S $  is the sheaf on  $ |T| $  defined as  $ \, \phi_* \cO_S(V) := \cO_S(\phi^{-1}(V)) \, $.
\end{definition}

\medskip

  Given a superspace  $ \, S = \big( |S|, \cO_S \big) \, $,  let  $ \cO_{S,0} $  and  $ \cO_{S,1} $  be the sheaves on  $ |S| $  defined as follows:  $ \; \cO_{S,0}(U) := {\cO_{S}(U)}_0 \; $,  $ \; \cO_{S,1}(U) := {\cO_{S}(U)}_1 \; $  for each open subset  $ U $  in  $ |S| \, $.  Then  $ \cO_{S,0} $  is a sheaf of ordinary commutative algebras,
while  $ \cO_{S,1} $  is a sheaf of  $ \cO_{S,0} $--modules.

\medskip

\begin{definition}
   A  {\it superscheme\/}  is a superspace  $ \, S := \big( |S|, \cO_S \big) \, $  such that $ \, \big( |S|, \cO_{S,0} \big) \, $ is an ordinary scheme and  $ \cO_{S,1} $  is a quasi-coherent sheaf of  $ \cO_{S,0} $--modules.  A  {\it morphism\/}  of superschemes is a morphism of the underlying superspaces.
\end{definition}

\medskip

\begin{definition}  \label{spec}
  Let  $ \, A \in \salg \, $  and let  $ \cO_{A_0} $  be the structural sheaf of the ordinary scheme  $ \, \uspec(A_0) = \big( \spec(A_0), \cO_{A_0} \big) \, $,  where  $ \spec(A_0) $  denotes the prime spectrum of the commutative ring  $ A_0 \; $.  Now  $ A $  is a module over  $ A_0 \, $,  so we have a sheaf  $ \cO_A $  of  $ \cO_{A_0} $--modules  over  $ \spec(A_0) $  with stalk  $ A_p \, $,  the  $ p $--localization of the  $ A_0 $--module  $ A \, $,  at the prime  $ \, p \in \spec(A_0) \, $.
                                       \par
   We define the superspace  $ \; \uspec(A) := \big( \spec(A_0), \cO_A \big) \; $.  By its very definition  $ \uspec(A) $  is a superscheme.

\medskip

   Given  $ \, f : A \lra B \, $  a superalgebra morphism, one can define  $ \, \uspec(f) : \uspec(B) \lra \uspec (A) \, $  in a natural way, very similarly to the ordinary setting, thus making  $ \uspec $  into a functor  $ \, \uspec : \salg \lra \sets \, $,  where  $ \salg $  is the category of (commutative) superalgebras and  $ \sets $  the category of sets (see  \cite{ccfd},  ch.~5, or  \cite{eh},  ch.~1, for more details).
\end{definition}

\medskip

\begin{definition}
  We say that a superscheme  $ X $  is  \textit{affine\/}  if it is isomorphic to  $ \, \uspec(A) \, $  for some commutative superalgebra  $ A \, $.
\end{definition}

\medskip

   Clearly any superscheme is locally isomorphic to an affine superscheme.

\medskip

\begin{example}
   The  {\sl affine superspace\/}  $ \, \mathbb{A}_\bk^{p|q} \, $,  also denoted  $ \, \bk^{p|q} \, $,  is defined   --- for each  $ \, p \, $, $ q \in \N \, $  ---   as  $ \, \mathbb{A}_\bk^{p|q} := \big( \mathbb{A}_\bk^{p} , \cO_{\mathbb{A}_\bk^{p|q}} \big) \, $,  with
 \vskip-3pt
  $$  \cO_{\mathbb{A}_\bk^{p|q}}\big|_U  \, = \;  \cO_{\mathbb{A}^p_\bk}\big|_U \otimes \bk[\xi_1 \dots \xi_q] \;\; ,  \qquad U \hbox{ open in } \bk^p  $$
\noindent
 where  $ \bk[\xi_1 \dots \xi_q] $  is the exterior (or ``Grassmann'') algebra generated by  $ \xi_1 $,  $ \dots $,  $ \xi_q \, $,  and  $ \cO_{\mathbb{A}^p_\bk} $  denotes the sheaf of polynomial functions on the classical affine space  $ \, \mathbb{A}^p_\bk := \bk^p \; $.  Indeed, this is an example of affine superscheme, because  $ \; \mathbb{A}_\bk^{p|q} \, \cong \, \uspec \big( \bk[x_1,\dots,x_p] \otimes_\bk \bk[\xi_1 \dots \xi_q] \big) \; $.
\end{example}

\medskip

   The concept of supermanifold provides another important example of superspace.  While our work is mainly focused on the algebraic category, we neverthless want to briefly introduce the differential setting, since our definition of Chevalley supergroup is modelled on the differential homogeneous
one-parameter subgroups, as we shall see in section  \ref{che-sgroup}.

\medskip

   We start with an example describing the local model of a supermanifold.  Hereafter, when we speak of {\sl supermanifolds},  we assume  $ \bk $  to be  $ \R $  or  $ \C \, $.

\smallskip

\begin{example}
   We define the superspace  $ \bk^{p|q} $  as the topological space  $ \bk^p $  endowed with the following sheaf of superalgebras.  For any open subset  $ \, U \subseteq \bk^p \, $  we set
 $ \; \cO_{\bk^{p|q}}(U) := \cO_{\bk^p}(U) \otimes \bk\big[\xi^1 \dots \xi^q\big] \, \phantom{\Big|} $
%
%
where  $ \cO_{\bk^p} $  denotes here the sheaf of smooth,
resp.~analytic, functions on  $ \bk^p $  when $ \, \bk = \R \, $,
resp.~$ \, \bk = \C \, $.
\end{example}

\smallskip

\begin{definition}  \label{supermanifold}
%
%
 A  {\it supermanifold\/}  of dimension  $ \, p|q \, $  is a superspace
$ \, M = (|M|,\cO_M) \, $  which is locally isomorphic (as
superspace) to  $ \, \bk^{p|q} \, $;  that is, for all  $ \, x \in
|M| \, $ there exist an open neighborhood  $ \, V_x \subset |M| \,
$  of $ x $  and an open subset  $ \, U \subseteq \bk^p \, $  such
that $ \; {\cO_{M}}\big|_{V_x} \cong {\cO_{k^{p|q}}}\big|_U \; $.
A  {\it morphism\/}  of supermanifolds is simply a morphism of
superspaces.  Supermanifolds, together with their morphisms, form
a category that we denote with  $ \smflds \, $.
\end{definition}

\medskip

   The theory of supermanifolds resembles very closely the classical theory.
More details on the basic facts of supergeometry can be found for
example in  \cite{vsv},  Ch.~4, or in \cite{ccfd}, Ch.~3--4.  Here
instead, we turn now to examine the notion of functor of points,
both in the algebraic and the differential category.

\medskip

\begin{definition}
   Let  $ X $  be a superscheme.  Its  {\it functor of points\/}  is the functor  $ \; h_X : \salg \lra \sets \; $  defined as  $ \; h_X(A) \, := \, \Hom \big(\, \uspec(A) \, , X \big) \; $  on the objects and as  $ \; h_X(f)(\phi) := \phi \circ \uspec (f) \; $  on the arrows.
If  $ h_X $  is group valued, i.~e.~it is valued
in the category  $ \grps $  of groups, we say that  $ X $  is a  {\it
supergroup}. When  $ X $  is affine, this is equivalent to the
fact that $ \, \cO(X) \, $   --- the superalgebra of global
sections of the structure sheaf on  $ X $  ---   is a
(commutative)  {\sl Hopf superalgebra}. More in general, we call  {\it supergroup functor\/}  any functor  $ \; G : \salg \lra \grps \; $.
                                                \par
   Any representable supergroup functor is the same as an affine supergroup: indeed, the former corresponds to the functor of points of the latter.
                                                \par
   Following a customary abuse of notation, we shall then use the same letter to denote both the superscheme  $ X $  and its functor of points  $ h_X \, $.
\end{definition}

\smallskip

   Similarly we can define the functor of points for supermanifolds.

\medskip

\begin{definition}
  For any supermanifold  $ M \, $,  we define its  {\it functor of points\/}  $ \; h_M : \smflds^\circ \lra \sets \; $,  where  $ \smflds^\circ $  denotes the opposite category to  $ \smflds \, $,  as follows:
                                                        \par
   ---  $ \; M \mapsto h_M(T) := \Hom \big(\, T , M \big) \; $  for any object  $ M $  in  $ \smflds^\circ \, $,
                                                        \par
   ---  $ \; h_M(f) : \phi \mapsto h_M(f)(\phi) := \phi \circ f \; $  for any arrow  $ \, f \in \Hom \big(\, T' , T \big) \, $  in  $ \smflds^\circ $,  \, and any  $ \, \phi \in \Hom \big(\, T' , M \big) \, $.
                                                        \par
   If the functor  $ h_M $  is group valued we say that  $ M $  is a  {\it Lie supergroup}.
\end{definition}

\medskip

   The importance of the functor of points is spelled out by a version of Yoneda's Lemma, that essentially tells us that the functor of points recaptures all the information carried by the supermanifold or the superscheme.

\smallskip

\begin{proposition} ({\sl Yoneda's Lemma})
                                                     \par
   Let  $ \cC $  be a category,  $ \cC^\circ $  its opposite category, and two objects  $ M $,  $ N $  in  $ \cC \, $.  Consider the functors
 $ \; h_M : \cC^\circ \! \lra \sets \; $  and  $ \; h_N : \cC^\circ \! \lra \sets \; $
defined by  $ \; h_M(T) := \Hom \big(\, T , M \big) \, $,  $ \; h_N(T) := \Hom \big(\, T , N \big) \; $  on any object  $ T $  in  $ \C^\circ $  and
by  $ \; h_M(f)(\phi) := \phi \circ f \, $,  $ \; h_N(f)(\psi) := \psi \circ f \; $  on any arrow  $ \, f \in \Hom \big(\, T' , T \big) \, $  in  $ \smflds^\circ $,  for any  $ \, \phi \in \Hom \big(\, T' , M \big) \, $  and  $ \, \psi \in \Hom \big(\, T' , N \big) \, $.
                                                     \par
   Then there exists a one-to-one correspondence between the natural transformations  $ \, \big\{ h_M \lra h_N \big\} \, $  and the morphisms  $ \, \Hom(M,N) \, $.
\end{proposition}

\medskip

   This has the following immediate application to supermanifolds: two supermanifolds are isomorphic if and only if their functors of points are.

\smallskip

   The same is true also for superschemes even if, with our definition of
their functor of points, this is not immediately clear.  In fact,
given a superscheme  $ X $,  we can give another definition of
functor of points, equivalent to the previous one, as the functor
from the category of superschemes to the category of sets, defined
as  $ \; T \lra \Hom(T,X) \; $.  Now, Yoneda's Lemma tells us that
two superschemes are isomorphic if and only if their functors of
points are.
                                                   \par
   For more details on functors of points in the two categories, and the equi\-valence of the two given definitions in the algebraic setting, see  \cite{ccfd},  Ch.~3--5.

\smallskip

   In the present work, we shall actually consider only affine supergroups,
which we are going to describe mainly via their functor of points.

\smallskip

   The next examples turn out to be very important in the sequel.

\bigskip

\begin{examples} {\ }
 \vskip9pt
   {\it (1)} \,  Let  $ V $  be a super vector space.  For any superalgebra  $ A $  we define  $ \; V(A) \, := \, {(A \otimes V)}_0 \, = \, A_0 \otimes V_0 \oplus A_1 \otimes V_1 \; $.  This is a representable functor in the category of superalgebras, whose representing object is  $ {\hbox{Pol}}(V) \, $,  the algebra of polynomial functions on  $ V \, $. Hence any super vector space can
be equivalently viewed as an affine superscheme.
 \vskip5pt
   {\it (2)} \,  {\sl  $ \rGL(V) $  as an algebraic supergroup}.  Let  $ V $  be a finite dimensional super vector space of dimension $p|q$.  For any superalgebra  $ A \, $,  let  $ \, \rGL(V)(A) := \rGL\big(V(A)\big) \, $  be the set of isomorphisms  $ \; V(A) \lra V(A) \; $.  If we fix a homogeneous basis for  $ V \, $,  we see that  $ \, V \cong \bk^{p|q} \; $;  \, in other words,  $ \, V_0 \cong \bk^p \, $  and  $ \, V_1 \cong \bk^q \, $.  In this case, we also denote $ \, \rGL(V) \, $  with  $ \, \rGL(p|q) \, $.  Now,  $ \rGL(p|q)(A) $  is the group of invertible matrices of size  $ (p+q) $  with diagonal block entries in  $ A_0 $  and off-diagonal block entries in $ A_1 \, $.  It is well known that the functor  $ \rGL(V) $  is representable; see (e.g.),  \cite{vsv}, Ch.~3,  for further details.
 \vskip5pt
   {\it (3)} \,  {\sl  $ \rGL(V) $  as a Lie supergroup}.  Let  $ V $  be a super vector space of dimension  $ p|q $  over  $ \R $  or  $ \C \, $.  For any supermanifold  $ T $,  define  $ \rGL(V)(T) $  as the set of isomorphisms  $ \, V(T) \lra V(T) \, $;  by an abuse of notation we shall use the same symbol to denote  $ \rGL $  in the algebraic and the differential setting.  When we are writing  $ V(T) $,  we are taking  $ V $  as a supermanifold, hence  $ \, V(T) = \Hom(T,V) \, $.  By a result in  \cite{ko}  (\S 2.15, page 208), we have that $\; \Hom(T,V) = \Hom \big( \cO_V(V), \cO_T(T) \big) \; $.  If we fix a homogeneous basis for  $ V \, $,  $ \Hom(T,V) $
can be identified with the set of all  $
(p+q)$--uples  with entries in  $ \cO_T(T) $,  the first  $ p $
entries being even and the last  $ q $  odd.  As before,  $
\rGL(V)(T) $  can be identified with the group of  $ (p+q) \times
(p+q) $  invertible matrices, whose diagonal blocks have entries
in  $ \cO_T(T)_0 $,  while the off-diagonal blocks have entries in
$ \cO_T(T)_1 \, $.  Now again,  $ \rGL(V) $  is a representable functor
(see \cite{vsv},  Ch.~6), i.e.~there exists a supermanifold   ---
again denoted by  $ \rGL(V) $  ---   whose functor of points is exactly
$ \rGL(V) \, $.
\end{examples}

\bigskip
\bigskip
\bigskip

 \section{Lie superalgebras}  \label{Lie-superalgebras}

\medskip

   From now on, we assume  $ \bk $  to be such that 2 and 3 are not zero and they are
not zero divisors in  $ \bk \, $.
Moreover,  {\sl all\/  $ \bk $--modules  hereafter will be assumed to have no  $ p \, $--torsion  for  $ \, p \in \{2,3\} \, $}.
%
%
%

\medskip

\begin{definition}
  Let  $ \, \fg = \fg_0 \oplus \fg_1 \, $  be a super (i.e.,  $ \Z_2 $--graded)  $ \bk $--module  (with no  $ p \, $--torsion  for  $ \, p \in \{2,3\} \, $,  as mentioned above).  We say that  $ \fg $  is a Lie superalgebra, if we have a bracket  $ \; [\ ,\ ] : \fg \times \fg \lra \fg \; $  which satisfies the following properties (for all  $ x $,  $ y \in \fg \, $  homogeneous):
 \vskip5pt
   {\it (1)} \  {\sl Anti-symmetry:}  $ \qquad \displaystyle{ [x,y] \, + \, {(-1)}^{p(x) \, p(y)}[y,x] \; = \; 0 } $
 \vskip5pt
   {\it (2)} \  {\sl Jacobi identity:}
  $$  {(-\!1)}^{p(x) \, p(z)} \, [x,[y,z]] \, + \, {(-\!1)}^{p(y) \, p(x)} \, [y,[z,x]] \, + \, {(-\!1)}^{p(z) \, p(y)} \, [z,[x,y]] \; = \; 0  $$
\end{definition}

\medskip

\begin{example}
  Let  $ \, V = V_0 \oplus V_1 \, $  be a  {\sl free\/}  super  $ \bk $--module,  and consider  $ \End(V) \, $,  the endomorphisms of  $ V $  as an ordinary  $ \bk $--module.  This is again a free super  $ \bk $--module,  $ \; \End(V) = \End(V)_0 \oplus \End(V)_1 \, $,  \; where $ \End(V)_0 $  are the morphisms which preserve the parity, while  $ \End(V)_1 $  are the morphisms which reverse the parity.  If  $ V $  has finite rank, and we choose a basis for  $ V $  of homogeneous elements (writing first
the even ones), then  $ \End(V)_0 $  is the set of all diagonal block matrices, while  $ \End(V)_1 $  is the set of all off-diagonal block matrices.  Thus  $ \End(V) $  is a Lie superalgebra with bracket
  $$  [A,B]  \; := \;  A B - {(-1)}^{|A||B|} \, B A
\eqno \text{for all homogeneous \ }  A, B \in \End(V) \; .  $$
   The standard example is $ \, V := \bk^{p|q} = \bk^p \oplus \bk^q \, $,  with  $ \, V_0 := \bk^p \, $  and  $ \, V_1 := \bk^q \, $.  In this case, we also write  $ \; \End\big(\bk^{p|q}\big) := \, \End(V) \; $  or  $ \; \rgl(p\,|q) := \End(V) \; $.
                                                                     \par
   In  $ \End(V) $  we can define the  {\it supertrace\/}  as follows:
  $$  \str \begin{pmatrix}
              A  &  B  \\
              C  &  D
           \end{pmatrix}  := \;  \tr(A) - \tr(D)  \quad .  $$
\end{example}

\medskip

   For the rest of this section, we assume  $ \, \bk \, $  to be an algebraically closed field of characteristic zero   --- though definitions make sense in general.

\medskip

\begin{definition}
  A non-Abelian Lie superalgebra  $ \fg $  is called a  \textit{classical Lie superalgebra\/}  if it is simple, i.e.~it has no nontrivial (homogeneous) ideals, and  $ \fg_1 $  is completely reducible as a  $ \fg_0 $--module.  Furthermore,  $ \fg $  is said to be  \textit{basic\/}  if, in addition, it admits a non-degenerate, invariant bilinear form.
\end{definition}

\medskip

\begin{examples} (cf.~\cite{ka}, \cite{sc})
 \vskip2pt
   {\it (1)} ---  $ \rsl(m|\,n) \, $.  Define  $ \, \rsl(m|n) \, $  as the subset of  $ \rgl(m|n)  $  all matrices with supertrace zero.  This is a Lie subalgebra of  $ \rgl(m|n) \, $,  with  $ \Z_2 $--grading
  $$  \rsl(m|\,n)_0 \, = \, \rsl(m) \oplus \rsl(n) \oplus \rgl(1)  \;\; ,  \qquad
\rsl(m|\,n)_1 \, = \, f_m \otimes f'_n \, \oplus \, f'_m \otimes f_n  $$
where  $ f_r $  is the defining representation of  $ \rsl(r) $
and  $ f'_r $  is its dual (for any  $ r \, $).  When  $ \, m \neq
n \, $  this Lie superalgebra is a  {\sl classical\/}  one.

\smallskip

   {\it (2)} ---  $ \rosp(p\,|q) \, $.  Let  $ \phi \, $  denote a nondegenerate consistent supersymmetric bilinear form in  $ \, V := \bk^{p|q} \, $.  This means that  $ V_0 $  and  $ V_1 $  are mutually orthogonal and the restriction of  $ \phi $  to  $ V_0 $  is symmetric and to  $ V_1 $  is skewsymmetric (in particular,  $ \, q = 2\,n \, $  is even).  We define in  $ \rgl(p\,|q) $  the subalgebra  $ \; \rosp(p\,|q) := {\rosp(p ,|q)}_0 \oplus {\rosp(p\,|q)}_1 \; $  by setting, for all  $ \, s \in \{0,1\} \, $,
  $$  {\rosp(p\,|q)}_s  \, := \,  \Big\{ \ell \in \rgl(p\,|q) \;\Big|\; \phi\big(\ell(x),y\big) = -{(-1)}^{s\,|x|} \, \phi\big(x, \ell(y)\big) \; \forall \, x, y \in \bk^{p|q} \,\Big\}  $$
and we call  $ \rosp(p\,|q) $  an  {\it orthosymplectic\/}  Lie superalgebra.  Again, all the  $ \rosp(p|\,q)$'s  are  {\sl classical\/}  Lie superalgebras, actually Lie supersubalgebras of  $ \rgl(p\,|q) \, $.  Note also that  $ \, \rosp(0|\,q) \, $  is the symplectic Lie algebra, while  $ \rosp(p|\,0) $  is the  orthogonal Lie algebra.
                                                              \par
   We can also describe the explicit matrix form of the elements of  $ \rosp(p\,|q) \, $.  First, note that, in a suitable block form, the bilinear form  $ \phi $  has matrix
  $$  \phi  \; = \,  \begin{pmatrix}
               0   &  I_m  &  0  &   0   &  0  \\
              I_m  &   0   &  0  &   0   &  0  \\
                0  &   0   &  1  &   0   &  0  \\
                0  &   0   &  0  &   0   &  I_n  \\
                0  &   0   &  0  & -I_n  &  0
                     \end{pmatrix}
  \quad ,  \qquad
      \phi  \; = \,  \begin{pmatrix}
               0   &  I_m  &   0   &  0  \\
              I_m  &   0   &   0   &  0  \\
               0   &   0   &   0   & I_n \\
               0   &   0   & -I_n  &  0
                     \end{pmatrix}  $$
according to whether  $ \; (p,q) = (2 m \! + \! 1, 2n) \, $  \; or  $ \; (p,q) = (2m,2n) \; $.  Then, in the block form given by the partition of rows and columns according to  $ \; (p+q) = m \! + \! m \! + \! 1 \! + \! n \! + \! n \; $  or to  $ \; (p+q) = m \! + \! m \! + \! n \! + \! n \; $  (depending on the parity of  $ p \, $),  the orthosymplectic Lie superalgebras  $ \rosp(p\,|q) $  read as follows:   %
 \vskip-19pt
  $$  \displaylines{
   \rosp(p\,|q)  =  \rosp(2 m + 1 | 2  n)  \, = \,  \left\{\!
     \begin{pmatrix}
        A  &    B    &    u   &   X   &   X_1   \\
        C  &   -A^t  &    v   &   Y   &   Y_1   \\
     -v^t  &   -u^t  &    0   &  z^t  &  z_1^t  \\
    Y_1^t  &  X_1^t  &   z_1  &   D   &    E    \\
     -Y^t  &   -X^t  &   -z   &   F   &   -D
     \end{pmatrix}  \! :
        \begin{matrix}
           B = -B^T  \\
           C = -C^T  \\
           E = E^T  \\
           F = F^T
        \end{matrix}  \,\right\}^{\phantom{\Big|}}  \cr
   \rosp(p\,|q)  \, = \,  \rosp(2 \, m \,| 2\, n)  \; = \;  \left\{
     \begin{pmatrix}
        A  &    B    &  X  &  X_1  \\
        C  &   -A^t  &  Y  &  Y_1  \\
    Y_1^t  &  X_1^t  &  D  &   E  \\
     -Y^t  &   -X^t  &  F  &  -D
     \end{pmatrix}  \; : \;
        \begin{matrix}
           B = -B^T  \\
           C = -C^T  \\
           E = E^T  \\
           F = F^T
        \end{matrix}  \;\right\}^{\phantom{\Big|}}  }  $$
   Moreover, if  $ \, m, n \geq 2 \, $,  then we have   --- with notation like in  {\it (1)}  ---   that
  $$  \displaylines{
   {\rosp(2m\!+\!1|2n)}_0 = \, \frak{o}(2m\!+\!1) \oplus \frak{sp}(2n) \; ,
\quad  {\rosp(2m|2n)}_0 = \, \frak{o}(2m) \oplus \frak{sp}(2n)  \cr
%
%
   {\rosp(p\,|2n)}_1 \, = \; f_p \otimes f_{2n}  \quad \forall \; p > 2 \; ,  \;\;\qquad  {\rosp(2|2n)}_1 \, = \, f_{2n}^{\,\oplus 2}  }  $$
\end{examples}

\bigskip

\begin{definition}
   Define the following Lie superalgebras:
 \vskip4pt
\noindent
 \; {\it (1)} \;   $ A(m,n) := \rsl(m \! + \! 1 |\, n \! + \! 1) \, $,  \,
$ A(n,n) := \rsl(n \! + \! 1 |\, n \! + \! 1) \Big/ {\bk I_{2n}} \; $,   \,\;  $ \forall \; m \! \neq \! n \, $;
 \vskip4pt
\noindent
 \; {\it (2)} \hfill   $ B(m,n) := \rosp(2m+1|\,2n) \; $,   \hfill   $ \forall \;\;\; m \geq 0 \, $,  $ \, n \geq 1 \; $; \quad {\ }
  \vskip4pt
\noindent
 \; {\it (3)} \hfill   $ C(n) := \rosp(2|\,2n-2) \, $,   \hfill   for all  $ \, n \geq 2 \, $; \quad {\ }
 \vskip4pt
\noindent
 \; {\it (4)} \hfill   $ D(m,n) := \rosp(2m|\,2n) \; $,   \hfill   for all  $\, m \geq 2 \, $,  $ \, n \geq 1 \, $; \quad {\ }
 \vskip4pt
\noindent
 \; {\it (5)} \hfill   $ \quad \displaystyle{
 P(n) := \, \left\{ \begin{pmatrix}
                         A  &  B  \\
                         C  & -A^t
\end{pmatrix} \in \rgl(n\!+\!1|\,n\!+\!1) \;\bigg|\;
  \begin{matrix}  A \in \rsl(n\!+\!1)  \\
           B^t = B \, , \; C^t = -C
  \end{matrix}   \;\right\} } $   \hfill {\ }
 \vskip4pt
\noindent
 \; {\it (6)} \hfill   $ \displaystyle{
 Q(n) := \, \bigg\{\! \begin{pmatrix}
                         A  &  B  \\
                         B  &  A
\end{pmatrix} \in \rgl(n\!+\!1|\,n\!+\!1) \;\bigg|\;\, B \in \rsl(n\!+\!1) \;\bigg\} \Bigg/ \bk I_{2(n+1)} } $  \hfill {\ }
\end{definition}

\bigskip

   The importance of these examples lies in the following (cf.~\cite{ka}, \cite{sc}):

\bigskip

\begin{theorem} \label{classificationtheorem}
   Let  $ \bk $  be an algebraically closed field of characteristic zero.  Then the classical Lie superalgebras over\/  $ \bk $  are either isomorphic to a simple Lie algebra or to one of the following classical Lie superalgebras:
  $$  \displaylines{
   A(m,n) \, ,  \;\; m \! \geq \! n \! \geq \! 0 \, , \, m+n > 0 \, ;  \quad  B(m,n) \, ,  \,\; m \geq 0, n \geq 1 \, ;   \quad  C(n) \, ,  \,\; n \geq 3  \phantom{{}_\big|}  \cr
   D(m,n) \, ,  \;\; m \geq 2 , \, n \geq 1 \, ;  \qquad  P(n) \, ,  \;\; n \geq 2 \, ;   \quad  Q(n) \, ,  \;\; n \geq 2  \phantom{{}_\big|}  \cr
   F(4) \; ;  \qquad  G(3) \; ;  \qquad  D(2,1;a) \, ,  \;\; a \in \bk \setminus \{0, -1\}  }  $$
(for the definition of the third line items, and for a proof, we refer to \cite{ka}).
%
%
\end{theorem}

\medskip

\begin{remark}   \label{cycl-Lie-ssalg}
 Let  $ \bk $  be a commutative unital ring
%
%
 as at the beginning of the section,
and  $ \fg $  a Lie  $ \bk $--superalgebra.  A Lie supersubalgebra  $ \fk $  of  $ \fg $  is called  {\it cyclic\/}  if it is generated by a single element  $ \, x \in \fg \, $:  then we write  $ \, \fk = \langle x \rangle \, $.
                                                                \par
   In contrast to the classical case, one has  {\sl not\/}  a priori  $ \, \langle x \rangle = \bk.x \, $,
 because one may have
$ \, [x,x] \not= 0 \, $.  For  {\sl homogeneous\/}  $ \, x \in \fg \, $,  three cases may occur:
  $$  \displaylines{
   \hfill   x \in \fg_0  \;\;\; \Longrightarrow \;\;\;  [x,x] = 0  \;\;\; \Longrightarrow \;\;\;  \langle x \rangle =  \bk.x   \hfill (2.1)  \cr
   \hfill   x \in \fg_1 \; ,  \,\; [x,x] = 0  \;\;\; \Longrightarrow \;\;\;  \langle x \rangle =  \bk.x  \phantom{\Big|}   \hfill (2.2)  \cr
   \hfill   x \in \fg_1 \; ,  \,\; [x,x] \not= 0  \;\;\; \Longrightarrow \;\;\;  \langle x \rangle \, = \, \bk.x \oplus \bk.[x,x]  \hfill (2.3)  }  $$
In particular, the sum in (2.3) is direct because  $ \, [x,x] \in \fg_0 \, $,  and  $ \, \fg_0 \, \cap \, \fg_1 \! = \! \{0\} \, $.  Moreover, this sum exhausts the Lie supersubalgebra generated by  $ x $  because  $ \, \big[x,[x,x]\big] = 0 \, $,  by the (super) Jacobi identity.  The Lie superalgebra structure is trivial in the first two cases; in the third instead, setting  $ \, y := [x,x] \, $,  it is
  $$  |x| = 1 \, ,  \;\; |y| = 0 \, ,  \qquad  [x,x] = y \; ,  \quad  [y,y] = 0 \; ,  \quad  [x,y] = 0 = [y,x].  $$
\end{remark}

\bigskip
\bigskip
\bigskip

 \section{Homogeneous one-parameter supersubgroups}
\label{1-param-sgrp}

\medskip

   A one-parameter subgroup of a Lie group is the unique (connected) subgroup  $ K $  which corresponds, via Frobenius theorem, to a specific one-dimensional Lie subalgebra  $ \fk $  of the tangent Lie algebra  $ \fg $  of the given Lie group  $ G \, $.  To describe such  $ K $  one can use the exponential map, which gives  $ \; K = \exp(\fk) \; $:  thus,  $ \fk $ is generated by some non-zero vector  $ \, X \in \fk \, $,  which actually  {\sl spans\/}  $ \fk \, $,  and using  $ X $  and the scalars in  $ \bk $  one describes  $ K $  via the exponential map.  Finally, when  $ \fg $  is linearized and expressed by matrices, the exponential map is described by the usual formal series on matrices  $ \; \exp(X) := \sum_{n=0}^{+\infty} X^n \big/ n! \; $.
 \vskip3pt
   We shall now adapt this approach to the context of Lie supergroups.
 \vskip5pt
   Let  $ G $  be a Lie supergroup over $ \bk \, $ (as usual, for supermanifolds we take $ \, \bk = \R \, $  or  $ \, \bk = \C \, $),  and  $ \, \fg = \Lie\,(G) \, $  its Lie superalgebra (for the construction of the latter, see for example  \cite{vsv},  Ch.~6).  Assume furtherly that  $ G $  is embedded as a supergroup into  $ \rGL(V) $  for some suitable supervector space  $ V \, $;  in other words,  $ G $  is realized as a matrix Lie supergroup.  Consequently its Lie superalgebra $ \fg $  is embedded into  $ \rgl(V) \, $.  As customary in supermanifold theory, we denote with  $ G $  also
the functor of points of the Lie supergroup  $ G \, $.  In the differential setting,  $ \, G : \smflds \lra \grps \, $,  $ \, G(T) = \, \Hom(T,G) \, $,  \, where  $ \smflds $  denotes the category of supermanifolds.

\medskip

   Recall that   --- see Definition \ref{cycl-Lie-ssalg}  ---   in the super context the role of one-dimensional Lie subalgebras is played by cyclic Lie subalgebras.

\medskip

\begin{definition}
   Let  $ \, X \in \fg = \Lie\,(G) \, $  be a homogeneous element.  We define  {\it one-parameter subgroup\/}  associated to  $ X $  the Lie subgroup of  $ G $  corresponding to the cyclic Lie supersubalgebra  $ \langle X \rangle $,  generated by  $ X $  in  $ \fg \, $,  via the Frobenius theorem for Lie supergroups (see  \cite{vsv},  Ch.~4,  and  \cite{ccfd}, Ch.~4).
\end{definition}

\medskip

   Now we describe these one-parameter subgroups.  Fix a supermanifold  $ T \, $,  and set  $ \, A := \cO(T) \, $  (the superalgebra of global sections).  Let  $ \, t \in A_0 \, $,  $ \, \theta \in A_1 \, $,  and  $ \, X \in \fg_0 \, $,  $ \, Y \in \fg_1 \, $,  $ \, Z \in \fg_0 \, $  such that  $ \, [Y,Z] = 0 \, $.  We define
  $$  \displaylines{
   \hskip14pt   \exp \big( t \, X\big)  \, := \,  {\textstyle \sum\limits_{n=0}^{+\infty}} \, t^n X^n \big/ n!  \; = \;  1 + t \, X + {\frac{\,t^2}{2!}} \, X^2 + \cdots  \quad\!\!\! \in {\rGL\big(V(T)\big)}   \hfill (2.5)  \cr
   \hskip13pt   \exp \big( \vartheta \, Y \big)  \; :=  \phantom{\Big|}
1 \, + \, \vartheta \, Y  \quad\!\! \in \, \rGL\big(V(T)\big)   \hfill (2.6)  \cr
   \hskip13pt   \exp \big( t \, Z + \vartheta \, Y \big)  \; :=  \;  \exp \big( t \, Z \big) \cdot \exp \big( \vartheta \, Y \big)  \; = \;  \exp \big( \vartheta \, Y \big) \cdot \exp \big( t \, Z \big)  \; =   \hfill  \cr
   \hskip27pt   =  \;  \exp \big( t \, Z \big) \cdot \big( 1 + \vartheta \, Y \big)  \; = \;  \big( 1 + \vartheta \, Y \big) \cdot \exp \big( t \, Z \big)  \quad\!\!\! \in \, \rGL\big(V(T)\big)   \hfill (2.7)  }  $$
All these expressions single out well-defined elements in  $
\rGL\big(V(T)\big) \, $.  In particular,  $ \exp\big(t\,X\big) $
in (2.5) belongs to the subgroup of  $ \rGL\big(V(T)\big) $  whose
elements are all the block matrices whose off-diagonal blocks are
zero.  This is the standard group of matrices  $ \, \rGL \big (A_0
\otimes V_0 \big) \times \rGL \big( A_1 \otimes V_1 \big) \, $,
\, and  $ \exp\big(t\,X\big) $  is defined inside here as the
usual exponential of a matrix.
                                                        \par
   More in general, one can define the matrix exponential as a natural transformation between the functors of points of the Lie superalgebra  $ \fg $  and of the Lie supergroup  $ G \, $;  see also  \cite{be},  Part II, Ch.~2, for yet another approach.  Our interest lies in the algebraic category, so we do not pursue this point of view.
 \vskip4pt
   Note that the set  $ \; \exp\big(A_0\,X\big) = \big\{ \exp\big(t\,X\big) \,\big|\, t \in A_0 \big\} \; $  is clearly a subgroup of  $ G \, $, once we define, very naturally, the multiplication as
  $$  \exp\big(t\,X\big) \cdot \exp\big(s\,X\big)  \; = \;  \exp\big((t+s)\,X\big)  $$
   On the other hand, if we consider the same definition for
  $$  \exp\big(A_1\,Y\big)  \, := \,  \big\{ \exp\big(\vartheta\,Y\big) \,\big|\, \vartheta \! \in \! A_1 \big\} $$
we see it is not a subgroup: in fact,
  $$  \exp\big(\vartheta_1\,Y\big) \cdot \exp\big(\vartheta_2\,Y\big)  \, = \,  \big( 1 + \vartheta_1 \, Y \big) \, \big( 1 + \vartheta_2 \, Y \big)  \, = \,  1 + \vartheta_1 \, Y \! + \vartheta_2 \, Y \! + \vartheta_1 \vartheta_2 \, Y^2  $$

formally in the universal enveloping algebra, while on the other hand:
  $$  \exp\big((\vartheta_1+\vartheta_2)\,Y\big)  \, = \,  1 + \big( \vartheta_1 + \vartheta_2 \big) \, Y  \, = \,  1 + \vartheta_1 \, Y + \vartheta_2 \, Y  $$
So, recalling that  $ \, Y^2 = [Y,Y] \big/ 2 \; $,  \, we see that
$ \exp\big(A_1\,Y\big) $  is a subgroup if and only if  $ \, [Y,Y]
= 0 \, $  or  $ \, \vartheta_1 \, \vartheta_2 = 0 \, $  for all  $
\, \vartheta_1, \vartheta_2 \in A_1 \, $.  This reflects the fact
that the  $ \bk $--span  of  $ \, X \in \fg_0 \, $  is always a
Lie supersubalgebra of  $ \fg \, $,  but the  $ \bk $--span  of  $
Y \in \fg_1$  is a Lie supersubalgebra iff  $ \; [Y,Y] = 0 \, $,
\, by (2.1--3).
                                                          \par
   Thus, taking into account (2.3), when  $ \, [Y,Y] \not= 0 \, $
we must consider  $ \, \exp \big( \langle Y \rangle (T) \big) =
\exp \big( A_1 \, Y + A_0 \, Y^2 \big) \, $,  as the one-parameter
subgroup corresponding to the  Lie supersubalgebra  $ \langle Y
\rangle \, $.  The outcome is the following:

\medskip

\begin{proposition}  \label{class_1-param-sgrps}
   \hskip-1pt   There are three distinct types of one-parameter subgroups
associated to an homogeneous element in  $ \fg \, $. Their functor
of points are:
  $$  \displaylines{
   \text{\it (a)} \quad  \text{for any \ }  X \in \fg_0 \; ,  \;
\text{we have}   \hfill  \cr
   \quad   x_X(T) \, = \, \big\{ \exp(tX) \,\big|\, t \in {\cO_T(T)}_0 \,\big\} \, = \; \bk^{1|0}(T) \; = \;
 \Hom\big(C^\infty(\R),\cO_T(T)\big)   \hfill  \cr
   \text{\it (b)} \quad  \text{for any \ }  Y \in \fg_1 \; ,  \,\; [Y,Y] = 0 \, ,  \; \text{we have}   \phantom{\Big|}   \hfill  \cr
   \quad   x_Y(T) \, = \, \big\{ \exp(\vartheta\,Y) = 1 + \vartheta \, Y \,\big|\, \vartheta \in {\cO_T(T)}_1 \,\big\} \, =   \hfill  \cr
   \hfill   = \; \bk^{0|1}(T) \; = \; \Hom\big(\bk[\xi],\cO_T(T)\big)  \cr
   \text{\it (c)} \quad  \text{for any \ }  Y \in \fg_1 \; ,  \,\; Y^2 := [Y,Y] \big/ 2 \not= 0 \, ,  \; \text{we have}   \phantom{\Big|}   \hfill  \cr
   \quad   x_Y(T) \, = \, \big\{ \exp\big( t\,Y^2 + \vartheta\,Y \big) \,\big|\, t \in {\cO_T(T)}_0 \, , \, \vartheta \in {\cO_T(T)}_1 \,\big\} \, =   \hfill  \cr
   \hfill   = \; \bk^{1|1}(T) \; = \; \Hom\big(C^\infty(\R)[\xi],\cO_T(T)\big)  }  $$
where  $ C^\infty(\R) $  denotes the global sections of the differential functions on  $ \R \, $   --- if  $ \, \Bbbk = \R \, $;  if  $ \, \Bbbk = \C \, $  instead we shall similarly take analytic functions.

   \indent   In cases (a) and (b) the multiplication structure is obvious, and in case (c) it is given by  $ \;\; (t,\vartheta) \cdot \big(t',\vartheta'\big)  \, = \,
\big(\, t + t' - \vartheta \, \vartheta' \, , \, \vartheta +
\vartheta' \,\big) \, $.
\end{proposition}

\begin{proof}
  The case  {\it (a)},  namely when  $ X $  is even, is clear.  When instead  $ X $  is odd we have two possibilities: either  $ \, [X,X] = 0 \, $  or  $ \, [X,X] \neq 0 \, $.  The first possibility corresponds, by Frobenius theorem, to a  $ 0|1 $--dimensional  subgroup, whose functor of points, one sees immediately, is representable and of the form  {\it (b)}.  Let us now examine the second possibility.  The Lie subalgebra  $ \langle X \rangle $  generated by  $ X $  is of dimension  $ 1|1 $,  by (2.3); thus by Frobenius theorem it corresponds to a Lie subgroup of the same dimension, isomorphic to  $ \bk^{1|1} $.
                                              \par
   Now we compute the group structure on this  $ \bk^{1|1} \, $,  using the usual functor of points notation to give the operation of the supergroup.  For any commutative superalgebra  $ A \, $,  we have to calculate  $ \, t'' \in A_0 \, $,  $ \, \vartheta'' \in A_1 \, $ such that
  $$  \exp \big( t \, X^2 + \vartheta \, X \big) \cdot \exp \big( t' \, X^2 + \vartheta' \, X \big) \, = \, \exp \big( t'' \, X^2 + \vartheta'' \, X \big)  $$
where  $ \, t, t' \in A_0 \, $,  $ \, \vartheta, \vartheta' \in
A_1 \, $.  The direct calculation gives
  $$  \displaylines{
   \exp\big( t \, X^2 + \vartheta \, X \big) \cdot \exp \big( t' \, X^2 + \vartheta' \, X \big) \; =   \hfill  \cr
   = \; \big( 1 + \vartheta \, X \big) \, \exp\big( t \, X^2 \big) \cdot \exp \big( t' \, X^2 \big) \, \big( 1 + \vartheta' \, X \big) \; =  \cr
   \hfill   = \; \big( 1 + \vartheta \, X \big) \, \exp\big( (t+t') \, X^2 \big) \, \big( 1 + \vartheta' \, X \big) \; =  \cr
   = \; \exp\big( (t+t') \, X^2 \big) \, \big( 1 + \vartheta \, X \big) \, \big( 1 + \vartheta' \, X \big) \; =   \hfill  \cr
   = \; \exp\big( (t+t') \, X^2 \big) \, \big( 1 + (\vartheta + \vartheta') \, X - \vartheta \, \vartheta' \, X^2 \big) \; =  \cr
   \hfill   = \; \exp\big( (t+t') \, X^2 \big) \, \big( 1 - \vartheta \, \vartheta' \, X^2 \big) \, \big( 1 + (\vartheta + \vartheta') \, X \big) \; =  \cr
   = \; \exp\big( (t+t') \, X^2 \big) \, \exp\big(\!-\vartheta \, \vartheta' \, X^2 \big) \, \big( 1 + (\vartheta + \vartheta') \, X \big) \; =   \hfill  \cr
   = \; \exp\big( (t + t' - \vartheta \, \vartheta') \, X^2 \big) \, \big( 1 + (\vartheta + \vartheta') \, X \big) \; =  \cr
   \hfill   = \; \exp\big( (t + t' - \vartheta \, \vartheta') \, X^2 + (\vartheta + \vartheta') \, X \big)  }  $$
where we use several special properties of the formal exponential.
\end{proof}

\medskip

\begin{remark}
   The supergroup structure on  $ \bk^{1|1} $  in  Proposition \ref{class_1-param-sgrps}{\it (c)\/}  was introduced in \cite{dm}.  See also  \cite{ro},  \S 6.5 and \S 9.6 for a treatment of one-parameter (super)subgroups in a differential setting (with the same outcome as ours).
\end{remark}

%

  %
  %

\chapter{Chevalley bases and Chevalley algebras}  \label{che-bas_alg}

 {\it
   In this chapter we introduce one of the main ingredients of our
construction, namely the notion of ``Chevalley basis''. We shall formulate
such concept in a way that directly extends the classical approach due
to Chevalley.  First, we recall the combinatorial machinery attached to any
classical Lie superalgebra  $ \fg \, $,  namely:  Cartan subalgebras,
root systems, root vectors, root spaces and so on.  Then we define,  a
``Chevalley basis'' for  $ \fg $  as a basis of  $ \, \fg $
made of Cartan elements and root vectors, satisfying certain integral
conditions on their brackets.

\smallskip

   The main concern of this chapter is to prove an existence result
for Chevalley bases of classical Lie superalgebras.  We provide a
complete proof, going through a case-by-case analysis and we give
an explicit example of a Chevalley basis for all classical Lie superalgebras,
but for the exceptional ones for which we refer to  \cite{ik}.
In addition, we present a sketch of an alternative approach,
which expands an argument kindly communicated to us
by our referee, who suggested a uniform proof encompassing
all cases, but  $ P(n) $,  at once,
so that our explicit treatment still remains necessary for that
particular case.

\smallskip

   Both the definition and the existence problem for Chevalley bases were
somehow known in the community of Lie superalgebra theorists
(see, e.g., \cite{ik},  \cite{sw}), since they are a very natural
generalization of the corresponding ordinary setting.
Nevertheless, the results achieved so far   --- to the best of our knowledge
---   did not cover  {\sl all\/}  of classical Lie superalgebras,
and so  our results appear to be more general
that then ones available in the literature.
 }

\bigskip

   Let our ground ring be an algebraically closed field  $ \KK $  of
characteristic zero.

\smallskip

   We assume  $ \fg $  to be a  {\sl classical\/}  Lie superalgebra:
the whole construction will clearly extend to any direct sum of finitely many
summands of this type.  We now prove that  $ \fg $  has a very remarkable
basis, the analogue of what Chevalley found for finite dimensional
(semi)simple Lie algebras.

\bigskip
\bigskip
\bigskip

  \section{Root systems}  \label{root-syst}

\medskip

   Fix once and for all a  {\sl Cartan subalgebra\/}  $ \fh $  of  $ \fg_0 \, $.  The adjoint action of  $ \fh $  splits  $ \fg $  into eigenspaces, namely
  $$  \fg_\alpha  \; := \;  \big\{\, x \in \fg \;\big|\; [h,x] = \alpha(h) x \, , \; \forall \; h \! \in \! \fh \big\}   \eqno  \forall \;\; \alpha \in \fh^*   \qquad  $$
so that  $ \; \fg  \, = \, \bigoplus_{\alpha \in \fh^*} \fg_\alpha \; $.  Then we define
  $$  \displaylines{
   \Delta_0  \, := \,  \big\{\, \alpha \in \fh^* \setminus \{0\} \;\big|\; \fg_\alpha \cap \fg_0 \not= \{0\} \big\}  \, = \,  \{\,\text{{\sl even roots\/}  of  $ \fg $} \,\}  \cr
   \Delta_1  \, := \,  \big\{\, \alpha \in \fh^* \;\big|\; \fg_\alpha \cap \fg_1 \not= \{0\} \big\}  \, = \,  \{\,\text{{\sl odd roots\/}  of  $ \fg $} \,\}  \cr
   \Delta  \, := \,  \Delta_0 \, \cup \Delta_1  \, = \,  \{\,\text{{\sl roots\/}  of  $ \fg $} \,\}  }  $$
$ \Delta $  is called the  {\sl root system\/}  of  $ \fg \, $,  and each  $ \fg_\alpha $  is called a  {\sl root space}.
                                                       \par
   In particular,  $ \Delta_0 $  is the root system of the (reductive) Lie algebra  $ \fg_0 \, $,  and  $ \Delta_1 $  is the set of weights of the representation of  $ \fg_0 $  in  $ \fg_1 \, $.

\medskip

   If  $ \fg $  is not of type  $ P(n) $  nor  $ Q(n) $,  there is an even non-degenerate, invariant bilinear form on  $ \fg \, $,  whose restriction to  $ \fh $  is in turn an invariant bilinear form on  $ \fh \, $.  If instead  $ \fg $  is of type  $ P(n) $  or  $ Q(n) $,  then such a form on  $ \fh $  exists because  $ \fg_0 $  is simple (of type  $ A_n $).  In any case, we denote this form by $ \big( x, y \big) \, $,  and we use it to identify  $ \fh^* $  to  $ \fh $,  via  $ \, \alpha \mapsto H_\alpha \, $,  and then to define a similar form on  $ \fh^* $,  such that  $ \; \big( \alpha' , \alpha'' \big) = \big( H_{\alpha'}, H_{\alpha''} \big) \; $.  Each  $ H_\alpha $  is called the {\sl coroot\/}  associated to  $ \alpha \, $;  these coroots can be explicitly described as in  \cite{ik},  \S 2.5; in particular, one has  $ \, \alpha(H_\alpha) = 2 \, $  whenever  $ \, (\alpha,\alpha) \not= 0 \, $.
                                                       \par
   For  $ \fg $  of type  $ P(n) \, $  we shall adopt the following abuse of notation.  For any even root  $ \alpha_{i,j} $  (notation of  \cite{fss},  \S 2.48), by  $ H_{\alpha_{i,j}} $  we shall mean the coroot mentioned above; for any odd root  $ \beta_{i,j} $  instead, we shall set  $ \, H_{\beta_{i,j}} := H_{\alpha_{i,j}} \, $.
%
%

\medskip

   The main properties of the root system of  $ \fg $  are collected in the following:

\medskip

\begin{proposition}  \label{prop-root-syst}
   (see \cite{ka,sc,se})  Assume  $ \fg $  is  {\sl classical},  and  $ \, n \in \N \, $.
 \vskip9pt
   \textit{(a)} \;\;  $ \fg \not= Q(n) \; \Longrightarrow \; \Delta_0 \cap \Delta_1 = \emptyset \; ,  \qquad  \fg = Q(n) \; \Longrightarrow \; \Delta_1 = \Delta_0 \cup \{0\} \; $.
 \vskip8pt
   \textit{(b)} \;  $ -\Delta_0 = \Delta_0 \; ,  \;\; -\Delta_1 \subseteq \Delta_1 \; $.  \quad  If  $ \, \fg \not= P (n) \; $,  then  $ \; -\Delta_1 = \Delta_1 \; $.
 \vskip8pt
   \textit{(c)} \;  Let  $ \, \fg \not= P(2) \, $,  and  $ \, \alpha $,  $ \beta \in \Delta \, $,  $ \; \alpha = c \, \beta \, $,  \, with  $ \, c \in \KK \setminus \{0\} \, $.  Then
 \vskip-21pt
  $$  \alpha , \beta \in \Delta_r \;\; (r=0,1)  \, \Rightarrow \,  c = \pm 1 \;\; ,  \,\quad\,
      \alpha \in \Delta_r \, , \, \beta \in \Delta_s \, ,  \; r \not= s  \, \Rightarrow \,  c = \pm 2 \;\; .  $$
 \vskip1pt
   \textit{(d)} \;  If  $ \; \fg \not\in \big\{ A(1,1) \, , P(3) \, , Q(n) \big\} \, $,  then  $ \; \dim_\KK(\fg_\alpha) = 1 \; $  for each  $ \, \alpha \in \Delta \, $.
                                                       \par
   As for the remaining cases, one has:

\begin{itemize}
   \item
  If  $ \; \fg = A(1,1) \, $,  then  $ \; \dim_\KK(\fg_\alpha) = 1 \! + \! r \; $  for each  $ \, \alpha \in \Delta_r \, $;
   \item
  If  $ \; \fg = P(3) \, $,  then  $ \; \dim_\KK(\fg_\alpha) = 1 \, $  for  $ \, \alpha \in \Delta_0 \cup \big( \Delta_1 \setminus (-\Delta_1) \big) \, $,  \, and  $ \; \dim_\KK(\fg_\alpha) = 2 \, $  for  $ \, \alpha \in \Delta_1 \cap (-\Delta_1) \, $;
   \item
  If  $ \; \fg = Q(n) \, $,  then  $ \; \dim_\KK\big( \fg_\alpha \cap \fg_i \big) = 2 \; $  for  $ \, \alpha \in \Delta \setminus \{0\} \, $,  $ \, i \in \{0,1\} \, $,  and  $ \; \dim_\KK\big( \fg_{\alpha=0} \cap \fg_i \big) = n \; $  for  $ \, i \in \{0,1\} \, $,  with  $ \; \fg_{\alpha=0} \cap \fg_0 = \fh \, $.
\end{itemize}
\end{proposition}

\bigskip

   We fix a  {\sl distinguished simple root system\/}  for  $ \fg \, $,  say  $ \, \Pi = \{ \alpha_1, \dots, \alpha_\ell \} \, $,  as follows.  If  $ \, \fg \not\in \big\{ A(1,1), P(3), Q(n) \big\} \, $,  we take as  $ \Pi $  a subset  of  $ \Delta $  in which  $ \, \alpha_i \in \Delta_0 \, $  for all but one index  $ i \, $,  and such that any  $ \, \alpha \in \Delta \, $  is either a sum of some  $ \alpha_i $'s   --- then it is called positive ---   or the opposite of such a sum   --- then it is called negative (when  $ \fg $  is  {\sl basic},  fixing  $ \Pi $  is equivalent to fix a special triangular decomposition  $ \, \fg = \mathfrak{n}_- \oplus \fh \oplus \mathfrak{n}_+ \, $;  see  \cite{fss}, \S 2.45).  If  $ \, \fg = Q(n) \, $,  then  $ \, \Delta = \Delta_0 \cup \{0\} \, $,  we take as  $ \Pi $  any simple root system of  $ \Delta_0 \, $,  and we define positive and negative roots accordingly, letting the special odd root  $ \, \zeta_0 = 0 \, $  be (by definition) both positive and negative.  Finally, if  $ \, \fg \in \big\{ A(1,1), P(3) \big\} \, $,  then there exist linear dependence relations among the  $ \alpha_i $'s,  so that any odd root which is negative can also be seen as positive; we shall then deal with these cases in a different way.
                                                                   \par
   As for notation, we denote by  $ \Delta^+ $,  resp.~$ \Delta^- $,  the set of positive, resp.~negative, roots; also, we set  $ \; \Delta^\pm_r := \Delta^\pm \cap \Delta_r \; $  for  $ \, r \in \{0,1\} \, $.  Finally, we define:

\bigskip

\begin{definition}  \label{alpha-string}
   Given  $ \, \alpha, \beta \in \Delta \, $,  we call  {\sl  $ \, \alpha $--string  through  $ \beta \, $}  the set
  $$  \Sigma^\alpha_\beta  \; := \;  \big\{\, \beta - r \, \alpha \, , \dots , \, \beta - \alpha \, , \, \beta \, , \, \beta + \alpha \, , \dots , \, \beta + q \, \alpha \,\big\}  \quad\; \big( \subset \fh^* \big)  $$
with  $ \, r, q \in \N \, $  uniquely determined by the conditions
  $ \, \big( \beta - (r+1) \, \alpha \big) \not\in \Delta \, $,  $ \; \big( \beta + (q+1) \, \alpha \big) \not\in \Delta \, $,  \, and  $ \; (\beta + j \, \alpha) \in \Delta \cup \{0\} \; $  for all  $ \, -r \leq j \leq q \, $.
\end{definition}

\medskip

   One knows   --- cf.~for instance  \cite{se}  ---   that  $ \; \Sigma^\alpha_\beta \subseteq \Delta \cup \{0\} \; $.  Indeed, one has that  $ \; 0 \in \Sigma^\alpha_\beta \; $  if and only if  $ \; \alpha \in \big\{\! \pm 2 \, \beta \, , \, \pm \beta \, , \, \pm \beta / 2 \,\big\} \; $.

\bigskip
\bigskip
\bigskip

  \section{Chevalley bases and algebras}  \label{che_bas-alg}

\medskip

   The key result of this subsection is an analogue, in the super setting, of a classical result due to Chevalley.  It is the starting point we shall build upon later on.  We keep the notation and terminology of the previous subsections.

\smallskip

  {\sl From now on, we shall consider  $ \fg $  to be classical but  {\it not}  of any of the following types:  $ A(1,1) \, $,  $ P(3) \, $,  $ Q(n) $  or  $ D(2,1;a) $  with  $ \, a \not\in \Z \, $.  The cases  A--P--Q  will be treated separately in  \S \ref{cases A-P-Q};  the case  $ D(2,1;a) $  with  $ \, a \not\in \Z \, $  instead is dealt with in  \cite{ga}.}

\medskip

\begin{definition}  \label{def_che-bas}
 Let  $ \fg $  be a  {\sl classical\/}  Lie superalgebra  as above.  We call  {\it Chevalley basis\/}  of  $ \fg $  any homogeneous  $ \KK $--basis  $ \; B = {\big\{ H_i \big\}}_{1,\dots,\ell} \coprod {\big\{ X_\alpha \big\}}_{\alpha \in \Delta} \; $  such that:
 \vskip9pt
   \textit{(a)}  \quad   $ \, \big\{ H_1 , \dots , H_\ell \big\} \, $  is a  $ \KK $--basis  of  $ \fh \, $;  moreover, with  $ \, H_\alpha \! \in \fh \, $  as in  \S \ref{root-syst},

 \vskip4pt
   \centerline{ \qquad\qquad  $ \, \fh_\Z  \, := \,  \text{\it Span}_{\,\Z} \big( H_1 , \dots , H_\ell \big)  \, = \,  \text{\it Span}_{\,\Z} \big( \big\{ H_\alpha \,\big|\, \alpha \! \in \! \Delta \! \cap \! (-\Delta) \big\}\big) \; $; }
%
%
 \vskip8pt
   \textit{(b)}  \hskip4pt   $ \big[ H_i \, , H_j \big] = 0 \, ,   \hskip9pt
 \big[ H_i \, , X_\alpha \big] = \, \alpha(H_i) \, X_\alpha \, ,   \hskip15pt  \forall \; i, j \! \in \! \{ 1, \dots, \ell \,\} \, ,  \; \alpha \! \in \! \Delta \; $;
 \vskip11pt
   \textit{(c)}  \hskip7pt   $ \big[ X_\alpha \, , \, X_{-\alpha} \big]  \, = \,  \sigma_\alpha \, H_\alpha  \hskip25pt  \forall \;\; \alpha \in \Delta \cap (-\Delta) $
 \vskip4pt
\noindent
 with  $ H_\alpha $  as in  \S \ref{root-syst},  and  $ \; \sigma_\alpha := -1 \; $  if  $ \, \alpha \in \Delta_1^- \, $,  $ \; \sigma_\alpha := 1 \; $  otherwise;
 \vskip13pt
   \textit{(d)}  \quad  $ \, \big[ X_\alpha \, , \, X_\beta \big]  \, = \, c_{\alpha,\beta} \; X_{\alpha + \beta}  \hskip17pt   \forall \;\, \alpha , \beta \in \Delta \, : \, \alpha \not= -\beta \, , \; \beta \not= -\alpha \, $,  \, with
 \vskip9pt
   \hskip11pt   \textit{(d.1)} \,  if  $ \;\; (\alpha + \beta) \not\in \Delta \, $,  \; then  $ \;\; c_{\alpha,\beta} = 0 \; $,  \; and  $ \; X_{\alpha + \beta} := 0 \; $,
 \vskip7pt
   \hskip11pt   \textit{(d.2)} \,  if  $ \, \big( \alpha, \alpha \big) \not= 0 \, $  or  $ \, \big( \beta, \beta \big) \not= 0 \, $,  and (cf.~Definition \ref{alpha-string})  if  $ \; \Sigma^\alpha_\beta := \big\{ \beta - r \, \alpha \, , \, \dots \, , \, \beta + q \, \alpha \big\} \; $  is the  $ \alpha $--string  through  $ \beta \, $,  \, then  $ \; c_{\alpha,\beta} = \pm (r+1) \, $,  \, with the following exceptions: if  $ \, \fg = P(n) \, $,  with  $ \, n \not= 3 \, $,  and  $ \; \alpha = \beta_{j,i} \, $,  $ \, \beta = \alpha_{i,j} \; $  (notation of  \cite{fss},  \S 2.48, for the roots of  $ P(n) $),  then  $ \; c_{\alpha,\beta} = \pm (r+2) \; $;
 \vskip5pt
   \hskip11pt   \textit{(d.3)} \,  if  $ \, \big( \alpha \, , \alpha \big) = \, 0 = \big( \beta \, , \beta \big) $,  \, then  $ \; c_{\alpha,\beta} = \pm \beta\big(H_\alpha\big) \; $.
 \vskip10pt
   {\sl N.B.:}  this definition clearly extends to direct sums of finitely many  $ \fg $'s.
\end{definition}

\medskip

\begin{remarks}   \label{rem-Chev_1}  {\ }
 \vskip5pt
   {\it (1)} \,  Our definition extends to the super setup the same name notion for (semi)simple Lie algebras.  For type  $ A $  it was
 essentially known as ``folklore'',\break
\noindent
 but we cannot provide any reference.  In the orthosymplectic case (types  $ B \, $,  $ C $  and  $ D \, $)  it was considered, in weaker form, in  \cite{sw}.  More in general, it was previously introduced in  \cite{ik}  for all  {\sl basic\/}  types, i.e.~missing types  $ P $  and  $ Q \, $.
                                                                     \par
   {\sl N.B.:\/}  when reading  \cite{ik}  for  $ G(3) $  one should do a slight change, namely use the Cartan matrix   --- and Dynkin diagram, etc.~---   as in Kac's paper  \cite{ka}.
%
%
 \vskip4pt
   {\it (2)} \,  If  $ B $  is a Chevalley basis of  $ \fg \, $,  the definition implies that all structure coefficients of the (super)bracket in  $ \fg $  w.r.t.~$ B $  belong to  $ \Z \, $.
\end{remarks}

\medskip

\begin{definition}
   If  $ B $  is a Chevalley basis of  $ \fg \, $,  we set  $ \; \fg^\Z := \Z\text{\sl --span of}~B \, $,  and we call it the  {\it Chevalley superalgebra\/}  (of  $ \fg $).
\end{definition}

\medskip

\begin{remarks}   \label{rem-Chev_2}  {\ }
 \vskip5pt
   {\it (1)} \,  By  Remark \ref{rem-Chev_1}{\it (2)},  $ \fg^\Z $  is a Lie superalgebra over  $ \Z \, $.  One can check that a Chevalley basis  $ B $  is unique up to a choice of a sign for each root vector and the choice of the  $ H_i $'s:  thus  $ \fg^\Z $  is independent of the choice of  $ B \, $.
 \vskip4pt
%
%
%
   {\it (2)} \,  With notation as in  Definition \ref{def_che-bas}{\it(e)},  let  $ \fg $  be of type $ A $,  $ B $,  $ C $,  $ D $  or  $ P \, $.  Then
%
%
 if  $ \; \big( \alpha \, , \alpha \big) = 0 = \big( \beta \, , \beta \big) \; $  one has  $ \; \beta \big(H_\alpha\big) = \pm (r\!+\!1) \, $,  \, with the following exceptions: if  $ \, \fg = \rosp(M|2n) \, $  with  $ \, M \geq 1 \, $  (i.e.~$ \fg $  is orthosymplectic, not of type  $ B(0,n) \, $)  and   --- with notation of  \cite{fss},  \S 2.27 ---
  $$  (\alpha,\beta) = \pm \big( \varepsilon_i + \delta_j \, ,
\, -\varepsilon_i + \delta_j \big)  \qquad  \text{or}  \qquad  (\alpha,\beta) =
\pm \big( \varepsilon_i - \delta_j \, , \, -\varepsilon_i - \delta_j \big)  $$
then  $ \; \beta\big(H_\alpha\big) = \pm (r\!+\!2) \; $.  Therefore:
   {\sl If  $ \fg $  is of type $ A $,  $ B $,  $ C $,  $ D $  or  $ P \, $,  then condition  {\it (d.3)\/}  in  Definition \ref{def_che-bas}  reads just like  {\it (d.2)},  with the handful of exceptions mentioned before.}
 \vskip4pt
   {\it (3)} \,  For notational convenience, in the following we shall also set  $ \, X_\delta := 0 \, $  whenever  $ \delta $  belongs to the  $ \Z $--span  of  $ \Delta $  but  $ \, \delta \not\in \Delta \; $.
\end{remarks}

\bigskip
\bigskip
\bigskip

 \section{Existence of Chevalley bases}  \label{exist_Chev-bases}

\medskip

   The existence of a Chevalley basis for the types  $ A $,  $ B $,  $ C $,  $ D $  is a more or less known result; for example an (almost) explicit Chevalley basis for types  $ B \, $,  $ C $  and  $ D \, $ can be found in \cite{sw}.  More in general, an (abstract) existence result, with a uniform proof, is given in  \cite{ik}  for all  {\sl basic\/}  types   --- thus missing the  {\sl strange\/}  types,  $ P $  and  $ Q \, $.  In this section we present an existence theorem which covers  {\sl all\/}  cases, i.e.~including both basic and strange cases: our proof is constructive, in that we explicitly present a concrete Chevalley basis, for all cases but  $ F(4) $  and  $ G(3) $   --- for which we refer to  \cite{ik} ---   by a case-by-case analysis.  A sketch of a  {\sl uniform\/}  proof is presented in  Remark \ref{unif_exist_Chev-bases}  later on.

\medskip

 \vskip3pt

\begin{theorem}  \label{exist_che-bas}
   Every classical Lie superalgebra has a Chevalley basis.
\end{theorem}

\begin{proof}
 The proof is case-by-case, by direct inspection of each type.  Only cases  $ A(1,1) \, $,  $ P(3) \, $,  $ Q(n) $  are postponed to  \S \ref{cases A-P-Q}.

\smallskip

   In general, for the root vectors  $ X_\alpha $'s, we must carefully fix a proper normalization.  For the  $ H_i $'s  in the Chevalley basis, belonging to  $ \fh \, $,  one sees that, in the  {\sl basic\/}  cases, one can almost always take simple coroots (for a distinguished system of simple roots); case  $ P(n) $  is just slightly different.

\smallskip

   We call  $ \fg $  our classical Lie superalgebra.  As a matter of notation, from now on we denote by  $ \; \e_{i,j} \, := \, {\big( \delta_{h,i} \, \delta_{k,j} \big)}_{h,k=1,\dots,r+s} \; $  for all  $ \, i, j \in \{1,\dots,m+n\} \, $,  \, the elementary matrix in  $ \rgl(r|s) $  with a 1 in position  $ (i,j) $  and 0 elsewhere.
 \vskip4pt
   $ \underline{A(m,n)} \phantom{\bigg|}\, , \; m \not= n \; $:  \;  In this case  $ \, \fg = \rsl(m+1|\,n+1) \subseteq \rgl(m+1|\,n+1) \, $.  We fix the distinguished simple root system  $  \, \Pi := \{\alpha_1, \dots, \alpha_{m+n+1}\} \, $  with  $ \; \alpha_i := \varepsilon_i - \varepsilon_{i+1} \, \big(\, i=1,\dots,m \big) \, $,  $ \, \alpha_{m+1} := \varepsilon_{m+1} - \delta_1 \, $,  $ \, \alpha_{m+1+j} := \delta_j - \delta_{j+1} \, \big(\, j=1,\dots,n \big) \, $,  \, using standard notation (like in   \cite{fss},  say).  Then we define the  $ H_k $'s  as  $ \, H_k := H_{\alpha_k} = \e_{k,k} - \e_{k+1,k+1} \, $  for  $ \, k \not\in \{m\!+\!1,m\!+\!n\!+\!2\,\} \, $,  and  $ \, H_{m+1} := H_{\alpha_{m+1}} = \e_{m+1,m+1} + \e_{m+2,m+2} \; $.   For the root vectors  $ X_\alpha $'s  instead we just take all the  $ \e_{i,j} $'s  with  $ \, i \not= j \, $.  It is then a routine matter to check that these vectors form a Chevalley basis.
 \vskip3pt
   $ \underline{A(n,n)} \phantom{\bigg|} \; $:  \;  In this case  $ \, \fg = \rsl(n\!+\!1|\,n\!+\!1) \Big/ \KK \, I_{2(n+1)} \, $.  Keeping notation as above, we set  $ \; \overline{x} := x + \KK \, I_{2(n+1)} \; $  in  $ \; \rsl(n\!+\!1|\,n\!+\!1) \Big/ \KK \, I_{2(n+1)} \, = \fg \; $.  Then  $ \; I_{2(n+1)} = {\textstyle \sum_{i=1}^n} \, i \, H_i \, + \, (n\!+\!1) \, H_{n+1} \, - \, {\textstyle \sum_{j=1}^n} \, j \, H_{2(n+1)-j} \, $,  \, so that we have
  $$  {\textstyle \sum_{i=1}^n} \, i \, \overline{H}_i \, + \, (n\!+\!1) \, \overline{H}_{n+1} \, - \, {\textstyle \sum_{j=1}^n} \, j \, \overline{H}_{2(n+1)-j}  \, = \,  0  $$
a  $ \Z $--linear  dependence relation among simple coroots which reflects a similar relation among simple roots; thus we can get  $ \overline{H}_{2n+1} $  from the other simple coroots.  Then  $ \; {\big\{\, \overline{H}_i \big\}}_{i=1,\dots,2n} \bigcup \, {\big\{\, \overline{X}_{i,j} \big\}}_{i \not= j} \; $  is a  $ \KK $--basis  of  $ \fg $  which satisfies all properties in  Definition \ref{def_che-bas}.
 \vskip6pt
   $ \underline{B(m,n) \, , \, C(n) \, , \, D(m,n)} \phantom{\bigg|}\, , \; m \not= n \; $:  \;  Here  $ \, \fg = \rosp(M|\,N) \subseteq \rgl(M|\,N) \, $  for some  $ \,
M \in \N \, $,  which is odd, zero or even positive according to
whether we are in case  $ B $,  $ C $  or  $ D $  respectively,
and some even  $ \, N \in \N \, $.  Then  $ \, M \in \{2\,m + 1,
2\,m\} \,  $  and  $ \, N = 2\,n \, $  for suitable  $ \, m, n \in
\N \, $.  In any case,  $ \fg $  is an orthosymplectic Lie
superalgebra, and we can describe all cases at once.
                                                                          \par
   With notation as before, we consider the following root vectors (for all  $ \; 1 \leq i,j \leq m \, $,  $ \; M+1 \leq i', j' \leq  M+n \, $):
  $$  \displaylines{
   \hfill   E_{+\varepsilon_i - \varepsilon_j} \, := \, \e_{i,j} - \e_{j+m,i+m} \;\; ,  \hskip-5pt \qquad  E_{-\varepsilon_i + \varepsilon_j} \, := \, \e_{i+m,j+m} - \e_{j,i}   \qquad \hfill \big(\, i < j \,\big)  \cr
   \hfill   E_{+\varepsilon_i + \varepsilon_j} \, := \, \e_{i,j+m} - \e_{j,i+m} \;\; ,  \hskip-5pt \qquad  E_{-\varepsilon_i - \varepsilon_j} \, := \, \e_{j+m,i} - \e_{i+m,j}   \qquad \hfill \big(\, i < j \,\big)  \cr
   E_{+\varepsilon_i} := \, +\sqrt{2} \, \big( \e_{i,2m+1} - \e_{2m+1,i+m} \big) \;\; ,  \qquad  E_{-\varepsilon_i} := \, -\sqrt{2} \, \big( \e_{i+m,2m+1} - \e_{2m+1,i} \big)   \qquad  \cr
   \hfill   E_{+\delta_{i'} - \delta_{j'}} \, := \, \e_{i',j'} - \e_{j'+n,i'+n} \;\; ,  \qquad  E_{-\delta_{i'} + \delta_{j'}} \, := \, \e_{i'+m,j'+m} - \e_{j',i'}   \qquad \hfill \big(\, i' < j' \,\big)  \cr
   \hfill   E_{+\delta_{i'} + \delta_{j'}} \, := \, \e_{i',j'+n} + \e_{j',i'+n} \;\; ,  \qquad  E_{-\delta_{i'} - \delta_{j'}} \, := \, \e_{i'+n,j'} + \e_{j'+n,i'}   \qquad \hfill \big(\, i' \not= j' \,\big)  \cr
   E_{+2\,\delta_{i'}} \, := \, \e_{i',i'+n} \quad ,  \hskip25pt \qquad  E_{-2\,\delta_{i'}} \, := \, \e_{i'+n,i'}   \hskip79pt \quad  \cr
   E_{+\varepsilon_i + \delta_{j'}} \, := \, \e_{i,j'+n} + \e_{j',i+n} \;\; ,  \qquad  E_{-\varepsilon_i - \delta_{j'}} \, := \, \e_{i+m,j'} - \e_{j'+m,i}   \hskip29pt \qquad  \cr
   E_{+\varepsilon_i - \delta_{j'}} \, := \, \e_{i,j'} - \e_{j'+n,i+m} \;\; ,  \qquad  E_{-\varepsilon_i + \delta_{j'}} \, := \, \e_{i+m,j'+n} + \e_{j',i}   \hskip29pt \qquad  \cr
   E_{+\delta_{j'}} := \, +\sqrt{2} \, \big( \e_{2m+1,j'+n} + \e_{j',2m+1} \big) \;\; ,  \quad  E_{-\delta_{j'}} := \, -\sqrt{2} \, \big( \e_{j'+n,2m+1} - \e_{2m+1,j'} \big)  }  $$
where  $ \, \pm(\varepsilon_i \pm \varepsilon_j) \, $,  $ \, \pm \varepsilon_i \, $,  $ \, \pm(\delta_{i'} \pm \delta_{j'}) \, $,  $ \, \pm 2 \delta_{i'} \, $,  $ \, \pm(\varepsilon_i \pm \delta_{j'}) \, $,  $ \, \pm \delta_{j'} \, $  are the roots of the orthosymplectic Lie superalgebra  $ \fg $  as in  \cite{fss},  \S 2.27.  The  $ H_i $'s  are just  $ H_\alpha $,  with  $ \, \alpha \in \Pi'_\fg \, $  where  $ \Pi'_\fg $  is chosen as follows:
  $$  \displaylines{
   \Pi'_{B(m,n)} \, := \, \big\{ \delta_1 - \delta_2 \, , \dots , \, \delta_{n-1} - \delta_n \, , \, 2 \, \delta_n \, , \, \varepsilon_1 - \varepsilon_2 \, , \dots , \, \varepsilon_{m-1} - \varepsilon_m \, , \, \varepsilon_m \big\}   \hfill  \cr
   \hfill   \text{if  $ \, m \not= 0 \, $,}  \quad  \cr
   \Pi'_{B(0,n)} \, := \, \big\{ \delta_1 - \delta_2 \, , \dots , \, \delta_{n-1} - \delta_n \, , \, 2 \, \delta_n \big\}   \hfill  \cr
   \Pi'_{C(n)} \, := \, \big\{ \varepsilon_1 \, , \, \delta_1 - \delta_2 \, , \dots , \, \delta_{n-2} - \delta_{n-1} \, , \, 2 \, \delta_{n-1} \big\}   \hfill  \cr
   \Pi'_{D(m,n)} \! := \! \big\{ \delta_1 \! - \delta_2 \, , \dots , \delta_{n-1} \! - \delta_n \, , \dots , 2 \, \delta_n \, , \varepsilon_1 \! - \varepsilon_2 \, , \dots , \varepsilon_{m-1} \! - \varepsilon_m \, , \varepsilon_{m-1} \! + \varepsilon_m \big\}   \hfill  }  $$
(using standard notation).  Note that in all cases but  $ B(0,n) $  the chosen  $ \Pi'_\fg $  is just a distinguished set of simple roots.
%
%
 Setting  $ \, j' := M \! + j = 2\,m \! + 1 \, $,  with  $ \, m = 1 \, $  in case  $ C(n) \, $,  the corresponding coroots are
  $$  \displaylines{
  \quad   H_{\delta_j - \delta_{j+1}}  \, := \,  \big( \e_{j',j'} - \e_{j'+n,j'+n} \big) - \big( \e_{j'+1,j'+1} - \e_{j'+1+n,j'+1+n} \big)   \hfill  \text{in all cases}  \cr
  \quad   H_{\varepsilon_i - \varepsilon_{i+1}}  \, := \,  \big( \e_{i,i} - \e_{i+m,i+m} \big) - \big( \e_{i+1,i+1} - \e_{i+1+m,i+1+m} \big)   \hfill  \text{in all cases}  \cr
  \quad   H_{2\,\delta_n}  \, := \,  \big( \e_{n',n'} - \e_{n'+n,n'+n} \big)   \hfill  \text{in all cases} \cr
  \quad   H_{\varepsilon_1}  \, := \,  \big( \e_{1,1} - \e_{2,2} \big)   \hfill  \text{for  $ \, C(n) $}  \cr
  \quad   H_{2\,\delta_{n-1}}  \, := \,  \big( \e_{(n-1)', \, (n-1)'} - \e_{(n-1)'+n, \, (n-1)'+n} \big)   \hfill  \text{for  $ \, C(n) $}  \cr
  \quad\!   H_{\varepsilon_{m-1} + \varepsilon_m}  \! := \!  \big( \e_{m-1,m-1} \! - \e_{m-1+m,m-1+m} \big) \! + \!  \big( \e_{m,m} \! - \e_{m+m,m+m} \big)   \hfill  \text{for  \!\! $ D(m,\!n) $}  }  $$
 \vskip3pt
   Now, a Chevalley basis is formed by the root vectors  $ \, X_\alpha := E_\alpha \, $  and the Cartan generators (simple coroots)  $ \, H_\alpha \, $  as above: the verification follows by a careful, yet entirely straightforward, calculation.
 \vskip7pt
   $ \underline{F(4) \, , \, G(3)} \phantom{\bigg|} \; $:  \;  See  \cite{ik},  Theorem 3.9 (which applies to every  {\sl basic\/}  type).
 \vskip4pt
   $ \underline{D(2,1;a)} \, ,  \; a \in \Z \;\, $:  \;  Recall that  $ \, \fg = D(2,1;a) \, $  is a contragredient Lie superalgebra.  To describe it, we fix a specific choice of Dynkin diagram and corresponding Cartan matrix, like in  \cite{fss}, \S 2.28  (first choice), namely
  $$  {\buildrel 2 \over {\textstyle \bigcirc}} \hskip-3,5pt
\joinrel\relbar\joinrel\relbar\hskip-7pt{\buildrel 1 \over -}\hskip-7pt\relbar\joinrel\relbar\joinrel
\hskip-3,5pt
{\buildrel 1 \over {\textstyle \bigotimes}}
\hskip-3,5pt
\joinrel\relbar\joinrel\relbar\hskip-7pt{\buildrel a \over -}\hskip-7pt\relbar\joinrel\relbar\joinrel
\hskip-3,5pt
{\buildrel 3 \over {\textstyle \bigcirc}}  \hskip17pt ,   \hskip35pt
  {\big( a_{i,j} \big)}_{i,j=1,2,3;}  \, :=  \begin{pmatrix}
                                              0  &  1  &  a \,  \\
                                             -1  &  2  &  0 \,  \\
                                             -1  &  0  &  2 \,
                                             \end{pmatrix}  $$
Then  $ \, \fg = D(2,1;a) \, $  is defined as the Lie superalgebra over  $ \KK $  with generators  $ \; h_i \, $,  $ \, e_i \, $,  $ \, f_i \;\; (i=1,2,3) \, $,  \, with degrees  $ \; p(h_i) := 0 \, $,  $ \, p(e_i) := \delta_{1,i} \, $,  $ \, p(f_i) := \delta_{1,i} \;\; (i=1,2,3) \, $,  \, and with relations (for all  $ \, i,j=1,2,3 \, $)
  $$  \displaylines{
   \big[ h_i, h_j \big] = 0 \;\; ,  \qquad  \big[ e_1, e_1 \big] = 0 \;\; ,  \qquad  \big[ f_1, f_1 \big] = 0 \;\; ,  \cr
   \big[ h_i, e_j \big] = +a_{i,j} \, e_j \;\; ,  \qquad  \big[ h_i, f_j \big] = -a_{i,j} \, f_j \;\; ,  \qquad  \big[ e_i, f_j \big] = \delta_{i,j} \, h_i \;\; .  }  $$
Moreover, the root system is given by  $ \; \Delta_-  =  -\Delta_+ \; $  and
  $$  \Delta_+  \, = \,  \{ \alpha_1 \, , \, \alpha_2 \, , \, \alpha_3 \, , \, \alpha_1 + \alpha_2 \, , \, \alpha_1 + \alpha_3 \, , \, \alpha_1 + \alpha_2 + \alpha_3 \, , \, 2 \, \alpha_1 + \alpha_2 + \alpha_3 \}  $$
   \indent   Now we introduce the following elements:
  $$  \begin{matrix}
   e_{1,2} := \big[ e_1 , e_2 \big] \; ,  &  e_{1,3} := \big[ e_1 , e_3 \big] \; ,  &  e_{1,2,3} := \big[ e_{1,2} , e_3 \big] \; ,  &  e'_{1,1,2,3} := \big[ e_1 , e_{1,2,3} \big]  \phantom{\Big|}  \\
   f_{2,1} := \big[ f_2 , f_1 \big] \; ,  &  f_{3,1} := \big[ f_3 , f_1 \big] \; ,  &  e_{3,2,1} := \big[ f_3, f_{2,1} \big] \; ,  &  f'_{3,2,1,1} := \big[ f_{3,2,1} , f_1 \big]  \phantom{\Big|}
      \end{matrix}  $$
All these are root vectors, say  $ \, e_1 \! = \! X_{\alpha_1} \, $,  $ \, f_{3,1} \! = \! X_{-(\alpha_1 + \alpha_3)} \, $,  $ \, e_{1,2,3} \! = \! X_{\alpha_1 + \alpha_2 + \alpha_3} \, $,  and so on.  These, together with the original generators, do form a  $ \KK $--basis  of  $ \fg \, $.  The relevant new brackets among all these basis elements   --- dropping the zero ones, those coming from others by (super-)skewcommuta\-tivity, and those involving the  $ h_i $'s  (which are given by the fact that the  $ e_\bullet $'s  and the  $ f_\bullet $'s  are root vectors, involving all roots of  $ \fg \, $)  ---   are the following:
  $$  \displaylines{
   \quad   \big[ e_1 , e_2 \big] = e_{1,2} \; ,  \quad  \big[ e_1 , e_3 \big] = e_{1,3} \; ,  \quad  \big[ e_1 , e_{1,2,3} \big] = e'_{1,1,2,3}   \hfill  \cr
   \quad   \big[ e_1 , f_{2,1} \big] = f_2 \; ,  \quad  \big[ e_1 , f_{3,1} \big] = a \, f_3 \; ,  \quad   \big[ e_1 , f'_{3,2,1,1} \big] = -(1\!+\!a) f_{3,2,1}  \cr
   \quad   \big[ e_2 , e_{1,3} \big] = -e_{1,2,3} \; ,  \quad  \big[ e_2 , f_{2,1} \big] = f_1 \; ,  \quad  \big[ e_2 , f_{3,2,1} \big] = f_{3,1}   \phantom{\Big|}   \hfill  \cr
   \quad   \big[ e_3 , e_{1,2} \big] = -e_{1,2,3} \; ,  \quad  \big[ e_3 , f_{3,1} \big] = f_1 \; ,  \quad  \big[ e_3 , f_{3,2,1} \big] = f_{2,1}   \phantom{\Big|}   \hfill  \cr
   \quad   \big[ f_1 , f_2 \big] = -f_{2,1} \; ,  \quad  \big[ f_1 , f_3 \big] = -f_{3,1} \; ,  \quad  \big[ f_1 , f_{3,2,1} \big] = f'_{3,2,1,1}   \hfill  \cr
   \quad   \big[ f_1 , e_{1,2} \big] = e_2 \; ,  \quad  \big[ f_1 , e_{1,3} \big] = a \, e_3 \; ,  \quad   \big[ f_1 , e'_{1,1,2,3} \big] = (1\!+\!a) \, e_{1,2,3}  \cr
   \quad   \big[ f_2 , f_{3,1} \big] = f_{3,2,1} \; ,  \quad  \big[ f_2 , e_{1,2} \big] = -e_1 \; ,  \quad  \big[ f_2 , e_{1,2,3} \big] = -e_{1,3}   \phantom{\Big|}   \hfill  \cr
   \quad   \big[ f_3 , f_{2,1} \big] = f_{3,2,1} \; ,  \quad  \big[ f_3 , e_{1,3} \big] = -e_1 \; ,  \quad  \big[ f_3 , e_{1,2,3} \big] = -e_{1,2}   \phantom{\Big|}   \hfill  \cr
   \quad   \big[ e_{1,2} , e_{1,3} \big] = -e'_{1,1,2,3} \;\; ,  \qquad  \big[ e_{1,2} , f_{2,1} \big] \, = \, h_1 \! - \! h_2 \;\; ,   \phantom{\Big|}   \hfill  \cr
   \hfill   \big[ e_{1,2} , f_{3,2,1} \big] = a \, f_3 \;\; ,  \qquad  \big[ e_{1,2} , f'_{3,2,1,1} \big] = (1\!+\!a) f_{3,1}  \qquad  \phantom{\Big|}  \qquad  \cr
   \quad   \big[ e_{1,3} , f_{3,1} \big] \, = \, h_1 \! - \! a h_3 \; ,  \quad  \big[ e_{1,3} , f_{3,2,1} \big] = f_2 \; ,  \quad  \big[ e_{1,3} , f'_{3,2,1,1} \big] = (1\!+\!a) f_{2,1}   \phantom{\Big|}   \hfill  \cr
   \quad   \big[ f_{2,1} , f_{3,1} \big] = -\!f'_{3,2,1,1} \; ,  \quad  \big[ f_{2,1} , e_{1,2,3} \big] = a \, e_3 \; ,  \quad  \big[ f_{2,1} , e'_{1,1,2,3} \big] = -(1\!+\!a) \, e_{1,3}   \phantom{\Big|}   \hfill  \cr
   \quad   \big[ f_{3,1} , e_{1,2,3} \big] = e_2 \; ,  \quad  \big[ f_{3,1} , e'_{1,1,2,3} \big] = -(1\!+\!a) \, e_{1,2}   \phantom{\Big|}   \hfill  \cr
   \quad   \big[ e_{1,2,3} , f_{3,2,1} \big] = h_1 \! - \! h_2 \! - \! a h_3 \; ,  \quad  \big[ e_{1,2,3} , f'_{3,2,1,1} \big] = -(1\!+\!a) f_1 \;\; ,   \phantom{\Big|}   \hfill  \cr
   \hfill   \big[ f_{3,2,1} , e'_{1,1,2,3} \big] = -(1\!+\!a) \, e_1 \; ,  \quad\!  \big[ e'_{1,1,2,3} , f'_{3,2,1,1} \big] = -(1\!+\!a) \big( 2 h_1 \! - \! h_2 \! - \! a h_3 \big)  }  $$
Now we modify just two root vectors taking (recall  $ \, a \not= -1 \, $  by assumption)
  $$  e_{1,1,2,3} \, := \, +{(1\!+\!a)}^{-1} \, e'_{1,1,2,3} \;\;\; ,  \qquad  f_{3,2,1,1} \, := \, -{(1\!+\!a)}^{-1} f'_{3,2,1,1} \;\;\; ;  $$
then the above formulas has to be modified accordingly (many coefficients  $ (1\!+a) $  cancel out).  Looking at the final outcome it is then easy to check that
  $$  B  \, := \,  {\big\{ H_i \, , \, e_i \, , \, f_i \big\}}_{i=1,2,3} \,{\textstyle \bigcup}\, \big\{ e_{1,2} \, , e_{1,3} \, , e_{1,2,3} \, , e_{1,1,2,3} \, , f_{2,1} \, , f_{3,1} \, , f_{3,2,1} \, , f_{3,2,1,1} \big\}  $$
with  $ \; H_1 := h_1 \, $,  $ \; H_2 := {(1\!+a)}^{-1} \big( 2 h_1 \! - \! h_2 \! - \! a h_3 \big) \, $,  $ \; H_3 := h_3 \; $,  \, is indeed a Chevalley basis for  $ \, \fg = D(2,1;a) \, $.
 \vskip7pt
   $ \underline{P(n)} \, ,  \; n \not= 3 \;\, $:  \;  We fix the distinguished set of even simple roots
  $$  \Pi'_{P(n)}  \, := \,  \big\{\, \varepsilon_1 - \varepsilon_2 \, , \dots , \, \varepsilon_{n-1} - \varepsilon_n \, , \, \varepsilon_n - \varepsilon_{n+1} \, , \, 2 \, \varepsilon_{n+1} \,\big\}  $$
and the corresponding  {\sl even\/}  (simple) coroots, which are  $ \; H_i := H_{\varepsilon_i - \varepsilon_{i+1}} =
\, \big( \e_{i,i} - \e_{i+1,i+1} \big) -
\big( \e_{i+n+1,i+n+1} - \e_{i+1+n+1,i+1+n+1} \big) \; $
($ \, \forall \; 1 \! \leq \! i \! \leq \! n \, $),
$ \, H_{n+1} := H_{2\,\varepsilon_{n+1}} =
\big( \e_{n+1,n+1} - \e_{2(n+1),2(n+1)} \big) \; $.
As root vectors (the odd roots being   $ \, \pm \beta_{i,j} := \pm (\varepsilon_i
+ \varepsilon_j) \;\; \forall \; i \not= j \, $,  and  $ \, \gamma_i :=
2 \, \varepsilon_i \;\; \forall \; i \, $,  as in  \cite{fss},  \S 2.48) we take 
 \vskip5pt
   \centerline{ $ \displaystyle \qquad  \text{\sl (even)}  \hfill   E_{\alpha_{i,j}} := \, \e_{i,j} - \e_{n+1+j,n+1+i}   \hfill  \forall \;\, 1 \leq i \not= j \leq n $ }
 \vskip4pt
   \centerline{ $ \displaystyle \qquad  \text{\sl (odd)}  \hfill   E_{\gamma_i} := \, \e_{i,n+1+i}   \hfill  \forall \;\, 1 \leq i \leq n+1 $ }
 \vskip4pt
   \centerline{ $ \displaystyle \qquad  \text{\sl (odd)}  \hfill   E_{+\beta_{i,j}} := \, \e_{i,n+1+j} + \e_{j,n+1+i} =: E_{+\beta_{j,i}}   \hfill  \forall \;\, 1 \leq i < j \leq n $ }
 \vskip4pt
   \centerline{ $ \displaystyle \qquad  \text{\sl (odd)}  \hfill   E_{-\beta_{i,j}} := \, \e_{n+1+j,i} - \e_{n+1+i,j} =: -E_{-\beta_{j,i}}   \hfill  \forall \;\, 1 \leq i < j \leq n $ }
 \vskip5pt
\noindent
 Direct check shows that the above elements  $ H_i $  and root vectors form a basis as required.  This follows from the commutation formulas in  \cite{fss},  \S 2.48,  which only need the following correction:  $ \, \big[ E_{\alpha_{i,j}} \, , E_{\beta_{i,j}} \big] = 2 \, E_{\gamma_i} \; $  ($ i \! \not= \! j $).
\end{proof}

\bigskip

\begin{remark}  \label{unif_exist_Chev-bases}
 {\it A uniform proof of  Theorem \ref{exist_che-bas}}.

\smallskip

   We sketch here, quite roughly (and up to some details) another possible proof of  Theorem \ref{exist_che-bas},  kindly suggested by the referee.  This works by a {\sl uniform argument\/}  for all basic cases   --- like in  \cite{ik},  but with different arguments ---   and can also be adopted again (once the definition of Chevalley basis is set up) for the strange case  $ Q(n) $  as well.  Thus in the end only the strange case  $ P(n) $  is left apart: therefore, we assume hereafter that  $ \fg $  is  of  {\sl basic\/}  type.
 \vskip5pt
   To begin with, for the  $ H_i $'s  in the Chevalley basis
   (belonging to  $ \fh $)  one proceeds like in the proof above.
     For the  {\sl even\/}  root vectors  $ X_\alpha $  ($ \alpha \in \Delta_0 $)
   one takes them as they are given in a ``standard'' Chevalley basis
   of the Lie algebra  $ \fg_0 $   ---
 with easy adaptations when  $ \fg_0 $  is reductive.
   Finally, for the  {\sl odd\/}  root vectors  $ X_\beta $  ($ \beta \in
   \Delta_1 $)  one of course has to choose each one of them in the
   root space  $ \fg_\beta \, $,  which is one-dimensional, yet then
   one also must carefully fix a proper normalization to get integral
   coefficients in the expression for the Lie brackets.
 \vskip3pt
   As a first step, note that in all basic cases the  $ \fg_0 $--module  $ \fg_1 $  is a direct sum of simple  $ \fg_0 $--modules  whose highest weight is  {\sl minuscule\/}  (or  ``{\sl nonzero minimal\/} dominant'' in Humphreys' terminology, cf.~\cite{hu},  \S 13, exercise 13).
                                                                \par
   Now, a simple  $ \fg_0 $--module  $ V(\lambda) $  with minuscule
   highest weight  $ \lambda $ is as follows (cf.~\cite{ja},  Ch.~5A.1).  First, the set
of weights of such a  $ V(\lambda) $  is just  $ W.\lambda \, $,  the
$ W $--orbit  of  $ \lambda \, $,  where  $ W $  is the Weyl group of
$ \fg_0 \, $.  Then each weight space  $ {V(\lambda)}_\mu $  in  $
V(\lambda) $ is one-dimensional, hence given by a single basis vector
$ v_\mu $  so that  $ \, {V(\lambda)}_\mu = \KK \, v_\mu \; $.
 Thus we start with a  $ \KK $--vector  space  $ V $  having basis  $ {\big\{ v_\mu \big\}}_{\mu \in W.\lambda} \; $:  then a  $ \fg_0 $--module  structure on  $ V $,  for which it has highest  $ \lambda \, $,  is given by the following simple formulas:
  $$  \displaylines{
   \hfill   H.v_\mu := \mu(H) \, v_\mu \;\; ,  \;\qquad  \forall \;\; H \in \fh \; ,  \quad  \forall \;\; \mu \in W.\lambda   \hfill \phantom{(3.1)}  \cr
   \hfill   X_{+\alpha}.v_\mu := 0 \; ,  \quad  X_{-\alpha}.v_\mu := 0 \; ,  \;\;\quad  \forall \;\; \alpha \in \Delta_0^+ \, , \, \mu \in W.\lambda \; : \; \mu(H_\alpha) = 0   \hfill \phantom{(3.1)}  \cr
   \hfill   X_{+\alpha}.v_\mu := 0 \; ,  \;\;  X_{-\alpha}.v_\mu := v_{\mu-\alpha} \; ,  \quad  \forall \; \alpha \in \Delta_0^+ \, , \, \mu \in W.\lambda \, : \, \mu(H_\alpha) = +1   \hfill (3.1)  \cr
   \hfill   X_{+\alpha}.v_\mu := v_{\mu+\alpha} \; ,  \;\;  X_{-\alpha}.v_\mu := 0 \; ,  \quad  \forall \; \alpha \in \Delta_0^+ \, , \, \mu \in W.\lambda \, : \, \mu(H_\alpha) = -1   \hfill \phantom{(3.1)} }  $$
                                                             \par
We now shall use these remarks in order to construct at the same time
an isomorphic copy of  $ \fg $  and a Chevalley basis inside it.
 \vskip3pt
   First, we know that  $ \fg_1 $  as an  $ \fh $--module  splits into  $ \, \fg_1 = \oplus_{\beta \in \Delta_1} \fg_\beta \, $  where each odd root space  $ \fg_\beta $  is one-dimensional.  Now we fix a nonzero vector  $ \, y_\beta \in \fg_\beta \! \setminus \{0\} \, $  for each  $ \, \beta \in \Delta_1 \, $:  we shall find our odd root vector  $ X_\beta $  in the Chevalley basis by a suitable normalization of  $ y_\beta \, $,  which is fixed by imposing relations  {\it (c)\/}  and  {\it (d)\/}  in  Definition \ref{def_che-bas}.
                                                             \par
   Now note that  $ \, [\, y_\beta \, , y_{-\beta}] \in \fh \, $.
Let us call $H(\beta):=[\, y_\beta \, , y_{-\beta}]$.
Therefore, for all  $ \, \alpha \in \Delta_0 \, $  one has
  $$  \big[X_\alpha , [\, y_\beta \, , y_{-\beta}] \big] \, = \, \big[ X_\alpha \, , H(\beta) \big] \, = \, - \alpha\big(H(\beta)\big) \, X_\alpha  $$
while the   Jacobi identity yields
  $$  \big[ X_\alpha \, , [\, y_\beta \, , y_{-\beta}] \big] \, = \, \big[ [X_\alpha \, , y_\beta] \, , y_{-\beta} \big] + \big[\, y_\beta \, , [X_\alpha \, , y_{-\beta}] \big] \, = \, 0  $$
when  $ \, \alpha(H_b) = \beta(H_\alpha) = 0 \, $,  since in that case one has  $ \, (\alpha \pm \beta) \not\in \Delta \, $.  This means that  $ \; \alpha\big(H(\beta)\big) = 0 \iff \alpha(H_\beta) = 0 \, $, \, so that  $ H(\beta) $  is a scalar multiple of  $ H_\beta \, $,  say  $ \, H(\beta) = n_\beta \, H_\beta \, $  for some  $ \, n_\beta \in \KK \, $;  moreover, one has  $ \, n_\beta \not= 0 \, $  as the same analysis gives  $ \, \big[ X_\alpha \, , [\, y_\beta \, , y_{-\beta}] \big] \not= 0 \, $  when $ \, \beta(H_\alpha) \not= 0 \, $.
                                                             \par
   Therefore, we shall fix our odd root vectors  $ \, X_\beta \; (\beta \in \Delta_1) \, $  as given by  $ \; X_\beta := n_\beta^{-1/2} \, y_\beta \; $  (which makes sense because  $ \KK $  is algebraically closed): it follows that  $ \; \big[ X_\beta \, , \, X_{-\beta} \big] = H_\beta \, $,  \, so relations  {\it (c)\/}  in  Definition \ref{def_che-bas}  do hold.
 \vskip3pt
   Now we modify the Lie superalgebra structure on  $ \fg \, $,
   keeping the same vector space structure but changing the Lie
   bracket  $ \, [\,\ ,\ ] \, $  as follows.  We keep  $ \, [\,\ ,\ ]
   \, $  untouched when restricted to  $ \fg_0 $  (hence  $ \fg_0 $
   keeps the same Lie algebra structure) and to  $ \fg_1 \, $,
   i.e.~when computed on elements which are homogeneous of the same
   parity.  On the other hand, we modify the bracket on elements of
   different parities, simply by re-defining the (adjoint) action of  $ \fg_0 $  onto  $ \fg_1 $  using formulas (3.1)  {\sl with the  $ X_\beta $'s  playing the role of the  $ v_\mu $'s}.
 In other words, we (re)normalize the  $ \fg_0 $--action  on  $ \fg_1
 \, $,  so to have an  {\sl isomorphic copy\/}  $ \fg'_1 $  of the  $
 \fg_0 $--module  $ \fg_1 \, $,  whose structure is described by (3.1)
 with the  $ X_\beta $'s  replacing the  $ v_\mu $'s.
                                                            \par
   In addition, the Lie  bracket of  $ \fg $  defines on the  $ \fg_0 $--module  $ \fg_1 $  a  $ \fg_0 $--valued,  symmetric bilinear form  $ \; \psi : (\eta\,,\zeta) \mapsto \psi(\eta\,,\zeta) := [\,\eta\,,\zeta\,] \; $,  \, for which the Jacobi identity reads
  $$   x.\psi(\eta\,,\zeta)  \; = \;  \psi(x.\eta\,,\zeta) + \, \psi(\eta\,,x.\zeta)   \eqno \forall \;\; x \in \fg_0 \, , \; \eta, \zeta \in \fg_1  \qquad  (3.2)  $$
(where the  $ \fg_0 $--action  on  $ \fg_0 $  itself is again the adjoint action).  Using any  $ \fg_0 $--module  isomorphism  $ \, \Phi : \fg_1 \, {\buildrel \cong \over \longrightarrow} \, \fg'_1 \, $,  we define a form  $ \, \psi' : \fg'_1 \times \fg'_1 \longrightarrow \fg_0 \, $  by  $ \, \psi' := \psi \circ \Phi^{\times 2} \, $  which again will enjoy the similar properties as in (3.2).  Then the formula  $ \; \big[\, \eta', \zeta' \big] := \psi'\big( \eta', \zeta' \big) \; $   --- for all  $ \eta' , \zeta' \in \fg'_1 \, $  ---   defines a  $ \fg_0 $--valued  bracket on  $ \fg'_1 \, $:  along with the  $ \fg_0 $--action  on  $ \fg'_1 $  and the Lie bracket on  $ \fg_0 $  itself, this uniquely determines an overall bracket on  $ \, \fg' := \fg_0 \oplus \fg'_1 \, $.  By construction,  $ \fg' $  with this bracket is a Lie superalgebra isomorphic to  $ \fg \, $.
 \vskip3pt
%
%
%
 \vskip3pt
   Finally, we still have to check that the new odd root vectors  $ X_\beta $  we chose do satisfy   --- for the Lie superalgebra  $ \fg' $  ---   conditions  {\it (d)\/}  in  Definition \ref{def_che-bas}.  This follows by direct check: indeed, if  $ \, \gamma, \delta \in \Delta_1 \, $  with  $ \, \gamma + \delta \not= 0 \, $  then we have  $ \; [X_\gamma \, , \, X_\delta] = c_{\gamma,\delta} \, X_{\gamma+\delta} \; $  for some  $ \, c_{\gamma,\delta} \in \KK \, $,  and so
  $$  \big[ X_{-(\gamma+\delta)} \, , \, [X_\gamma \, , \, X_\delta] \big]  \; = \;  c_{\gamma,\delta} \, \big[ X_{-(\gamma+\delta)} \, , \, X_{\gamma+\delta} \big]  \; = \;  c_{\gamma,\delta} \, H_{-(\gamma+\delta)}   \eqno (3.3)  $$
On the other hand, the Jacobi identity gives
  $$  \displaylines{
   \big[ X_{-(\gamma+\delta)} \, , [X_\gamma \, , X_\delta] \big]  \; = \;  \big[ [X_{-(\gamma+\delta)} \, , X_\gamma] \, , X_\delta \big] \, + \, \big[ X_\gamma \, , [X_{-(\gamma+\delta)} \, , X_\delta] \big]  \; =   \hfill  \cr
   \hfill   = \,  [X_{-\delta} \, , X_\delta] + [X_\gamma \, , X_{-\gamma}]  \, = \,  -H_{-\delta} + H_\gamma  \, = \,  H_\delta + H_\gamma  \, = \,  H_{\gamma+\delta}  \, = \,  -H_{-(\gamma+\delta)}  \cr
 }  $$
Comparing with (3.3), this gives  $ \, c_{\gamma,\delta} = -1 \, $,  which actually proves that the conditions required in  Definition \ref{def_che-bas}{\it (d)\/}  actually do hold.
 \vskip3pt
   Tiding everything up, we eventually find that   --- by construction ---   the new odd root vectors actually complete our Chevalley basis, q.e.d.
\end{remark}

%

 %
 %

\chapter{Kostant superalgebras}  \label{kost-superalgebra}

 {\it
   In this chapter, we introduce the ``Kostant superalgebra''  $ \kzg $
attached to  $ \fg $  and to a fixed Chevalley basis  $ B $  of  $ \fg \, $.
This is a  $ \Z $--integral  form of  $ \, U(\fg) $,  defined as the unital
$ \, \Z $--subalgebra of  $ \, U(\fg) $  generated by odd root vectors,
divided powers in the even root vectors and binomial coefficients
in the Cartan elements (in  $ B \, $),  very much the same as in the
ordinary setting.
                                                              \par
   The properties of Lie brackets of elements in the
Chevalley basis  $ B $,  allow us to describe the ``commutation rules''
among the generators of  $ \, \kzg \, $.  Using these rules we can prove
``PBW-like Theorem'' for  $ \kzg \, $,  in the following terms: we fix any
total order of  $ B \, $,  and then   

\medskip

   ``$ \, \kzg $  is a free  $ \Z $--module,  and a  $ \Z $--basis
is given by the set of ordered   monomials in the generators of
$ \, \kzg \, $''.

\medskip

   The proof proceeds by standard arguments, using the commutation rules
to expand any unordered monomial in the generators as a
$ \Z $--linear  combination of ordered ones.

\smallskip

   As in ordinary setting, the Kostant form inside the
enveloping superalgebra of $\fg$ and its PBW--theorem play an
absolutely crucial role in the definition and construction of
Chevalley supergroups.

\smallskip

   Again the notion of Kostant superalgebra, and its description via a
PBW-like theorem  --- though not for all types ---
have appeared previously in the literature in the
works  \cite{ik},  \cite{sw}.  Here again, we are providing a unified
treatment and a more general version of the results available in
the literature.
 }

\bigskip

   Let $\KK$ be an algebraically closed field of characteristic zero.

\medskip

   Throughout this section we assume $\fg$ to be a classical Lie superalgebra, with  $ \fg $  not of type  $ A(1,1) $,  $ P(3) $,  $ Q(n )$  or  $ D(2,1;a) $  with  $ \, a \not\in \Z \, $.  We treat cases  {\it A-P-Q\/}  in  \S~\ref{cases A-P-Q},  while  $ D(2,1;a) $  with  $ \, a \not\in \Z \, $  is disposed of in  \cite{ga}.

\bigskip
\bigskip
\bigskip

  \section{Kostant's  $ \Z $--form}  \label{kost-form}

\bigskip

   For any  $ \KK $--algebra  $ A \, $,  we define
the  {\it binomial coefficients}
  $$  \bigg(\! {y \atop n} \!\bigg)  \; :=  \;  {\frac{\, y (y \! - \! 1) \cdots (y \! - \! n \! + \! 1) \,}{n!}}  $$
for all $ \, y \! \in \! A \, $,  $ \, n \in \N \, $.  We recall a (standard) classical result, concerning
$ \Z $--integral  valued polynomials in a polynomial algebra  $ \KK \big[ y_1, \dots, y_\ell \big] \, $:

\bigskip

\begin{lemma}  \label{int-polynomials}
   (cf.~\cite{hu},  \S 26.1)  Let  $ \, \KK[y_1, \dots, y_t] \, $  be the\/  $ \KK $--algebra  of polynomials in the indeterminates  $ y_1 \, $, $ \dots \, $,  $ y_t \, $.  Let also
  $$  \text{\it Int}_{\Z} \Big( \KK \big[ y_1, \dots, y_t \big] \Big)  \; := \;  \Big\{\, f \! \in \KK \big[ y_1, \dots, y_t \big] \;\Big|\, f(z_1, \dots, z_t) \in \Z \;\; \forall \, z_1, \dots, z_t \in \! \Z \,\Big\}  $$
Then  $ \, \text{\it Int}_{\Z} \big( \KK[y_1, \dots, y_t] \big) \, $  is a  $ \Z $--subalgebra  of  $ \, \KK[y_1, \dots, y_t] \, $,  which is free as a  $ \Z $--(sub)module,  with  $ \Z $--basis  $ \; \big\{\, {\textstyle \prod_{i=1}^t} \, \big({y_i \atop n_i}\big) \;\big|\; n_1, \dots, n_t \in \N \,\big\} \; $.
\end{lemma}

\bigskip

   Let  $ U(\fg) $  be the universal enveloping superalgebra of  $ \fg \, $.  We fix a Chevalley basis  $ \; B \, = \, \{ H_1 , \dots , H_\ell \} \, \coprod \, {\big\{ X_\alpha \big\}}_{\alpha \in \Delta} \, $  of  $ \fg $  as in  \S \ref{che_bas-alg},  and let  $ \, \fh_\Z \, $  be the free  $ \Z $--module  with basis  $ \, \{ H_1, \dots, H_\ell \} \phantom{\big|} $.  Given  $ \, h \in U(\fh) \, $,  we denote by  $ \, h(H_1,\dots,H_\ell) \, $  the expression of  $ h $  as a function of the  $ H_i $'s.  As immediate consequence of  Lemma \ref{int-polynomials},  we have the following:

\bigskip

\begin{corollary}  \label{int-polynom_kost} {\ }
 \vskip5pt
    {\it (a)} \,  $ \mathbb{H}_\Z := \Big\{\, h \! \in \! U(\fh) \;\Big|\; h\big(z_1,\dots,z_\ell\big) \in \Z \, , \,\; \forall \, z_1, \dots, z_\ell \in \Z \,\Big\} \, $  is a  {\sl free}  $ \Z $--sub\-module  of  $ \, U(\fh) \, $,  with basis  $ \, B_{U(\fh)} := \Big\{\, {\textstyle \prod_{i=1}^\ell \! \Big(\! {H_i \atop n_i} \!\Big)} \,\Big|\, n_1, \dots, n_\ell \! \in \! \N \,\Big\} \, $.
 \vskip3pt
   {\it (b)} \, The  $ \Z $--subalgebra  of  $ \, U(\fg) $  generated by all the elements  $ \, \Big(\! {{H - z} \atop n} \!\Big) \, $  with  $ \, H \in \fh_\Z \, $,  $ \, z \in \Z \, $,  $ \, n \in \N \, $,  coincides with  $ \, \mathbb{H}_\Z \, $.
\end{corollary}

\bigskip

   Now, mimicking the classical construction, we define a  $ \Z $--form  of  $ U(\fg) \, $:

\bigskip

\begin{definition}  \label{def-kost-superalgebra}
   We call ``divided powers'' all elements  $ \;  X_\alpha^{(n)} \! := X_\alpha^n \big/ n! \, $,  \, for  $ \, \alpha \! \in \! \Delta_0 \, $,  $ \, n \! \in \! \N \, $.
 We call  {\sl Kostant superalgebra},  or Kostant's  $ \Z $--form  of  $ U(\fg) \, $,  the unital  $ \Z $--subsuperalgebra  of  $ U(\fg) \, $, denoted by  $ \kzg \, $,  generated by
  $$  X_\alpha^{(n)} \; ,  \!\quad X_\gamma \; ,  \!\quad  {\textstyle \Big(\! {H_i \atop n} \!\Big)}
\qquad \hfill \forall \;\; \alpha \! \in \! \Delta_0 \; , \; n \! \in \! \N \, ,
\; \gamma \! \in \! \Delta_1 \, , \; i = 1, \dots, \ell \, .  $$
\end{definition}

\bigskip

\begin{remarks}
   Let  $ \, \kzgz \, $  be the unital  $ \Z $--subalgebra  of  $
U(\fg_0) $  generated by the elements  $ \; X_\alpha^{(n)} \, $,  $
\, \Big(\! {H_i \atop n} \!\Big) \; $  with  $ \, \alpha \in \Delta_0
\, $,  $ \, n \in \N \, $.  This gives back ``almost'' a classical
object, namely the Kostant's  $ \Z $--form  of  $ U(\fg_0) \, $:  the
latter is defined (and well-known) in terms of a classical Chevalley
basis when  $ \fg_0 $  is semisimple   --- which is often, but not
always, the case for simple Lie superalgebras  $ \fg $  of classical
type; moreover, that definition can be easily extended to the
reductive case (depending on a choice ).  Nevertheless, in general the algebra
$ \, \kzgz \, $  we are considering is slightly different, namely a bit larger,
because of the way we chose the Cartan elements  $ H_1 \, $,  $ \dots $, 
$ H_\ell $  in the Chevalley basis.
\end{remarks}

\bigskip

   Another important classical result is the following:

\bigskip

\begin{lemma}  \label{comm_class-div-pow}
   (cf.~\cite{hu},  \S 26.2)  Let  $ \, \alpha \in \Delta_0 \, $,  and  $ \, m, n \in \N \, $.  Then
  $$  X_\alpha^{(n)} \, X_{-\alpha}^{(m)}  \,\; = \;\,  {\textstyle \sum\limits_{k=0}^{min(m,n)}} \; X_{-\alpha}^{(m-k)} \, {\textstyle \Big({{H_\alpha - m - n + 2 \, k} \atop k}\Big)} \, X_\alpha^{(n-k)}  $$
\end{lemma}

\bigskip
\bigskip
\bigskip

  \section{Commutation rules}  \label{comm-rul_kost}

\medskip
 \vskip7pt   

   In the classical setup, a description of  $ \kzgz $  comes from a ``PBW-like'' theorem: namely,  $ \kzgz $  is a free  $ \Z $--module  with  $ \Z $--basis  the set of ordered monomials (w.~r.~to any total order) whose factors are divided powers in the root vectors  $ X_\alpha $  ($ \alpha \in \Delta_0 \, $)  or binomial coefficients in the  $ H_i $  ($ \, i = 1, \dots, \ell \, $).
                                        \par
   We shall prove a similar result in the ``super-framework''.  Like in the classical case, this follows from a direct analysis of commutation rules among divided powers in the even root vectors, binomial coefficients in the  $ H_i $'s  and odd root vectors.  To perform such an analysis, we list hereafter all such rules; in particular, we also need to consider slightly more general relations.
%
%
                                        \par
   The relevant feature is that all coefficients in these relations are in  $ \Z \, $.

 \vskip5pt

   We split the list into two sections:  {\it (1)\/}  relations involving only  {\sl even\/}  generators (known by classical theory);  {\it (2)\/}  relations involving also  {\sl odd\/}  generators.

 \vskip5pt

   All these relations are proved via simple induction arguments: the classical ones (in the first list) are well-known (see  \cite{hu},  \S 26),  and the new ones are proved in a similar way, using  Theorem \ref{exist_che-bas}.  Details are left to the reader.
%
%
%
%

\noindent
 {\bf (1) Even generators only  {\rm  (that is  $ \, \Big(\! {H_i \atop m} \!\Big) $'s  and  $ \, X_\alpha^{(n)} $'s  only,  $ \, \alpha \in \Delta_0 \, $)}:}
 \vskip1pt
  $$  \displaylines{
   \hfill   {\textstyle \Big({H_i \atop n}\Big)} \, {\textstyle \Big({H_j \atop m}\Big)}  \; = \;  {\textstyle \Big({H_j \atop m}\Big)} \, {\textstyle \Big({H_i \atop n}\Big)}   \qquad    \hfill (4.1)  \cr
   \hfill   \phantom{{}_{\big|}} \forall \;\; i,j \in \{1,\dots,\ell\} \, , \;\;\; \forall \;\; n, m \in \N   \qquad \qquad  \cr
   \hfill   X_\alpha^{(n)} \, f(H)  \; = \;  f\big(H - n \; \alpha(H)\big) \, X_\alpha^{(n)}   \qquad   \hfill (4.2)  \cr
   \hfill   \phantom{{}_{\big|}} \forall \;\; \alpha \in \Delta_0 \, , \; H \in \fh \, , \; n \in \N \, , \; f(T) \in \KK[T]   \qquad \qquad  \cr
   \hfill   X_\alpha^{(n)} \, X_\alpha^{(m)}  \; = \;  {\textstyle \Big(\! {{n \, + \, m} \atop m} \!\Big)} \, X_\alpha^{(n+m)}   \quad \qquad  \forall \; \alpha \in \Delta_0 \, , \;\; \forall \;\; n, m \in \N  \phantom{{}_{\Big|}}   \hfill (4.3)  \cr
   \hfill   X_\alpha^{(n)} \, X_\beta^{(m)}  \; = \;  X_\beta^{(m)} \, X_\alpha^{(n)} \, + \, \textit{l.h.t}   \quad \qquad  \forall \; \alpha, \beta \in \Delta_0 \, , \;\; \forall \;\; n, m \in \N   \hfill (4.4)  }  $$
where  {\it l.h.t.}~stands for a  $ \Z $--linear  combinations of monomials in the  $ X_\delta^{(k)} $'s  and in the $ \Big(\! {H_i \atop c} \!\Big) $'s  whose ``height''   --- that is, by definition, the sum of all ``exponents''  $ k $ occurring in such a monomial ---   is less than  $ \, n+m \, $.
                                         \par
   A  {\sl special case\/}  is the following (already seen in  Lemma \ref{comm_class-div-pow}):
 \vskip-7pt
  $$  \displaylines{
   \hfill   X_\alpha^{(n)} \, X_{-\alpha}^{(m)}  \,\; = \;\,  {\textstyle \sum_{k=0}^{min(m,n)}} \; X_{-\alpha}^{(m-k)} \, {\textstyle \Big( {{H_\alpha \, - \, m \, - \, n \, + \, 2 \, k} \atop k} \Big)} \, X_\alpha^{(n-k)}   \qquad   \hfill (4.5)  \cr
   \hfill   \forall \; \alpha \in \Delta_0 \, , \;\; \forall \;\; m, n \in \N   \qquad \qquad  }  $$
 \vskip21pt

\noindent
 {\bf (2) Odd and even generators  {\rm  (also involving the  $ X_\gamma $'s,  $ \, \gamma \in \Delta_1 \, $)}:}
 \vskip1pt
  $$  \displaylines{
   \hfill   X_\gamma \, f(H)  \; = \;  f\big(H - \gamma(H)\big) \, X_\gamma   \qquad   \hfill (4.6)  \cr
   \hfill   \phantom{{}_{\big|}} \forall \;\; \gamma \in \Delta_1 \, , \; h \in \fh \, , \; f(T) \in \KK[T]   \qquad \qquad  \cr
   \hfill   c_{\gamma,\gamma} \, X_{2\,\gamma}  \; = \;  \big[ X_\gamma, X_\gamma \big]  \; = \;  2 \, X_\gamma^2 \quad ,  \qquad  \forall \;\; \gamma \in \Delta_1   \hfill (4.7)  \cr
   \text{hence}   \hfill   2 \, \gamma \not\in \Delta  \;\; \Longrightarrow \;\;  X_\gamma^n = 0 \; ,  \qquad  \forall \;\; n \geq 2  \hfill (4.8) \cr
   \text{and}   \hfill   2 \, \gamma \in \Delta  \;\; \Longrightarrow \;\;  X_\gamma^2 = c_{\gamma,\gamma} \big/ 2 \cdot X_{2\,\gamma} = \pm 2 \, X_{2\,\gamma}   \hfill (4.9)  \cr
   \big(\, \text{because} \;\;\;  c_{\gamma,\gamma} = \pm 4  \quad  \text{if}  \quad  \gamma \, , \, 2 \, \gamma \in \Delta \, , \;\; \text{see  Definition \ref{def_che-bas}} \,\big)  \phantom{{}_{\big|}}   \hfill  \cr
   \hfill   X_{-\gamma} \, X_\gamma  \; = \;  - X_\gamma \, X_{-\gamma} \, + \, H_\gamma   \hskip29pt  \forall \;\, \gamma \in \Delta_1 \cap \big(\! -\Delta_1 \big)   \hfill (4.10)  \cr
   \text{with}  \quad  H_\gamma := \big[ X_\gamma \, , \, X_{-\gamma} \big] \, \in \, \fh_\Z \; ,  \phantom{{}_{\big|}}  \cr
   \hfill   X_\gamma \, X_\delta  \; = \;  - X_\delta \, X_\gamma \, + \, c_{\gamma,\delta} \, X_{\gamma + \delta} \;\, ,   \hfill   \;\; \phantom{{}_{|}} \forall \;\, \gamma , \delta \in \Delta_1 \, , \; \gamma + \delta \not= 0   \hskip11pt (4.11)  \cr
   \text{with  $ \, c_{\gamma,\delta} \, $  as in  Definition \ref{def_che-bas},}  \phantom{{}_{\big|}}  \cr
   \hfill   X_\alpha^{(n)} \, X_\gamma  \; = \;  X_\gamma \, X_\alpha^{(n)} \, + \, {\textstyle \sum_{k=1}^n} \, \Big( {\textstyle \prod_{s=1}^k} \, \varepsilon_s \Big) \, {\textstyle \Big(\! {{r \, + \, k} \atop k} \!\Big)} \, X_{\gamma + k \, \alpha} \, X_\alpha^{(n-k)}   \hfill (4.12)  \cr
   \hfill   \forall \;\; n \in \N \, ,  \;\;\; \forall \;\; \alpha \in \Delta_0 \, , \; \gamma \in \Delta_1  \, : \;  \alpha \not= \pm 2 \, \gamma \, ,   \phantom{{}_{\big|}}  \quad  \cr
   \text{with} \hskip9pt  \sigma^\alpha_\gamma = \big\{ \gamma - r \, \alpha \, , \dots , \gamma \, , \dots , \gamma + q \, \alpha \,\big\} \; ,  \hskip7pt  X_{\gamma + k \, \alpha} := 0  \,\;\text{\ if\ }\;  (\gamma \! + \! k \, \alpha) \not\in \Delta \, ,   \hfill  \cr
   \text{and}  \hskip11pt   \varepsilon_s = \pm 1  \text{\;\;\ such that \ }  \big[ X_\alpha \, , X_{\gamma + (s-1) \, \alpha} \big] = \, \varepsilon_s \, (r+s) \, X_{\gamma + s \, \alpha} \;\; ,  \phantom{{}_{\big|}}  \hfill  \cr
   \hfill   X_\gamma \, X_\alpha^{(n)}  \; = \;  X_\alpha^{(n)} \, X_\gamma \;\; ,  \qquad  X_{-\gamma} \, X_{-\alpha}^{\,(n)}  \; = \;  X_{-\alpha}^{\,(n)} \, X_{-\gamma}   \hskip21pt   \hfill (4.13)  \cr
   \hfill   X_{-\gamma} \, X_\alpha^{(n)}  \; = \;  X_\alpha^{(n)} \, X_{-\gamma} \, + \, z_\gamma \; \gamma(H_\gamma) \; X_\alpha^{(n-1)} \, X_\gamma   \hskip21pt   \hfill (4.14)  \cr
   \hfill   X_\gamma \, X_{-\alpha}^{\,(n)}  \; = \;  X_{-\alpha}^{\,(n)} \, X_\gamma \, - \, z_\gamma \; \gamma(H_\gamma) \; X_{-\alpha}^{\,(n-1)} \, X_{-\gamma}   \hskip21pt \phantom{{}_{|}}   \hfill (4.15)  \cr
   \hfill   \forall \;\; n \in \N \; ,  \;\;\; \forall \;\; \gamma \in \Delta_1 \; , \;\; \alpha = 2 \, \gamma \in \Delta_0 \; , \;\; z_\gamma := c_{\gamma,\gamma} / 2 = \pm 2   \qquad  }  $$

\bigskip

\begin{remark}
   In  \cite{sw}   --- dealing with the orthosymplectic case only ---   the following commutation formula is given
%
  $$  X_\gamma \, {H_i \choose t}  \; = \;  \sum\nolimits_{r=0}^t \, {(-1)}^{t-r} \, {\gamma(H_i) \choose {t \, - \, r}} \cdot {H_i \choose r} \, X_\gamma  $$
for all  $ \, H_i = H_{\alpha_i} \, $,  with  $ \, \alpha_i \in
\Delta_0 \, $  simple, and  $ \, \gamma \in \Delta_1 \, $.
Actually, this is equivalent to (4.6), because the  $ \Big(\! {H_i
\atop m} \!\Big) $'s  generate the polynomials in  $ H_i \, $,
and in general the following identity holds:
%
%
  $$  {{H_\alpha \, - \, \gamma(H_\alpha)} \choose t}  \; = \;  \sum\nolimits_{r=0}^t \, {(-1)}^{t-r} \, {\gamma(H_\alpha) \choose {t \, - \, r}} \cdot {H_\alpha \choose r}  $$
\end{remark}

\bigskip
\bigskip
\bigskip

  \section{Kostant's PBW-like theorem}  \label{kost-theo}

\bigskip

   We are now ready to state and prove our super-version of Kostant's theorem:

\bigskip

\begin{theorem}  \label{PBW-Kost}
 The Kostant superalgebra  $ \, \kzg \, $  is a free\/  $ \Z $--module.  For any given total order  $ \, \preceq \, $  of the set  $ \, {\Delta} \cup \big\{ 1, \dots, \ell \big\} \, $,  a  $ \, \Z $--basis  of  $ \kzg $  is the set  $ {\mathcal B} $  of ordered ``PBW-like monomials'', i.e.~all products (without repeti\-tions) of factors of type  $ X_\alpha^{(n_\alpha)} $,  $ \Big(\! {H_i \atop n_i} \!\Big) $  and  $ \, X_\gamma $   --- with  $ \, \alpha \in \Delta_0 \, $,  $ \, i \in \big\{ 1, \dots, \ell \big\} \, $,  $ \, \gamma \in \Delta_1 \, $  and  $ \, n_\alpha $,  $ n_i  \in \N $  ---   taken in the right order with respect to  $ \, \preceq \; $.
\end{theorem}

\begin{proof}  Let us call ``monomial'' any product in any possible order, possibly with repetitions, of several  $ X_\alpha^{(n_\alpha)} $'s,  several  $ \Big(\! {{H_i - z_i} \atop s_i} \!\Big) $'s   --- with  $ z_i \! \in \! \Z $  ---   and several  $ X_\gamma^{\,m_\gamma} $'s, $m_\gamma \in \N$.
For any such monomial  $ \M \, $,  we define three numbers, namely:
 \vskip4pt
%
%
   --- its ``height''  $ \, \textit{ht}(\M) \, $,  \, namely the sum of all  $ n_\alpha $'s  and  $ m_\gamma $'s  occurring in  $ \M $  (so it does not depend on the binomial coefficients in the  $ H_i $'s);
 \vskip3pt
   --- its ``factor number''  $ \, \textit{fac}(\M) \, $,  \, defined to be the total number of factors (namely  $ X_\alpha^{(n_\alpha)} $,  $ \Big(\! {{H_i - z_i} \atop n_i} \!\Big) $  or  $ X_\gamma \, $)  within  $ \M $  itself;
 \vskip3pt
   --- its ``inversion number''  $ \, \textit{inv}(\M) \, $,  \, which is the sum of all inversions of the order  $ \, \preceq \, $  among the indices of factors in  $ \M $  when read from left to right.
 \vskip5pt
   We can now act upon any such  $ \M $  with any of the following operations:
 \vskip4pt
   {\it --\,(1)} \,  we move all factors  $ \Big(\! {H_i - z_i \atop s} \!\Big) $  to the leftmost position, by repeated use of relations (4.2) and (4.6): this produces a new monomial  $ \M' $  multiplied on the left by a suitable product of several (new) factors  $ \Big(\! {{H_i - \check{z}_i} \atop s_i} \!\Big) \, $;
 \vskip4pt
   {\it --\,(2)} \,  we re-write any power of an odd root vector, say  $ \, X_\gamma^{\,n_\gamma} $,  $ \, n_\gamma \! > \! 1 \, $,  \, as
  $$  X_\gamma^{\,n_\gamma}  \, = \;  X_\gamma^{\,2 \, d_\gamma + \epsilon_\gamma}  \, = \;  z \, X_{2 \gamma}^{\,d_\gamma} \, X_\gamma^{\,\epsilon_\gamma}  \, = \;  z \, d_\gamma! \; X_{2 \gamma}^{(d_\gamma)} \, X_\gamma^{\,\epsilon_\gamma}  $$
for some  $ \, z \in \Z \, $,  $ \, d_\gamma \in \N \, $,  $ \, \epsilon_\gamma \in \{0,1\} \, $  with
$ \, n_\gamma = 2 \, d_\gamma + \epsilon_\gamma \, $,  using (4.7--9);
 \vskip4pt
   {\it --\,(3)} \,  whenever two or more factors  $ X_\alpha^{(n)} $'s  get side by side, we splice them together into a single one, times an integral coefficient, using (4.3);
 \vskip4pt
   {\it --\,(4)} \,  whenever two factors within  $ \M $  occur side by side  {\sl in the wrong order w.~r.~to\/}  $ \preceq \; $,  i.e.~we have any one of the following situations
  $$  \displaylines{
   \M  \; = \;  \cdots \, X_\alpha^{(n_\alpha)} \, X_\beta^{(n_\beta)} \cdots \quad ,
 \qquad  \M  \; = \;   \cdots \, X_\alpha^{(n_\alpha)} \, X_\gamma^{\,m_\gamma} \cdots  \cr
   \M  \; = \;  \cdots \, X_\delta^{\,m_\delta} \, X_\alpha^{(n_\alpha)} \cdots \quad ,
 \qquad  \M  \; = \;  \cdots \, X_\gamma^{\,m_\gamma} \, X_\delta^{\,m_\delta} \cdots  }  $$
with  $ \, \alpha \succneqq \beta \, $,  $ \, \alpha \succneqq \gamma \, $,  $ \, \delta \succneqq \alpha \, $  and  $ \, \gamma \succneqq \delta \, $  respectively, we can use all relations (4.4--5) and (4.10--18) to re-write this product of two distinguished factors, so that all of  $ \M $  expands into a  $ \Z $--linear  combination of new monomials.
 \vskip7pt
   By definition,  $ \kzg $  is spanned over  $ \Z $  by all (unordered) monomials in the  $ X_\alpha^{(n)} $'s,  the $\Big( {{H_i} \atop n_i} \Big)$
and the  $ X_\gamma $'s.
Let us consider one of these, say  $ \M \, $.
                                            \par
   First of all,  $ \M $  is a PBW-like monomial, i.e.~one in  $ {\mathcal B} \, $,  if and only if no one of steps  {\it (2)\/}  to  {\it (4)\/}  may be applied.  But if not, we now see what is the effect of applying such steps.  We begin with  {\it (1)\/}:  applied to  $ \M \, $,  it gives
  $$  \M  \; = \;  {\mathcal H} \, \M'  $$
where  $ {\mathcal H} $  is some product of  $ \Big(\! {H_i - \check{z}_i \atop s_i} \Big) \, $'s,  and  $ \M' $  is a new monomial such that
  $$  \textit{ht}\big(\M'\big) = \textit{ht}\big(\M\big) \; ,  \quad  \textit{fac}\big(\M'\big) \leq \textit{fac}\big(\M\big) \; ,  \quad  \textit{inv}\big(\M'\big) \leq \textit{inv}\big(\M\big)  $$
and the strict inequality in the middle holds if and only if  $ \, {\mathcal H} \not= 1 \, $,  \, i.e.~step  {\it (1)\/}  is non-trivial.  Indeed, all this is clear when one realizes that  $ \M' $  is nothing but  ``$ \M $  ripped off of all factors  $ \Big(\! {{H_i - z_i} \atop s_i} \!\Big) $'s.''
 \vskip4pt
   Then we apply any one of steps  {\it (2)},  {\it (3)\/}  or  {\it (4)\/}  to  $ \M' \, $.  Step  {\it (2)\/}
gives
  $$  \M'  \; = \;  z \, \M''  \quad ,   \eqno  \text{with\ \ } \textit{ht}\big(\M''\big) \, \lneqq \, \textit{ht}\big(\M'\big)  \hskip1pt \qquad  $$
for some  $ \, z \in \Z \, $  and some monomial  $ \M'' \, $ (possibly zero).
%
%
   Step  {\it (3)\/}  yields
  $$  \M'  \; = \;  z \, \M^\vee  \;\; ,   \eqno  \text{with\ \ } \textit{ht}\big(\M^\vee\big) \! = \textit{ht}\big(\M'\big) \, ,  \,\; \textit{fac}\big(\M^\vee\big) \lneqq \textit{fac}\big(\M'\big)  \quad  $$
for some  $ \, z \in \Z \, $  and some monomial  $ \M^\vee \, $.  Finally, step  {\it (4)\/}  instead gives
  $$  \displaylines{
   \hskip35pt   \M'  \; = \;\,  \M^\curlyvee + \; {\textstyle \sum_k} \; z''_k \, \M_k \;\; ,   \hfill  \text{with\ \ } \; \textit{ht}\big(\M_k\big) \lneqq \textit{ht}\big(\M'\big) \quad \forall \;\, k \; ,  \quad  \cr
   \hfill   \text{and\ \ } \; \textit{ht}\big(\M^\curlyvee\big) = \textit{ht}\big(\M'\big) \; ,  \;\; \textit{inv}\big(\M^\curlyvee\big) \lneqq \textit{inv}\big(\M'\big)  \quad  }  $$
where  $ \, z_k \in \Z \, $  (for all  $ k \, $),  and  $ \M^\curlyvee $  and the  $ \M_k $'s  are monomials.
 \vskip5pt
   In short, through either step  {\it (2)},  or  {\it (3)},  or  {\it (4)},  we achieve an expansion
  $$  \M'  \; = \;  {\textstyle \sum_k} \; z''_k \, {\mathcal H} \, \M'_k \;\; ,  \qquad  z_k \in \Z   \eqno (4.16)  $$
where   --- unless the step is trivial, for then we get all equalities ---   we have
  $$  \Big( \textit{ht}\big(\M'_k\big) \! \lneqq \textit{ht}\big(\M'\big) \!\Big) \vee \Big( \textit{fac}\big(\M'_k\big) \! \lneqq \textit{fac}\big(\M'\big) \!\Big) \vee \Big( \textit{inv}\big(\M'_k\big) \! \lneqq \textit{inv}\big(\M'\big) \!\Big)   \eqno (4.17)  $$
   \indent   Now we can repeat, namely we apply step  {\it (1)\/}  and step  {\it (2)\/}  or  {\it (3)\/}  or  {\it (4)\/}  to every monomial  $ \M_k $  in (4.16); and then we iterate.  Thanks to (4.17), this process will stop after finitely many iterations.  The outcome then will read
 \vskip-7pt
  $$  \M'  \; = \;  {\textstyle \sum_j} \; \dot{z}_j \, {\mathcal H''_j} \, \M''_j \;\; ,  \qquad  \dot{z}_j \in \Z   \eqno (4.18)  $$
where  $ \; \textit{inv}\big(\M''_j\big) = 0 \; $  for every index
$ j \, $,  i.e.~all monomials  $ \M''_j $  are ordered and without
repetitions, that is they belong to  $ {\mathcal B} \, $.  Now
each  $ {\mathcal H''_j} $  belongs to  $ \mathbb{H}_\Z $  (notation of  Corollary \ref{int-polynom_kost}),  just by construction.  Then  Corollary \ref{int-polynom_kost}  ensures that each  $ {\mathcal H''_j} $  expand into a  $ \Z $--linear  combination of ordered monomials in the  $ \Big(\! {H_i \atop n_i} \!\Big) $'s.  Therefore (4.18) yields
  $$  \M  \; = \;  {\textstyle \sum_s} \;\, \hat{z}_s \; {\mathcal H^\wedge_s} \, \M^\wedge_s \;\; ,  \qquad  \hat{z}_s \in \Z   \eqno (4.19)  $$
where every  $ {\mathcal H^\wedge_s} $  is an ordered monomial,
without repetitions, in the  $ \Big(\! {H_i \atop n_i} \!\Big) $'s,
while for each  index  $ s $  we have  $ \, \M^\wedge_s = \M''_j
\, $  for some  $ j \, $   --- those in (4.18).
                                            \par
   Using again   --- somehow ``backwards'', say ---   relations (4.2) and (4.6), we can ``graft'' every factor  $ \Big(\! {H_i \atop n_i} \!\Big) \, $  occurring in each  $ {\mathcal H^\wedge_s} $  in the right position inside the monomial  $ \M^\wedge_s \, $,  so to get a new monomial  $ \M^\circ_s $  which is ordered, without repetitions, but might have factors of type  $ \Big(\! {{H_i - z_i} \atop m_i} \!\Big) $  with  $ \, z_i \in \Z \setminus \{0\} \, $   --- thus not belonging to  $ {\mathcal B} \, $.  But then  $ \, \Big(\! {{H_i - z_i} \atop m_i} \!\Big) \in \mathbb{H}_\Z \, $,  hence again  by  Corollary \ref{int-polynom_kost}  that factor expands into a  $ \Z $--linear  combination of ordered monomials, without repetitions, in the  $ \Big(\! {H_j \atop \ell_j} \!\Big) $'s.  Plugging every such expansion inside each monomial  $ \M^\wedge_s $  instead of each  $ \Big(\! {{H_i - z_i} \atop m_i} \!\Big) $   ---  $ \, i = 1, \dots, \ell \, $  ---   we eventually find
  $$  \M  \; = \;  {\textstyle \sum_r} \;\, z_r \; \M^!_r \;\; ,  \qquad  z_r \in \Z   \eqno (4.20)  $$
 \vskip3pt
\noindent
 where now every  $ \M^!_r $  is a PBW-like monomial,  i.e.~$ \, \M^!_r \in {\mathcal B} \, $  for every  $ r \, $.
 \vskip5pt
   Since  $ \kzg \, $,  by definition, is spanned over  $ \Z $  by all monomials in the  $ X_\alpha^{(n)} $'s,  the  $ X_\gamma $'s  and the  $ \Big(\! {H_i \atop m} \!\Big) $'s,  by the above  $ \kzg $  is contained in  $ \, \textit{Span}_{\,\Z}({\mathcal B}) \, $.  On the other hand, by definition and by  Corollary \ref{int-polynom_kost},  $ \textit{Span}_{\,\Z}({\mathcal B}) $  in turn is contained in  $ \kzg \, $.  So  $ \; \kzg = \textit{Span}_{\,\Z}({\mathcal B}) \, $,  i.e.~$ {\mathcal B} $  spans  $ \kzg $  over  $ \Z \, $.
 \vskip5pt
   At last, the ``PBW theorem'' for Lie superalgebras over fields ensures that  $ {\mathcal B} $  is a  $ \KK $--basis  for  $ U(\fg) \, $,  because  $ \; B := \big\{ H_1 , \dots , H_\ell \big\} \, {\textstyle \coprod} \, \big\{ X_\alpha \,\big|\, \alpha \in \Delta \big\} \; $  is a  $ \KK $--basis  of  $ \fg \, $  (cf.~\cite{vsv}).  Thus  $ {\mathcal B} $  is linearly independent over  $ \KK \, $,  hence over  $ \Z \, $.  Therefore  $ {\mathcal B} $  is a  $ \Z $--basis  for  $ \kzg \, $,  and the latter is a free  $ \Z $--module.
\end{proof}

\bigskip

\begin{remarks}  {\ }
 \vskip3pt
   {\it (a)} \,  To give an example, let us fix any total order  $ \preceq $  in  $ \, \Delta \cup \big\{ 1, \dots, \ell \big\} \, $  such that  $ \; \Delta_0 \preceq \big\{ 1, \dots, \ell \big\} \preceq \Delta_1 \;\, $.  Then the basis  $ {\mathcal B} $  from  Theorem \ref{PBW-Kost}  reads
  $$  {\mathcal B}  =  \Big\{\, {\textstyle \prod_{\alpha \in \Delta_0}} X_\alpha^{(n_\alpha)} \, {\textstyle \prod_{i=1}^\ell} {\textstyle \Big(\! {H_i \atop n_i} \!\Big)} \, {\textstyle \prod_{\gamma \in \Delta_1}} X_\gamma^{\epsilon_\gamma} \,\;\Big|\;\, n_\alpha \, , \, n_i \! \in \! \N \, , \, \epsilon_\gamma \! \in \! \{0,1\} \Big\}   \eqno (4.21)  $$
 \vskip1pt
   {\it (b)} \,  For  $ \, \fg = \mathfrak{gl}(m|n) \, $,  a  $ \Z $--basis  like (4.21) was more or less known in literature (see, e.g.,  \cite{bku}).  For  $ \, \fg = \mathfrak{osp}(n|m) \, $,  it is given in  \cite{sw},  Theorem 3.6.
\end{remarks}

\bigskip

   Theorem \ref{PBW-Kost}  has a direct consequence.  To state it, note that  $ \fg_1^\Z $,  the odd part of  $ \, \fg^\Z \, $,  has  $ \, \big\{ X_\gamma \,\big|\, \gamma \! \in \! \Delta_1 \big\} \, $  as  $ \Z $--basis,  by construction; then let  $ \, \bigwedge \fg_1^\Z \, $  be the exterior\/  $ \Z $--algebra  over  $ \fg_1^\Z \, $.  Then we have an integral version of the well known factorization  $ \; U(\fg) \, \cong \, U(\fg_0) \otimes_{\KK} \bigwedge \fg_1 \; $  (see  \cite{vsv}):

\vskip19pt

\begin{corollary}  \label{kost_tens-splitting}
   There exists an isomorphism of  $ \, \Z $--modules
  $$  \kzg  \,\; \cong \;\,  K_\Z(\fg_0) \, \otimes_\Z {\textstyle \bigwedge} \, \fg_1^\Z  $$
\end{corollary}

\begin{proof}  Let us choose a PBW-like basis  $ {\mathcal B} $  of  $ \kzg $   --- from  Theorem \ref{PBW-Kost}  ---   as in (4.21).  Then each PBW-like monomial can be factorized into a left factor  $ \lambda $  times a right factor  $ \rho \, $,  namely
 $ \; \prod\limits_{\alpha \in \Delta_0} \! X_\alpha^{(n_\alpha)} \, \prod\limits_{i=1}^\ell \Big(\! {H_i \atop n_i} \!\Big) \, \prod\limits_{\gamma \in \Delta_1} \!
X_\gamma^{\epsilon_\gamma} = \lambda \cdot \rho \; $
 with  $ \, \lambda := \prod_{\alpha \in \Delta_0} X_\alpha^{(n_\alpha)} \, \prod_{i=1}^\ell \Big(\! {H_i \atop n_i} \!\Big) \, $  and  $ \, \rho := \! \prod_{\gamma \in \Delta_1} \! X_\gamma^{\epsilon_\gamma} \, $.  But all the  $ \lambda $'s  span  $ K_\Z(\fg_0) $  over  $ \Z $  (by classical Kostant's theorem) while the  $ \rho $'s  span  $ \bigwedge \fg_1^\Z \, $.
\end{proof}

\bigskip

\begin{remarks}  {\ }
 \vskip5pt
   {\it (a)}  Following a classical pattern (and  cf.~\cite{bkl},  \cite{bku},  \cite{sw}  in the super context) we can define the  {\sl superalgebra of distributions\/}  $ \, {\mathcal D}\textit{ist}\,(G) \, $  on any supergroup  $ G \, $.
 Then  $ \; {\mathcal D}\textit{ist}\,(G) = \kzg \otimes_\Z \KK \; $,  when  $ \; \fg := \Lie\,(G) \; $  is classical.
 \vskip4pt
   {\it (b)}  In this section we proved that the assumptions of Theorem 2.8 in  \cite{sw}  do hold for any supergroup  $ G $  whose tangent Lie superalgebra is classical.  Therefore,  {\sl all results in  \cite{sw}  do apply to such supergroups}.
\end{remarks}

%

 %
 %

\chapter{Chevalley supergroups}  \label{che-sgroup}

\bigskip
 {\it
   This chapter is the core of our work: it presents the
construction of  ``Chevalley supergroups'' and gives their main
and fundamental properties.

\smallskip

   Our construction goes through some main steps we are
going to outline here.

\smallskip

   The first step demands to fix a specific  $ \fg $--module,  which
has to be faithful, semisimple, finite-dimensional and rational.

\smallskip

   Then any such  $ \fg $--module  $ V $,
contains a  $ \Z $--integral  form  $ M $  which is  $ \kzg $--stable;
such an  $ M $  will then be called an ``admissible lattice'' of  $ \, V \, $.
The choice of an admissible lattice  $ M $  for  $ V $  provides a functorial recipe yielding a representation of  $ \, A \otimes \kzg \, $  on  $ \, A \otimes M \, $,  for each  $ \, A \in \salg \, $  by scalar extension.

\smallskip


   Our second step consists in taking the supergroup functor  $ \, \rGL(M) : \! A \mapsto \rGL(M)(A) := \rGL(A \otimes M) \, $.  Inside it, we consider all the one-parameter supersubgroups corresponding to the root vectors and the Cartan elements in the Chevalley basis.
Then we take the (full) subgroup functor  $ \, G_V \, $, inside $ \, \rGL(M) \, $,
such that  $ \, G_V(A) \, $  is the subgroup generated by all these one-parameter supersubgroups.
We then define the sheafification of  $ \, G_V \, $,  denoted  $ \, \bG_V \, $,
to be our ``Chevalley supergroup'' attached to  $ \fg $  and to  $ V \, $.

\medskip

                                                 \par
   The third step consists in the proof of  the representability of  $ \, \bG_V $,
in other words we show that  $ \bG_V $  is the functor of points of an affine algebraic
supergroup, thus deserving truly the denomination of Chevalley supergroup.
We achieve this via a factorization result: namely, we show that  $ \bG_V $
factors, as a set-valued functor, into two set-valued functors  $ \, \bG_0 $
and  $ \, \bG_1 \, $,  and in order to do this we do through a fine analysis
of commutation rules between one-parameter supersubgroups.
The factors are then explicitly described:  $ \bG_0 $  is essentially a
classical reductive algebraic group realized through the classical Chevalley
construction, hence it is representable,  while $ \bG_1 $  is just an affine,
purely odd superspace, so it is also representable. We believe that this
description is in itself interesting, since it sheds light on the structure
of simple algebraic supergroups.
At this point it follows immediately that  the direct product
$ \, \bG_0 \times \bG_1 \cong \bG_V \, $  is also representable.
                                                 \par
   Notice that any affine supergroup   
admits such a factorization  $ \, \bG \cong \bG'_0 \times \bG'_1 \, $  
However in our construction we need to prove it, since this is our way to reach
the representability result.

\smallskip

   In the last part of the chapter, we show that  $ \bG_V $  depends only on  $ V \, $,  i.e.~it is independent of choice of an admissible lattice  $ M \, $.  Moreover, we prove that it has good functoriality properties with respect to  $ V \, $;  in particular, we explain the relation between different Chevalley supergroups  built upon the same classical Lie superalgebra  $ \fg $  and different  $ \fg $--modules  $ V \, $.  The outcome is that this relation is much the same as in the classical framework: the representation  $ V $  bears some topological data encoded into  $ \bG_0 \, $,  whereas the ``purely odd'' factor  $ \bG_1 $  carries nothing new.  Also, we prove Lie's Third Theorem for  $ \bG_V \, $:  namely, letting  $ \, \Lie(\bG_V) : \salg \longrightarrow \lie \, $  being the Lie algebra valued functor associated to the supergroup  $ \bG_V \, $,  we prove that  $ \, \Lie(\bG_V) = \fg \, $.
 }

 \section{Admissible lattices}  \label{adm-lat}

   Let  $ \KK $  be an algebraically closed field of characteristic zero.  If  $ R $  is a unital subring of  $ \KK \, $,  and  $ V $  a finite dimensional  $ \KK $--vector  space, a subset  $ \, M \! \subseteq \! V \, $  is called  $ R $--{\it lattice\/}  (or  $ R $--form)  of  $ V $  if  $ \, M \! = \text{\it Span}_R(B) \, $  for some basis  $ B $  of  $ V \, $.
                                                              \par
   Let  $ \fg $ be a classical Lie superalgebra (as above) over  $ \KK \, $,  with  $ \, {\text rk}(\fg) = \ell \, $,  let a Chevalley basis $ B $  of  $ \fg $  and the Kostant algebra  $ \kzg $  be as in \S\S 3--4.

\bigskip

\begin{definition}
 Let  $ V $  be a  $ \fg $--module,  and let  $ M $  be a  $ \Z $--lattice  of it.
 \vskip7pt
\noindent
 \hskip7pt  {\it (a)} \,  We call  $ V $  {\it rational\/}  if  $ \, \fh_\Z := \text{\it Span}_\Z\big(H_1,\dots,H_\ell\big) \, $  acts diagonally on  $ V $  with eigenvalues in  $ \Z \, $;  in other words, one has
  $$  V  \; = \;  {\textstyle \bigoplus_{\mu \in \fh^*}} V_\mu  \qquad  \text{with}  \quad  V_\mu  \, := \, \big\{ v \! \in \! V \,\big|\, h.v = \mu(h) \, v \; \forall \, h \! \in \! \fh \big\} $$
{\sl and}
  $$  \mu(H_i) \in \Z  \qquad  \text{for all  $ \; i=1, \dots, \ell \; $  and all  $ \; \mu \in \fh^* \; $  such that  $ \, V_\mu \not= \{0\} \, $.}  $$
 \vskip5pt
\noindent
 \hskip7pt  {\it (b)} \,  We call  $ M $  {\it admissible (lattice)\/}
--- of  $ V $  ---   if it is  $ \kzg $--stable.
\end{definition}

\bigskip

\begin{remark}
 If  $ \fg $  is  {\sl not\/}  of either types  $ A(m,n) $,  $ C(n) $,  $ Q(n) $  or  $ D(2,1;a) \, $,  then any  {\sl finite dimensional}  $ \, \fg $--module  $ V $  is automatically rational.  However in the other cases the rationality assumption  {\sl is\/}  actually restrictive.
\end{remark}

\bigskip

\begin{theorem}
  Let  $ \fg $  be a classical Lie superalgebra.
 \vskip7pt
   {\it (a)} \,  Any rational, finite dimensional, semisimple  $ \fg $--module  $ V $  contains an admissible lattice  $ M \, $.
 \vskip4pt
   {\it (b)} \,  Any admissible lattice  $ M $  as in  {\it (a)\/}  is the sum of its weight components,  that is to say  $ \, M = \bigoplus\limits_{\mu \in \fh^*} \big( M \cap V_\mu \big) \; $.
\end{theorem}

\begin{proof}
 The proof is the same as in the classical case.  Without loss of generality, we can assume that  $ V $  be irreducible of highest weight.  Letting  $ v $  be a highest weight vector, take  $ \, M := \kzg.v \; $;  then  $ M $  spans  $ V $  over  $ \KK \, $,  and it is clearly  $ \kzg $--stable  because  $ \kzg $  is a subalgebra of  $ U(\fg) \, $.  The proof that  $ M $  is actually a  $ \Z $--lattice  of  $ V $  and that  $ M $  splits into  $ \, M = \oplus_\mu \big( M \cap V_\mu \big) \, $  is detailed in  \cite{st},  \S 2, Corollary 1, and it applies here with just a few minor changes.
\end{proof}

\medskip

   We also need to know the stabilizer of an admissible lattice:

\medskip

\begin{theorem}  \label{stabilizer}
   Let  $ V $  be a rational, finite dimensional\/  $ \fg $--module,  $ M $  an ad\-missible lattice of  $ \, V $,  and  $ \, \fg_V \! = \! \big\{ X \! \in \! \fg \,\big|\, X.M \! \subseteq \! M \big\} \, $.  If  $ \, V $  is faithful, then
  $$  \fg_V  \, = \,  \fh_V \; {\textstyle \bigoplus \big( \oplus_{\alpha \in \Delta}} \,
\Z \, X_{\alpha} \big) \;\; ,  \qquad  \fh_V \, := \,
\big\{ H \in \fh \;\big|\; \mu(H) \in \Z \, ,
\,\; \forall \; \mu \in \Lambda \big\}  $$
where  $ \Lambda $  is the set of all weights of  $ \, V $.  In particular,  $ \fg_V $  is a lattice in  $ \fg \, $,  and it is independent of the choice of the admissible lattice  $ M $  (but not of  $ \, V $).
\end{theorem}

\begin{proof} The classical proof in  \cite{st}, \S 2, Corollary 2, applies again, with some additional arguments to take care of odd root spaces.  Indeed, with the same arguments as in  [{\it loc.~cit.}]  one shows that  $ \; \fg_V = \fh_V \oplus \big(\! \oplus_{\alpha \in \Delta} \! \big(\, \fg_V \! \cap \KK\,X_{\alpha} \big) \big) \; $;  then one still has to prove that  $ \; \fg_V \! \cap \KK\,X_{\alpha} = \, \Z \, X_\alpha \; $  for all  $ \, \alpha \in \Delta \, $.  The arguments in  [{\it loc.~cit.}]  also show that  $ \; \fg_V \! \cap \KK \, X_{\alpha} = \, \Z \, X_\alpha \; $  is a cyclic  $ \Z $--submodule  of  $ \fg_V \, $,  hence it may be spanned (over  $ \Z $)  by some  $ \, {{\,1} \over {\,n_\alpha}} \, X_\alpha \, $  with  $ \, n_\alpha \in \N_+ \, $  (for all  $ \, \alpha \in \Delta $).  What is left to prove is that  $ \, n = 1 \, $.
                                                             \par
   For every even root  $ \, \alpha \in \Delta_0 \, $,  one can repeat once more the argument in  [{\it loc.~cit.}]  and eventually find  $ \, n_\alpha = 1 \, $.  Instead, for each  $ \, \alpha \in \Delta_1 \, $  one sees   --- by an easy case by case analysis, for instance ---   that there exists  $ \, \alpha' \in \Delta_1 \, $  such that  $ \; \big( \alpha + \alpha' \big) \in \Delta_0 \, $,  $ \; \big( \alpha - \alpha' \big) \not\in \Delta \; $.  Then (notation of  Definition \ref{def_che-bas}{\it (d)\/})  $ \; c_{\alpha,\alpha'} = \pm 1 \; $,  \, and so  $ \; \Big[ X_{\alpha'} \, , \, {{\,1} \over {\,n_\alpha}} \, X_\alpha \Big] \, = \, {{\,1} \over {\,n_\alpha}} \, X_{\alpha + \alpha'} \; $.  On the other hand, clearly  $ \, X_{\alpha'} \in \fg^\Z \subseteq \fg_V \, $,  hence  $ \; \Big[ X_{\alpha'} \, , \, {{\,1} \over {\,n_\alpha}} \, X_\alpha \Big] \in \big(\, \fg_V \! \cap \KK \, X_{\alpha + \alpha'} \big) \; $,  with  $ \; \big(\, \fg_V \! \cap \KK \, X_{\alpha + \alpha'} \big) = \, \Z \, X_{\alpha + \alpha'} \; $  because  $ \, \big( \alpha + \alpha' \big) \in \Delta_0 \, $.  Then the outcome is  $ \; {{\,1} \over {\,n_\alpha}} \, X_{\alpha + \alpha'} \, \in \, \Z \, X_{\alpha + \alpha'} \; $,  \, which eventually yields  $ \, n = 1 \, $,  \, q.e.d.
\end{proof}

\smallskip

%
%
%

\begin{remark}
   Let  $ Q $  and  $ P $  respectively be the root lattice and the weight lattice of  $ \fg \, $;  one knows that there exists simple, rational, finite dimensional  $ \fg $--mo\-dules  $ V_Q $  and  $ V_P $  whose weights span  $ Q $  and  $ P $  respectively.  Then by  Theorem \ref{stabilizer}  one clearly has  $ \; \fg_{V_Q} \subseteq \fg_V \subseteq \fg_{V_P} \; $  for any rational, finite dimensional  $ \fg $--module  $ V \, $.

\end{remark}

\bigskip
\bigskip
\bigskip

 \section{Construction of Chevalley supergroups}  \label{const-che-sgroup}

\medskip

%
%
%
%
   From now on, we retain the notation of  \S \ref{che-sgroup}:  in particular,  $ V $  is a rational, finite dimensional  $ \fg $--module,  and  $ M $  is an admissible lattice of it.
                                                                      \par
   We fix a commutative unital  $ \Z $--algebra  $ \bk \, $:  as in  \S \ref{Lie-superalgebras},  we assume  $ \bk $  to be such that 2 and 3 are not zero and they are not zero divisors in  $ \bk \, $.
%
%
%
 We set
  $$  \fg_\bk \, := \, \bk \otimes_\Z \fg_V \;\; , \qquad  V_\bk \, := \, \bk \otimes_\Z M \;\; ,  \qquad  U_\bk(\fg) \, := \, \bk \otimes_\Z \kzg \;\; .  $$
%
%
Then  $ \fg_\bk $  acts faithfully on  $ V_\bk \, $,  which yields
an embedding of  $ \fg_\bk $  into  $ \rgl(V_\bk) \, $.

\medskip

   For any  $ \, A \in \salg_\bk = \salg \, $,  the Lie superalgebra  $ \, \fg_A := A \otimes_\bk \fg_\bk \, $  acts faithfully on  $ \, V_\bk(A) := \, A \otimes_\bk V_\bk \, $,  hence it embeds into  $ \rgl\big(V_\bk(A)\big) \, $,  etc.

\bigskip

   Let  $ \, \alpha \in \Delta_0 \, $,  $ \, \beta, \gamma \in \Delta_1 \, $,  and let  $ X_\alpha \, $,  $ X_\beta $  and  $ X_\gamma $  be the associated root vectors (in our fixed Chevalley basis of  $ \fg \, $).  Assume also that  $ \; \big[ X_\beta \, , X_\beta \big] = 0 \; $  and  $ \; \big[ X_\gamma \, , X_\gamma \big] \neq 0 \; $;  \, we recall that the latter occurs if and only if  $ \, 2\,\gamma \in \Delta \, $.
                                                         \par
   Every one of  $ X_\alpha \, $,  $ X_\beta $  and  $ X_\gamma $
acts as a nilpotent operator on  $ V \, $,  hence on  $ M \, $
and  $ V_\bk \, $,  i.e.~it is represented by a nilpotent matrix in  $
\rgl(V_\bk) \, $;  the same holds for
  $$  t \, X_\alpha \; ,  \;\; \vartheta \, X_\beta \; ,
\;\; \vartheta \, X_\gamma + t \, X_\gamma^{\,2} \;\; \in \;
\End\big(V_\bk(A)\big)   \eqno \forall \;\; t \in A_0 \, ,  \; \vartheta \in A_1 \;\, .   \qquad (5.1)  $$
   Taking into account that  $ X_\gamma $  and  $ X_\gamma^{\,2} $  commute, and  $ \, X_\gamma^{\,2} = \pm 2 \, X_{2\gamma} \, $   --- by (4.9) ---   we see at once that, for any  $ \, n \in \N \, $,  we have  $ \, Y^n \! \big/ n! \in \big(\kzg\big)(A) \, $  for any  $ Y $  as in (5.1); moreover,  $ \, Y^n \! \big/ n! = 0 \, $  for  $ \, n \gg 0 \, $,  by nilpotency.  Therefore, the formal power series  $ \; \exp(Y) := \sum_{n=0}^{+\infty} \, Y^n\big/n! \; $,  \, when computed for  $ Y $  as in (5.1), gives a well-defined element in  $ \rGL\big(V_\bk(A)\big) \, $,  expressed as finite sum.
                                                         \par
   In addition, expressions like (2.5--7) again make sense in this purely algebraic framework   --- up to taking  $ \rGL\big(V_\bk(A)\big) $  instead of  $ \rGL\big(V(T)\big) \, $.
 \vskip7pt

   For any  $ \, \alpha \in \Delta \subseteq \fh^* \, $,  let  $ \, H_\alpha \in \fh_\Z \, $  be the corresponding coroot  (cf.~\ref{root-syst}).  Let  $ \, V = \oplus_\mu V_\mu \, $  be the splitting of  $ V $  into weight spaces; as  $ V $  is rational, we have  $ \, \mu(H_\alpha) \in \Z \, $  for all  $ \, \alpha \in \Delta \, $.  Now, for any  $ \, A \in \salg \, $  and  $ \, t \in U(A_0) \, $   --- the group of invertible elements in  $ A_0 $  ---   we set
  $$  h_\alpha(t).v  \; := \;  t^{\mu(H_\alpha)} \, v   \eqno \forall \;\; v \in V_\mu \, ,   \; \mu \in \fh^* \; ;  \qquad  $$
this defines another operator (which also can be locally expressed by exponentials)
  $$  h_\alpha(t) \, \in \, \rGL\big(V_\bk(A)\big)   \eqno \forall \;\; t \in U(A_0) \;\; .  \qquad (5.2)  $$
More in general,  if  $ \; H = \sum_{i=1}^\ell a_i H_{\alpha_i} \in \fh_\Z \; $  we define  $ \; h_H(t) := \prod_{i=1}^\ell h_{\alpha_i}^{a_i}(t) \;\, $.

\bigskip

\begin{definition}  \label{chevalley}  {\ }
 \vskip4pt
   {\it (a)} \,  Let  $ \, \alpha \! \in \! \Delta_0 \, $,  $ \, \beta, \gamma \! \in \! \Delta_1 \, $,  and  $ X_\alpha \, $,  $ X_\beta \, $,  $ X_\gamma $  as above.  We define the supergroup functors  $ x_\alpha \, $,  $ x_\beta \, $  and  $ x_\gamma $  from  $ \salg $  to  $ \grps $  as
  $$  \begin{array}{rl}
   x_\alpha(A) \!  &  := \,  \big\{ \exp\!\big( t \, X_\alpha \big) \;\big|
\; t \in A_0 \,\big\} \; =\big\{ \big( 1 + t \, X_\alpha + t^2 \, X_\alpha^{\,(2)} + \cdots \big) \;\big| \; t \in A_0 \,\big\}  \\
  \\
   x_\beta(A) \!  &  := \,  \big\{ \exp\!\big( \vartheta \, X_\beta \big)
\;\big|\; \vartheta \in A_1 \,\big\}  \; =\big\{ \big( 1 +  \vartheta \, X_\beta \big) \;\big|\; \vartheta \in A_1 \,\big\}  \\
  \\
   x_\gamma(A) \!  &  := \,  \big\{ \exp\!\big( \vartheta \, X_\gamma + t \, X_\gamma^{\,2} \big) \;\big|\; \vartheta \in A_1 \, , \, t \in A_0 \,\big\}  \, =  \\
  \phantom{\bigg|}  &  \phantom{:}= \,  \big\{ \big( 1+\vartheta \, X_\gamma\big) \exp\!\big( t \, X_\gamma^{\,2} \big) \;\big|\; \vartheta \in A_1 \, , \, t \in A_0 \,\big\}
      \end{array}  $$
(notice that  $ \, x_\alpha(A) \, $,  $ \, x_\beta(A) \, $,  $ \, x_\gamma(A) \subseteq \rGL\big(V_\bk(A)\big) \, $,  \, by construction).
 \vskip4pt
   {\it (b)} \,  Let  $ \, H \in \fh_\Z \, $.  We define the supergroup functor  $ h_H $  (also referred to as a ``multiplicative one-parameter supersubgroup'')  from  $ \salg $  to  $ \grps $
  $$  h_H(A)  \; := \;  \big\{\, t^H := h_H(t) \;\big|\; t \in \! U(A_0) \big\}
\subset \rGL\big(V_\bk(A)\big)  $$
We also write  $ \; h_i := h_{H_i} \, $  for  $ \, i = 1, \dots, \ell \, $,  and  $ \; h_\alpha := h_{H_\alpha} \; $  for  $ \, \alpha \in \Delta \, $.

\medskip

   Note that, as the  $ H_i $'s  form a  $ \Z $--basis  of  $ \fh_\Z \, $,  the subgroup of  $ \rGL\big(V_\bk(A)\big) $  generated by all the  $ h_H(t) $'s   ---  $ \, h \in \fh_\Z \, $,  $ \, t \in U(A_0) \, $  ---   is in fact generated by the  $ h_i(t) $'s   ---  $ \, i = 1, \dots, \ell \, $,  $ \, t \in U(A_0) \, $.
\end{definition}

\vskip13pt

\begin{notation} \label{oneparameter}
 By a slight abuse of language we will also write
  $$  x_\alpha(t) \! := \exp\!\big( t X_\alpha \big) \, , \;\;\;
      x_\beta(\vartheta) \! := \exp\!\big( \vartheta X_\beta \big) \, ,  \;\;\;
      x_\gamma(t,\vartheta) \! := \exp\!\big( \vartheta X_\gamma + t X_\gamma^{\,2} \big)  $$
Moreover, to unify the notation,  $ \, x_\delta(\bt) \, $  will denote, for  $ \, \delta \in
\Delta \, $,  any one of the three possibilities above, so that  $ \bt $  denotes a pair  $ \, (t,\theta) \in A_0 \times A_1 \, $,  with  $ \vartheta $  or  $ t $  equal to zero   --- hence dropped ---   when either  $ \, \delta \in \Delta_0 \, $,  or  $ \, \delta \in \Delta_1 \, $  with  $ \, \big[ X_\delta \, , X_\delta \big] = 0 \, $,  \, i.e.~$ \, 2 \, \delta \not\in \Delta \, $.  Finally, for later convenience we shall also formally write  $ \, x_\zeta(\bt) := 1 \, $  when  $ \zeta  $  belongs to the  $ \Z $--span  of  $ \Delta $  but  $ \, \zeta \not\in \Delta \, $.
\end{notation}

\bigskip

   Definition \ref{chevalley}  is modeled in analogy with the Lie supergroup setting (see Section \ref{1-param-sgrp}):  this yields the first half of

\bigskip

\begin{proposition} \label{hopf-alg} {\ }
                                                                      \par
   {\it (a)} \,  The supergroup functors  $ x_\alpha \, $,  $ x_\beta \, $  and  $ x_\gamma $  in  Definition \ref{chevalley}{\it (a)\/}  are representable, and they are  {\sl affine}  supergroups.  Indeed, for each  $ \, A \in \salg \, $,
  $$  \begin{matrix}
   x_\alpha(A) \, = \, \Hom\big(\,\bk[x],A\big) \; ,  &  \varDelta_\alpha(x) \, = \, x \otimes 1 + 1 \otimes x  \\
 \hskip-7pt \phantom{\bigg|}
   x_\beta(A) \, = \, \Hom\big(\,\bk[\xi],A\big) \; ,  &  \varDelta_\beta(\xi) \, = \, \xi \otimes 1 + 1 \otimes \xi  \\
   x_\gamma(A) \, = \, \Hom\big(\,\bk[x,\xi],A\big) \; ,  &  \quad  \varDelta_{\,\gamma}(x) \, = \, x \otimes 1 + 1 \otimes x - \xi \otimes \xi \; ,  \\
                                                      &  \varDelta_{\,\gamma}(\xi) \, = \, \xi \otimes 1 + 1 \otimes \xi
      \end{matrix}  $$
where  $ \, \varDelta_{\,\delta} \, $  denotes the comultiplication in the Hopf superalgebra of the one-parameter subgroup corresponding to the root  $ \, \delta \in \Delta \, $.
 \vskip5pt
   {\it (b)} \,  Every supergroup functor  $ h_H $  in  Definition \ref{chevalley}{\it (b)\/}  is representable, and it is an  {\sl affine}  supergroup.  More precisely,

  $$  h_H(A) \, = \, \Hom\big(\,\bk\big[z,z^{-1}\big],A\big) \;\; ,  \qquad
 \varDelta\big(z^{\pm 1}\big) \, = \, z^{\pm 1} \otimes z^{\pm 1}  \;\; .  $$
\end{proposition}

\bigskip

   Before we define the Chevalley supergroups, we give the definition of the  {\it reductive group\/}  $ \bG_0 $  associated to a given classical Lie superalgebra  $ \fg \, $.

\smallskip

   First of all, note that, by construction, any  $ h_H(A) $  and any  $ x_\alpha(A) $  for  $ \, \alpha \in \Delta_0 \, $  depends only on  $ A_0 \, $:  thus, these are indeed  {\sl classical\/}  affine groups (in other words, they are ``super'' only in a formal sense).
 Moreover, the  $ h_H(A) $'s  generate a  {\sl classical torus\/}  $ T(A_0) $  inside  $ \rGL\big(V_\bk(A_0)\big) $
  $$  T(A_0)  \, := \,  \big\langle\, h_H(A) \;\big|\, H \! \in \! \fh_\Z \,\big\rangle  \, = \,  \big\langle\, h_i(A) \;\big|\; i \! = \! 1,\dots,\ell \,\big\rangle.  $$
   \indent   Now,  $ \fg_0 $  is a reductive Lie algebra (semisimple iff  $ \fg $  has no direct summands of type  $ A \, $),  whose Cartan subalgebra is  $ \fh \, $.  The Chevalley basis of  $ \fg $  contains a basis of  $ \fg_0 $  which has all properties of a classical Chevalley basis for  $ \fg_0 $,  except for the fact
that the  $ H_i $'s  are associated to integral weights which
{\sl may not be even\/}  (simple) roots.
In other words a classical Chevalley basis for  $ \fg_0 $  is not a subset
of a Chevalley basis for  $ \fg \, $.
%
%
 In general one can
show that we can always choose a basis for  $ \fh $  in such a way that only
one of the  $ H_i $'s  is associated to an odd root.  It is however
important to stress that the integral lattice generated by the
elements  $ \, H_1, \dots, H_l \, $  inside  $ \, \fh \, $  is strictly larger
than the lattice generated inside the Cartan by just the even elements,
and this despite the fact that in most cases  $ \, \fh = \fh_0 \, $.
                                            \par
   By construction the stabilizer of  $ V $  in  $ \fg_0 $  is  $ \; {(\fg_0)}_V = \fh_V \bigoplus \big(\! \oplus_{\alpha \in \Delta_0} \Z X_\alpha \,\big) \; $  with  $ \fh_V $  as in  Theorem \ref{stabilizer} and so we  can mimic the classical construction of
Chevalley proceeding in the following way.

\smallskip

   Consider the group functor  $ \; G_0 : \text{(alg)} \relbar\joinrel\longrightarrow \grps \; $
--- where  $ \text{(alg)} $  is the category of commutative
$ \bk $--algebras  ---   with $ \, G_0(A_0) \, $,
for  $ \, A_0 \in \text{(alg)} \, $,  being the subgroup of
$ \rGL\big(V_\bk(A_0)\big) $  generated by the torus  $ T(A_0) $
%
%
 and the  $ x_\alpha(A) $'s  with  $ \, \alpha \in \Delta_0 \, $.  By the definition of  $ T(A_0) \, $,  we can also say that  $ G_0(A_0) $  is generated by the  $ h_i(A) $'s  and the  $ x_\alpha(A) $'s  with  $ \, \alpha \in \Delta_0 \, $,  \, i.e.
  $$  G_0(A_0)  \; := \;  \Big\langle T(A_0) \, , \, x_\alpha(A) \;\Big|\; \alpha \! \in \! \Delta_0 \Big\rangle  \; = \;  \Big\langle\, h_i(A) \, , \, x_\alpha(A) \;\Big|\; i \! = \! 1,\dots,\ell ; \, \alpha \! \in \! \Delta_0 \Big\rangle  $$
   \indent   By construction, the group-functors  $ G_0 $
and  $ T $ are subfunctors of the representable group functor
$ \rGL(V_\bk) \, $,  hence they both are presheaves (see Appendix A).
Let  $ \bG_0 $  and  $ \bT $  be their sheafification (see Appendix A).
Then  $ \bT $  is representable and we shall now show that also  $ \bG_0 $  is
representable.
                                                       \par
   We consider  $ G_0 $,  $ \bG_0 $,  $ T_0 $  and  $ \bT_0 $  as group-functors defined on  $ \salg $  which factor through  $ \text{(alg)} $,  setting  $ \, G_0(A) := G_0(A_0) \, $,  and so on.
                                                                         \par
   Consider $ \, \fh^0_\Z := \text{\it Span}_{\,\Z} \big( \big\{ H_\alpha \,\big|\, \alpha \! \in \! \Delta_0 \big\}\big) \, $;  this is another  $ \Z $--form  of  $ \fh \, $,  with  $ \, \fh^0_\Z \subseteq \fh_\Z \, $.  Now define
  $$  T'(A_0)  \; := \;  \Big\langle\, h_H(A) \;\Big|\, H \! \in \! \fh^0_\Z \,\Big\rangle  \;\; ,  \quad
   G'_0(A_0)  \; := \;  \Big\langle\, T'(A_0) \, , \, x_\alpha(A) \;\Big|\; \alpha \! \in \! \Delta_0 \Big\rangle  $$
The assignments  $ \, A \mapsto T'(A) := T'(A_0) \, $  and  $ \, A \mapsto G'_0(A) := G'_0(A_0) \, $  provide new group-functors defined on  $ \salg $,  which clearly factor through  $ \text{(alg)} $,  like  $ T $  and  $ G_0 $  above; also, they are presheaves too.  Then we define the functors  $ \bT' $  and  $ \bG'_0 $  as the sheafifications of  $ T' $  and  $ G'_0 $  respectively.
                                                                         \par
   On  {\sl local\/}  algebras   --- in  $ \text{(alg)} $  ---   the functor  $ G'_0 $  is isomorphic via a natural transformation with the functor of points of the Chevalley group-scheme associated with  $ \fg_0 $  and  $ V \, $.  Therefore, we have that  $ \bG'_0 $  {\sl is representable}.
 \vskip3pt
   The group  $ G'_0(A_0) $  and  $ T(A_0) $  are subgroups of  $ \rGL(V_\bk)(A_0) \, $,  whose mutual intersection is  $ T'(A_0) \, $.  The subgroup  $ G_0(A_0) \, $,  generated by  $ G'_0(A_0) $  and  $ T(A_0) $ inside  $ \rGL(V_\bk)(A_0) \, $,  can be seen as the fibered coproduct of  $ G'_0(A_0) $  and  $ T(A_0) $  over  $ T'(A_0) \, $.  In more down-to-earth terms, we can describe it as follows.  Inside  $ \rGL(V_\bk)(A_0) \, $,  the subgroup  $ T(A_0) $  acts by adjoint action over  $ G'_0(A_0) \, $,  hence the subgroup  $ G_0(A_0) \, $,  being generated by  $ T(A_0) $  and  $ G'_0(A_0) \, $,  is a quotient of the semi-direct product  $ \, T(A_0) \ltimes G'_0(A_0) \, $.  To be precise,
we have the functorial isomorphism:
  $$  G_0(A_0)  \,\; \cong \;\,  \big( T(A_0) \ltimes G'_0(A_0) \big) \! \Big/ K(A_0)  $$
where the assignment
  $$  A  \; \mapsto \;  K(A_0) \, := \, \big\{ \big(t,t^{-1} \big) \,\big|\, t \in T'(A_0) \,\big\}  $$
(on objects) defines   --- on  $ \salg $,  through  $ \text{(alg)} $  ---   a subgroup(-functor) of  $ \; T \ltimes G'_0 \; $.  Therefore  $ \; G_0 \cong \big( T \ltimes G'_0 \big) \! \Big/ K \; $  as group-functors, hence  $ \; G_0 \cong \big( T \times G'_0 \big) \! \Big/ K \; $  because  $ \; T \ltimes G'_0 \, \cong \, T \times G'_0 \; $  as set-valued functors (i.e., forgetting the group structure).  Taking sheafifications, we get  $ \; \bG_0 \cong \big( \bT \times \bG'_0 \big) \! \Big/ \mathbf{K} \; $:  \, but since both  $ \bT $  and  $ \bG'_0 $  are representable, the direct product  $ \, \bT \times \bG'_0 \, $  is representable too, hence its quotient  $ \; \big( \bT \times \bG'_0 \big) \Big/ \mathbf{K} \;\; \big(\! \cong \, \bG_0 \big) \; $  is representable as well,  that is   --- eventually ---   $  \bG_0 $  {\sl is representable},  q.e.d.

\bigskip

   We now define the Chevalley supergroups in a superscheme-theoreti\-cal way, through sheafification of a suitable functor from  $ \salg $  to  $ \text{(grps)} \, $,  ``generated''  inside  $ \rGL\big(V_\bk(A)\big) $  by the one-parameter supersubgroups given above.

\medskip

\begin{definition} \label{def_Che-sgroup_funct} {\ }
 Let  $ \fg $  and  $ V $  be as above.  We call  {\it Chevalley supergroup functor},  associated to  $ \fg $  and  $ V \, $,  the functor  $ \; G : \salg \lra \text{(grps)} \; $  given by:
 \vskip5pt
 \noindent \quad   --- if  $ \, A \! \in \! \text{\it Ob}\big(\salg\big) \, $  we let  $ \, G(A) \, $  be the subgroup of  $ \rGL\big(V_\bk(A)\big) $  generated by  $ G_0(A) $  and the one-parameter subgroups  $ x_\beta(A) $  with  $ \, \beta \in \Delta_1 \, $,  that is
  $$  G(A)  \; := \;  \Big\langle\, G_0(A) \, , \, x_\beta(A) \,\;\Big|\;\, \beta \in \Delta_1 \, \Big\rangle  $$
By the previous description of  $ G_0 \, $,  we see that  $ G(A) $  is also generated by the  $ h_i(A) $'s  and {\sl all\/}  the one-parameter subgroups  $ x_\delta(A) $'s   ---  $ \, \delta \in \Delta $  ---   or even by  $ T(A) $  and the  $ x_\delta(A) $'s,  that is
  $$  G(A)  \; := \;  \Big\langle h_i(A) \, , \, x_\delta(A) \;\Big|\; i = 1, \dots, \ell \, , \; \delta \in \Delta \Big\rangle  \; = \;  \Big\langle T(A) \, , \, x_\delta(A) \;\Big|\; \delta \in \Delta \Big\rangle  $$
 \vskip5pt
 \noindent \quad   --- if  $ \, \phi \in \text{Hom}_{\salg}\big(A\,,B\big) $,  then  $ \; \End_\bk(\phi) : \End_\bk\big(V_\bk(A)\big) \! \lra \End_\bk\big(V_\bk(B)\big) \,$  (given on matrix entries by  $ \phi $  itself)  respects the sum and the associative product of matrices.  Then  $ \End_\bk( \phi) $  clearly restricts to a group morphism  $ \, \rGL\big(V_\bk(A)\big) \lra \rGL\big(V_\bk(B)\big) \, $.
 The latter maps the generators of  $ G(A) $  to those of  $ G(B) $,  hence restricts to a group morphism  $ \; G(\phi) : G(A) \lra G(B) \; $.
 \vskip5pt
   We call {\it Chevalley supergroup\/}  the sheafification  $ \bG $  of $ G \, $.  By Appendix A, Theorem \ref{sheaf-iso},  $ \; \bG : \salg \lra \text{(grps)} \; $  is a functor such that  $ \, \bG(A) = G(A) \, $  when  $ \, A \in \salg \, $  is  {\sl local\/}  (i.e., it has a unique maximal homogeneous ideal).
\end{definition}

\medskip

\begin{remark}
   The sheafification is already necessary at the classical level, that is when we construct semisimple algebraic groups (from semisimple Lie algebras), as it is explained in  \cite{sga},  \S 5.7.  In fact, it is clearly stated in 5.7.6 that in general the one-parameter subgroups and the torus generate the algebraic group only over  {\sl local algebras}.
\end{remark}

\bigskip
\bigskip
\bigskip

 \section{Chevalley supergroups as algebraic supergroups}  \label{che-sgroup_alg}

\bigskip

   The way we defined the Chevalley supergroup  $ \bG $  does not imply   ---
at least not in an obvious way ---   that  $ \bG $  is  {\it representable},
in other words, that  $ \bG $  is the functor of points of an
algebraic supergroup scheme.  The aim of this section is to
prove this important property.

\medskip

   We shall start by studying the commutation relations of the generators
and derive a decomposition formula for  $ G(A) $  resembling the classical
Cartan decomposition in the Lie theory, and one reminding the classical
``big cell'' decomposition in the theory of reductive algebraic groups.

\medskip

   We begin with some more definitions:

\bigskip

\begin{definition}  \label{subgrps}
 For any  $ \, A \in \salg \, $,  we define the subsets of  $ G(A) $
  $$  \begin{array}{rl}
   G_1(A)  &  \! := \;  \left\{\, {\textstyle \prod_{\,i=1}^{\,n}} \,
x_{\gamma_i}(\vartheta_i) \;\Big|\; n \in \N \, , \; \gamma_1, \dots, \gamma_n \in \Delta_1 \, , \; \vartheta_1, \dots, \vartheta_n \in A_1 \,\right\}  \\
  \\
   G_0^\pm(A)  &  \! := \;  \left\{\, {\textstyle \prod_{\,i=1}^{\,n}}
\, x_{\alpha_i}(t_i) \;\Big|\; n \in \N \, , \; \alpha_1, \dots, \alpha_n \in \Delta^\pm_0 \, , \; t_1, \dots, t_n \in A_0 \,\right\}  \\
  \\
   G_1^\pm(A)  &  \! := \;  \left\{\, {\textstyle \prod_{\,i=1}^{\,n}} \,
x_{\gamma_i}(\vartheta_i) \;\Big|\; n \in \N \, , \; \gamma_1, \dots,
\gamma_n \in \Delta^\pm_1 \, , \; \vartheta_1, \dots, \vartheta_n \in A_1
\,\right\}  \\
   \\
   G^\pm(A)  &  \! := \;  \left\{\, {\textstyle \prod_{\,i=1}^{\,n}} \,
x_{\beta_i}(\bt_i) \;\Big|\; n \! \in \! \N \, , \, \beta_1, \dots, \beta_n \! \in \! \Delta^\pm , \, \bt_1, \dots, \bt_n \! \in \! A_0 \! \times \! A_1 \right\}
=  \\
   \phantom{\Bigg|}  &  \hfill   \, = \; \big\langle G_0^\pm(A) \, , \, G_1^\pm(A) \big\rangle
      \end{array}  $$
   \indent   Moreover, fixing any total order  $ \, \preceq \, $  on $ {\wDelta}_1^\pm \, $,  and letting  $ \, N_\pm = \big| {\wDelta}_1^\pm \big| \, $,  we set
  $$  G_1^{\pm,<}(A)  \,\; := \;  \left\{\; {\textstyle \prod_{\,i=1}^{\,N_\pm}} \, x_{\gamma_i}(\vartheta_i) \,\;\Big|\;\, \gamma_1 \prec \cdots \prec \gamma_{N_\pm} \in \Delta^\pm_1 \, , \; \vartheta_1, \dots, \vartheta_{N_\pm} \in A_1 \,\right\}  $$
and for any total order  $ \, \preceq \, $  on  $ {\wDelta}_1 \, $,  and letting  $ \, N := \big| \Delta \big| = N_+ + N_- \, $,  we set
  $$  G_1^<(A)  \,\; := \;  \left\{\, {\textstyle \prod_{\,i=1}^N} \, x_{\gamma_i}(\vartheta_i) \,\;\Big|\;\, \gamma_1 \prec \cdots \prec \gamma_{\scriptscriptstyle N} \in \Delta_1 \, , \; \vartheta_1, \dots, \vartheta_N \in A_1 \,\right\}  $$
   \indent   Note that for special choices of the order, one has  $ \; G_1^<(A) = G_1^{-,<}(A) \cdot G_1^{+,<}(A) $  \; or  $ \; G_1^<(A) = G_1^{+,<}(A) \cdot G_1^{-,<}(A) \;\, $.

\smallskip

   Similar notations will denote the sheafifications  $ \bG_1 \, $,
$ \bG^\pm $,  $ \bG_0^\pm $,  $ \bG_1^\pm $,  etc.
\end{definition}

\vskip9pt

\begin{remark}
 Note that  $ G_1(A) $,  $ G_0^\pm(A) $,  $ G_1^\pm(A) $  and  $ G^\pm(A) $  are subgroups of  $ G(A) \, $,  while  $ G_1^{\pm,<}(A) $  and  $ G_1^<(A) $  instead are  {\sl not},  in general.  And similarly  with  ``$ \, \bG \, $''  instead of  ``$ \, G \, $''  everywhere.
\end{remark}

\vskip7pt

   For the rest of our analysis, next result   --- whose proof is straightforward ---   will be crucial. Hereafter, as a matter of notation, when  $ \varGamma $  is any (abstract) group and  $ \, g, h \in \varGamma \, $  we denote by  $ \, (g,h) := g \, h \, g^{-1} \, h^{-1} \, $  their commutator in  $ \varGamma \, $.

\bigskip

\begin{lemma}  \label{comm_1-pssg} {\ }
 \vskip7pt
   (a) \,  Let  $ \; \alpha \in \Delta_0 \, $,  $ \, \gamma \in \Delta_1 \, $,  $ \, A \in \salg \, $  and  $ \, t \in A_0 \, $,  $ \, \vartheta \in A_1 \, $.  Then there exist  $ \, c_s \! \in \! \Z \, $  such that
  $$  \big( x_\gamma(\vartheta) \, , \, x_\alpha(t) \big)  \, = \,  {\textstyle \prod_{s>0}} \, x_{\gamma + s \, \alpha}\big(c_s \, t^s \vartheta\big)  \;\; \in \;\;  G_1(A)  $$
(the product being finite).  More precisely, with  $ \, \varepsilon_k = \pm 1 \, $  and  $ \, r \, $  as in (4.12),
  $$  \big( 1 + \vartheta \, X_\gamma \, , \, x_\alpha(t) \big)  \, = \,
{\textstyle \prod_{s>0}} \, \Big( 1 + {\textstyle \prod_{k=1}^s \varepsilon_k \cdot {{s+r} \choose r}} \cdot t^s \vartheta \, X_{\gamma + s \, \alpha} \Big)  $$
where the factors in the product are taken in any order (as they do commute).
 \vskip7pt
   (b)  Let  $ \; \gamma , \delta \! \in \! \Delta_1 \, $,  $ \, A \! \in \! \salg \, $,  $ \, \vartheta, \eta \! \in \! A_1 \, $.  Then (notation of  Definition \ref{def_che-bas})
  $$  \big( x_\gamma(\vartheta) \, , \, x_\delta(\eta) \big)  \; = \;\,  x_{\gamma + \delta}\big(\! - \! c_{\gamma,\delta} \; \vartheta \, \eta \big)  \; = \;  \big(\, 1 \! - \! c_{\gamma,\delta} \; \vartheta \, \eta \, X_{\gamma + \delta} \,\big)  \;\; \in \;\;  G_0(A)  $$
if  $ \; \delta \not= -\gamma \; $;  otherwise, for  $ \; \delta = -\gamma \, $,  we have
  $$  \big( x_\gamma(\vartheta) \, , \, x_{-\gamma}(\eta) \big)  \; = \;  \big(\, 1 \! - \vartheta \, \eta \, H_\gamma \,\big)  \; = \;  h_\gamma\big( 1 \! - \vartheta \, \eta \big)  \;\; \in \;\;  G_0(A)  $$
 \vskip7pt
   (c)  Let  $ \; \alpha , \beta \in \Delta \, $,  $ \, A \in \salg \, $,  $ \, t \in U(A_0) \, $,  $ \, \bu \in A_0 \! \times \! A_1 = A \, $.  Then
  $$  \hskip31pt   h_\alpha(t) \; x_\beta(\bu) \; {h_\alpha(t)}^{-1}  \; = \;  x_\beta\big( t^{\beta(H_\alpha)} \, \bu \big)  \;\; \in \;\;  G_{p(\beta)}(A)  $$
where  $ \; p(\beta) := s \, $,  by definition, if and only if  $ \, \beta \in \Delta_s \; $.
\end{lemma}

\begin{proof} The result follows directly from the classical results in  \cite{st},  pg.~22 and 29, and simple calculations, using the relations in  \S \ref{comm-rul_kost}  and the identity  $ \, {(\vartheta \eta)}^2 = - \vartheta^2 \, \eta^2 = 0 \, $.  In particular, like in the classical setting, every  $ h_\alpha(t) $  acts   --- via the adjoint representation ---   diagonally on each root vector  $ X_\beta \, $,  with weight  $ t^{\beta(H_\alpha)} \, $.  Also, we point out that the nilpotency of any  $ \, \vartheta, \eta \in A_1 \, $  implies that of  $ \, \vartheta \, \eta \, \big(\! \in \! A_0 \big) $,  which has several consequences.
                                                                            \par
   The first consequence to stress in that  $ \; \big( 1 \! - \vartheta \, \eta \big) \, , \, \big( 1 \! - 2 \, \vartheta \, \eta \big) \in U(A_0) \, $.  The second is that we have the identity  $ \; \big(\, 1 \! - \vartheta \, \eta \, H_\gamma \,\big) \, = \, h_\gamma\big( 1 \! - \vartheta \, \eta \big) \; $   --- to be thought of as an identity among operators on  $ V_\bk(A) \, $  ---   which follows from  $ \, {\big( 1 \! - \vartheta \, \eta \big)}^{\mu(H_\gamma)} \! = 1 \! - \vartheta \, \eta \; $:  \, this is mentioned (and used) in the second instance of  {\it (b)}.
\end{proof}

\medskip

\begin{remark}  \label{som-theta}
 A direct consequence of the previous Lemma is the following.
Assume  $ \, g_j \in x_{\delta_j}\big(A_1^{\,n_j}\big) :=
\big\{ x_{\delta_j}(\bu) \,\big|\, \bu \! \in \! A_1^{\,n_j} \big\} \, $
--- cf.~\S \ref{first_preliminaries}  ---   for  $ \, j = 1, 2 \, $
and  $ \delta_1 $  in  $ \Delta_1 \, $.
Then we have  $ \; \big(\, g_1 \, , \, g_2 \big) \in \prod_{s>0} x_{\delta_1 + s \, \delta_2}\big(A_1^{n_1 + s \, n_2}\big) \; $  if  $ \, \delta_2 \in \Delta_0 \, $  and  $ \; \big(\, g_1 \, , \, g_2 \big) \in x_{\delta_1 + \delta_2}\big(A_1^{n_1 + n_2}\big) \; $  or  $ \; \big(\, g_1 \, , \, g_2 \big) \in T\big(A_1^{(n_1 + n_2)}\big) \; $  if  $ \, \delta_2 \in \Delta_1 \, $.
\end{remark}

\bigskip

   Next result is a group-theoretical counterpart of the splitting  $ \, \fg = \fg_0 \oplus \fg_1 \, $.  It is a super-analogue of the classical Cartan decomposition (for reductive groups).  In the differential setting, it was   --- somewhat differently ---   first pointed out by Berezin  (see  \cite{be},  Ch.~2, \S 2).

\bigskip

\begin{theorem} \label{g0g1}
 Let  $ \, A \in \salg \, $.  There exist set-theoretic factorizations
  $$  \begin{matrix}
    G(A)  \; = \;  G_0(A) \; G_1(A) \quad ,  \phantom{\Big|}  &  \qquad  G(A)  \; = \;  G_1(A) \; G_0(A)  \\
    G^\pm(A)  \; = \;  G^\pm_0(A) \; G^\pm_1(A) \quad ,  \phantom{\Big|}  &  \qquad  G^\pm(A)  \; = \;  G^\pm_1(A) \; G^\pm_0(A)
      \end{matrix}  $$
\end{theorem}

\begin{proof}  The proof for  $ G(A) $  works for  $ G^\pm(A) $  as well, so we stick to the former.
                                                 \par
   It is enough to prove either one of the equalities, say the first one.  Also, it is enough to show that  $ \, G_0(A) \, G_1(A) \, $ is closed by multiplication, since it contains all generators of  $ G(A) $  and their inverses.  So we have to show that
  $$  g_0 g_1 \cdot g_0' g_1'  \, \in \,  G_0(A) \, G_1(A)   \eqno  \forall\;\;  g_0 , g_0' \in G_0(A) \, , \; g_1 , g_1' \in G_1(A) \;\;\, .   \qquad  $$
   \indent   By the very definitions, we need only to prove that
  $$  \displaylines{
   \big( 1 + \vartheta_1 \, X_{{\wbeta}_1} \big) \cdots \big( 1 + \vartheta_{n-1} \, X_{{\wbeta}_{n-1}} \big) \, \big( 1 + \vartheta_n \, X_{{\wbeta}_n} \big) \; x_{\walpha}(u) \,\; \in \; G_0(A) \, G_1(A)  \cr
   \big( 1 + \vartheta_1 \, X_{\beta_1} \big) \cdots \big( 1 + \vartheta_{n-1} \, X_{{\wbeta}_{n-1}} \big) \, \big( 1 + \vartheta_n \, X_{{\wbeta}_n} \big) \; h_{\delta}(t) \,\; \in \, G_0(A) \; G_1(A)  }  $$
for all  $ \, {\wbeta}_1 , \dots , {\wbeta}_n \! \in {\wDelta}_1 \, $,  $ \, \walpha \! \in {\wDelta}_0 \, $,  $ \, \delta \! \in \Delta \, $,  $ \, \vartheta_1, \dots, \vartheta_n \! \in \! A_1 \, $,  $ \, u \! \in \! A_0 \, $,  $ \, t \! \in U(A_0) \, $.  But this comes from an easy induction on  $ n \, $,  via the formulas in  Lemma \ref{comm_1-pssg}.
\end{proof}

\medskip

\begin{lemma} \label{precrucial}
   Let  $ \, A \in \salg \, $.  Then
  $$  \displaylines{
   G_1(A)  \; \subseteq \;  G_0\big(A_1^{(2)}\big) \, G_1^<(A)  \quad ,  \qquad  G_1(A) \; \subseteq \;  G_1^<(A) \, G_0\big(A_1^{(2)}\big)  \phantom{\Big|}  \cr
   G_1^\pm(A)  \; \subseteq \;  G_0^\pm\big(A_1^{(2)}\big) \, G_1^{\pm,<}(A)  \quad ,  \qquad  G_1^\pm(A) \; \subseteq \;  G_1^{\pm,<}(A) \, G_0^\pm\big(A_1^{(2)}\big)  \phantom{\Big|}  }  $$
\end{lemma}

\begin{proof}
 We deal with the first identity, the other ones are similar.  Indeed, we shall prove the slightly stronger result
  $$  \big\langle G_1(A) \, , \, G_\bullet \! \big(A_1^{(2)}\big) \big\rangle  \; \subseteq \;  G_0\big(A_1^{(2)}\big) \, G_1^<(A)   \eqno (5.3)  $$
where  $ \, \big\langle G_1(A) \, , \, G_\bullet \! \big(A_1^{(2)}\big) \big\rangle \, $  is the subgroup generated by  $ G_1(A) $  and  $ G_\bullet\!\big(A_1^{(2)}\big) \, $,  where the latter in turns denotes the subgroup
  $$  G_\bullet\!\big(A_1^{(2)}\big)  \; := \;  \Big\langle \big\{ h_\alpha(u), x_\alpha(t) \,\big|\, \alpha \in \Delta_0, u \in U\!\big(A_1^{(2)}\big), t \in A_1^{\,2} \,\big\} \Big\rangle  $$
   \indent   Any element of  $ \big\langle G_1(A) \, , \, G_\bullet \! \big(A_1^{(2)}\big) \big\rangle $  is a product  $ \; g \, = \, g_1 \, g_2 \cdots g_k \; $  in which each factor  $ g_i $  is either of type  $ \, h_{\alpha_i}(u_i) \, $,  or  $ \, x_{\alpha_i}(t_i) \, $,  or  $ \, x_{\gamma_i}(\vartheta_i) \, $,  with  $ \, \alpha_i \in \Delta_0 \, $, $ \, \gamma_i \in \Delta_1 \, $  and  $ \, u_i \in U\big(A_1^{(2)}\big) \, $,  $ \, t_i \in A_1^2 \, $,  $ \, \vartheta_i \in A_1 \, $.  Such a product belongs to  $ \, G_0\big(A_1^{(2)}\big) \, G_1^<(A) \, $  if all factors indexed by the $ \, \alpha_i \in \Delta_0 \, $  are on the left of those indexed by the $ \, \gamma_j \, \in \Delta_1 \, $, and moreover the latter occur in the order prescribed by  $ \, \preceq \, $.  In this case, we say that the factors of  $ g $  are ordered.  We shall now re-write  $ g $  as a product of ordered factors, by repeatedly commuting the original factors, as well as new factors which come in along this process.
 \vskip5pt
   Since we have only a finite number of odd coefficients in the expression for  $ g \, $,  we can assume without loss of generality that  $ A_1 $  is finitely generated as an $ A_0 $--module.  This implies that  $ \, A_1^{\,n} = \{0\} \, $  and  $ \, A_1^{\,(n)} = 0 \, $  for  $ n $  larger than the
number of odd generators of  $ A_1 \, $.
 \vskip7pt
   Let us consider two consecutive factors  $ \; g_i \, g_{i+1} \; $  in  $ g \, $.  If they are already ordered, we are done.  Otherwise, there are four possibilities:
 \vskip5pt
   {\it --- (1)} \,  $ \, g_i = x_{\alpha_i}(t_i) \, $,  $ \, g_{i+1} = h_{\alpha_i}(u_i) \, $,  \; or
$ \; g_i = x_{\gamma_i}(\vartheta_i) \, $,  $ \, g_{i+1} = h_{\alpha_i}(u_i) \, $.  \quad  In this case we rewrite
  $$  \displaylines{
   g_i \, g_{i+1}  \; = \;  x_{\alpha_i}(t_i) \, h_{\alpha_i}(u_i)  \;
= \; h_{\alpha_i}(u_i) \, x_{\alpha_i}(t'_i)  \cr
 \hbox{or}   \hfill
   g_i \, g_{i+1}  \; = \;  x_{\gamma_i}(\vartheta_i) \, h_{\alpha_i}(u_i)  \;
= \; h_{\alpha_i}(u_i) \, x_{\gamma_i}(\vartheta'_i)  \hfill  }  $$
%
%
%
 with  $ \; t'_i \in A_1^{\,n_i} \, $,  $ \, \vartheta'_i \in \wA_1^{\,m_i} \, $,  if  $ \; t_i \in A_1^{\,n_i} \, $,  $ \, \vartheta_i \in A_1^{\,m_i} \, $,  thanks to  Lemma \ref{comm_1-pssg}{\it (c)}.
In particular we replace a pair of unordered factors with a new pair of ordered factors.  Even more, this shows that any factor of type  $ h_{\alpha_i}\!(u_i) $  can be flushed to the left of our product so to give a new product of the same nature, but with all factors of type $h_{\alpha_i}(u_i)$
on the left-hand side.
 \vskip5pt
   {\it --- (2)} \,  $ \, g_i = x_{\alpha_i}(t_i) \, $,  $ \, g_{i+1} = x_{\gamma_{i+1}}(\vartheta_{i+1}) \, $.
\quad  In this case we rewrite
  $$  g_i \, g_{i+1}  \, = \,  g_{i+1} \, g_i \, g'_i  \;\quad  \text{with}  \;\quad  g'_i := \big(\, {g_i}^{\!-1} , \, {g_{i+1}}^{\!\!\!-1} \big) = \big(\, x_{\alpha_i}(-t_i) \, , \, x_{\gamma_{i+1}}(-\vartheta_{i+1}) \big)  $$
so we replace a pair of unordered (consecutive) factors with a pair of ordered (consecutive) factors followed by another, new factor  $ g'_i \, $.  {\sl However},  letting  $ \, n_1, n_2 \in \N_+ \, $  be such that  $ \, t_i \in A_1^{\,n_1} \, $,  $ \, \vartheta_{i+1} \in A_1^{\,n_2} \, $,  by  Remark \ref{som-theta}  this  $ \, g'_i \, $  {\sl is a product of new factors of type  $ \, x_{\alpha_j}(t'_j) \, $  with  $ \, t'_j \in A_1^{\;n_j} \, $,  $ \, n_j \geq n_i + n_{i+1} \, $}.
 \vskip5pt
   {\it --- (3)} \,  $ \, g_i = x_{\gamma_i}(\vartheta_i) \, $,  $ \, g_{i+1} = x_{\gamma_{i+1}}(\vartheta_{i+1}) \, $.  \quad  In this case we rewrite
  $$  g_i \, g_{i+1}  \, = \,  g_{i+1} \, g_i \, g'_i  \;\quad  \text{with}  \;\quad  g'_i := \big(\, {g_i}^{\!-1} , \, {g_{i+1}}^{\!\!\!-1} \big) = \big(\, x_{\gamma_i}(-\vartheta_i) \, , \, x_{\gamma_{i+1}}(-\vartheta_{i+1}) \big)  $$
so we replace a pair of unordered factors with a pair of ordered factors followed by a new factor  $ g'_i $  which  --- again by  Remark \ref{som-theta}  ---
 is again of type  $ \, x_\alpha(t) \, $  or  $ \, h_\alpha(u) \, $  with
$ \, t \in \wA_1^{\;n} \, $,  $ \, u \in U\big(A_1^{\,(n)}\big) \, $,
%
%
where  $ \, n \geq n_i + n_{i+1} \, $  for  $ \, n_1, n_2 \in \N_+ \, $  such that  $ \, \vartheta_i \in A_1^{\,n_1} \, $  and  $ \, \vartheta_{i+1} \in A_1^{\,n_2} \, $.
 \vskip5pt
   {\it --- (4)} \,  $ \, g_i = x_{\gamma}(\vartheta_i) \, $,  $ \, g_{i+1} = x_{\gamma}(\vartheta_{i+1}) \; $.
%
%
 \quad  In this case we rewrite
  $$  g_i \, g_{i+1}  \; = \;  x_{\gamma_i}(\vartheta_i) \, x_{\gamma_{i+1}}(\vartheta_{i+1})  \; = \;
x_\gamma(\vartheta_i) \, x_\gamma(\vartheta_{i+1})  \; = \;  x_\gamma(\vartheta_i \! + \! \vartheta_{i+1})  $$
so we replace a pair of unordered factors with a single factor.  In addition, each of the pairs  $ \, g_{i-1} \, g'_i \, $  and  $ \, g'_i \, g_{i+2} \, $  respects or violates the ordering according to what the corresponding
   \hbox{old pair  $ \, g_{i-1} \, g_i \, $  and  $ \, g_{i+1} \, , g_{i+2} \, $  did.}
 \vskip7pt
   Now we iterate this process: whenever we have any unordered pair of
consecutive factors in the product we are working with,
we perform any one of steps  {\it (1)\/}  through  {\it (4)\/}  explained
above.  At each step, we substitute an unordered pair with a single factor
(step {\it (4)\/}), which does not form any more unordered pairs than
the ones we had before,  or with an ordered pair (steps {\it (1)--(4)\/}),
possibly introducing new additional factors.  However, any new factor is either
 of type  $ \, x_\alpha(t) \, $,  with  $ \, t \in \wA_1^{\;n} \, $, or
or of type  $ \, h_\alpha(u) \, $,  with  $ \, u \in U\big(\wA_1^{\,(n)}\big) \, $,
%
%
 for values of  $ n $ which are (overall) strictly increasing after
each iteration of this procedure.  As  $ \, \wA_1^{\,n} = \{0\} \, $,  for
$ \, n \gg 0 \, $,  after finitely many steps such new factors are trivial,
i.e.~eventually all unordered (consecutive) factors will commute with each
other and will be re-ordered without introducing any new factors.
Therefore the process stops after finitely many steps, thus proving (5.3).
\end{proof}

\medskip

\begin{theorem} \label{realcrucial}
 For any  $ \, A \in \salg \, $  we have
  $$  G(A)  \; = \;  G_0(A) \, G_1^<(A) \quad ,  \qquad
      G(A)  \; = \;  G_1^<(A) \, G_0(A)  $$
\end{theorem}

\begin{proof}
 This follows at once from  Theorem \ref{g0g1}  and  Lemma \ref{precrucial}.
\end{proof}

\vskip11pt

   Our aim is to show that the decompositions we proved in the previous
proposition are essentially unique.  We need one more lemma:

\medskip

\begin{lemma}  \label{red-to-classical}
  Let  $ \, A \, , \, B \in \salg \, $,  with $ B $  being a subsuperalgebra of  $ \, A \, $.  Then  $ \, G(B) \leq G(A) \, $,  \, i.e.~$ \, G(B) $  is a subgroup of  $ \, G(A) \, $.
\end{lemma}

\begin{proof}
 This is not in general true for any supergroup functor; however,  $ G $  by definition is a subgroup of some  $ GL(V) \, $,  hence the elements in  $ G(A) $  are realized
as matrices with coefficients in  $ A \, $,  and those in  $ G(B) $  as matrices
with coefficients in  $ B \, $.  It is then clear that any matrix in  $ G(B) $  is in  $ G(A) $,  and two such matrices are equal in  $ G(B) $  if and only if they are equal as matrices in  $ G(A) $  as well.
\end{proof}

\medskip

   We are ready for our main result:

\bigskip

\begin{theorem} \label{crucial}
 For any  $ \, A \in \salg \, $,  the group product yields bijections
  $$  G_0(A) \times G_1^{-,<}(A) \times G_1^{+,<}(A) \,
\lhook\joinrel\relbar\joinrel\relbar\joinrel\relbar\joinrel\twoheadrightarrow
\, G(A)  $$
and all the similar bijections obtained by permuting the factors  $ G_1^{\pm,<}(A) $  and/or switching the factor  $ G_0(A) $  to the right.
\end{theorem}

\begin{proof}  We shall prove the first mentioned bijection.  In general,  Proposition \ref{realcrucial}  gives  $ \; G(A) = G_0(A) \, G_1^<(A) \, $,  \, so the product map from  $ \; G_0(A) \times G_1^<(A) \; $  to  $ G(A) $  is onto; but in particular, we can choose an ordering on  $ \Delta_1 $  for which  $ \, \Delta_1^- \preceq \Delta_1^+ \, $,  hence  $ \, G_1^<(A) = G_1^{-,<}(A) \, G_1^{+,<}(A) \, $,  so we are done for surjectivity.
                                                             \par
   To prove that the product map is also injective amounts to showing that for any  $ \, g \in G(A) \, $,  the factorization  $ \; g = g_0 \, g_- \, g_+ \, $  with  $ \, g_0 \in G_0(A) \, $  and  $ \, g_\pm \in G_1^{\pm,<}(A) \, $ is unique.  In other words, if we have
  $$  g  \; = \;  g_0 \, g_- \, g_+  \; = \;  f_0 \, f_- \, f_+ \;\; ,   \eqno g_0 \, , \, f_0 \in G_0(A) \, ,  \;\; g_\pm \, , \, f_\pm \in G_1^{\pm,<}(A)  \quad  $$
we must show that $ \, g_0 = f_0 \, $  and  $ \, g_\pm = f_\pm \; $.

\smallskip

   To begin with, we write the last factors in our identities as
  $$  g_\pm  \; = \;  {\textstyle \prod_{d=1}^{N_\pm}} \, \big( 1 + \, \vartheta^\pm_d \, X_{\gamma^\pm_d} \big) \;\; ,  \qquad  f_\pm  \; = \;  {\textstyle \prod_{d=1}^{N_\pm}} \, \big( 1 + \, \eta^\pm_d \, X_{\gamma^\pm_d} \big)  $$
for some  $ \, t_i \, ,  s_j \in A_0 \, $,  and  $ \, \vartheta_j \, , \eta_j \in A_1 \, $,  with  $ \, N_\pm \! = \big|\Delta_1^\pm\big| \, $.  Here the  $ \, \gamma^\pm_d \in \Delta_1^\pm \, $  are all the positive or negative odd roots, ordered as in  Definition \ref{subgrps}.
                                                    \par
   Expanding the products expressing  $ g_\pm $  and  $ f_\pm $  we get
  $$  \displaylines{
   \qquad   g_\pm  \; = \;\,  1 \, + \, {\textstyle \sum\limits_{d=1}^{N_\pm}} \, \vartheta^\pm_d \, X_{\gamma^\pm_d} \, + \, {\textstyle \sum\limits_{k=2}^{N_\pm}} \; {\textstyle \, \sum\limits_{\underline{d} \, \in {\{1,\dots,N_\pm\}}^k}} \, {(-1)}^{k \choose 2} \, \underline{\vartheta}_{\,\underline{d}}^\pm \, X_{\gamma^\pm_{\underline{d}}}  \cr
   \qquad   f_\pm  \; = \;\,  1 \, + \, {\textstyle \sum\limits_{d=1}^{N_\pm}} \, \eta^\pm_d \, X_{\gamma^\pm_d} \, + \, {\textstyle \sum\limits_{k=2}^{N_\pm}} \; {\textstyle \, \sum\limits_{\underline{d} \, \in {\{1,\dots,N_\pm\}}^k}} \, {(-1)}^{k \choose 2} \, \underline{\eta}_{\,\underline{d}}^\pm \, X_{\gamma^\pm_{\underline{d}}}  }  $$
where
  $$  X_{\gamma^\pm_{\underline{d}}}  \, := \,  X_{\gamma^\pm_{d_1}} \cdots X_{\gamma^\pm_{d_k}} \; ,  \qquad  \underline{\vartheta}^\pm  \, := \,  \vartheta_{d_1} \cdots \vartheta_{d_k} \; ,  \qquad \underline{\eta}^\pm  \, := \,  \eta_{d_1} \cdots \eta_{d_k}  $$
for every  $ \, k > 1 \, $  and every  $ k $--tuple  $ \, \underline{d} := (d_1,\dots,d_k) \in {\{1,\dots,N_\pm\}}^k \, $.  For later use, note that these formulas imply also
  $$  \displaylines{
   f_-^{\,-1} \, g_-  = \,  1 + {\textstyle \sum\limits_{d=1}^{N_-}} \, \big( \vartheta^-_d \! - \eta^-_d \big) \, X_{\gamma^-_d} + {\textstyle \sum\limits_{k=2}^{2N_-}} \; {\textstyle \sum\limits_{\underline{d} \, \in {\{1,\dots,N_-\}}^k}} \hskip-7pt \Phi_{\gamma^-_{\underline{d}}}\big( \underline{\vartheta}^- , \underline{\eta}^- \big) \, X_{\gamma^-_{\underline{d}}}   \hfill (5.4)  \cr
   f_+ \, g_+^{\,-1}  = \,  1 + {\textstyle \sum\limits_{d=1}^{N_+}} \, \big( \eta^+_d \! - \vartheta^+_d \big) \, X_{\gamma^+_d} + {\textstyle \sum\limits_{k=2}^{2N_+}} \; {\textstyle \sum\limits_{\underline{d} \, \in {\{1,\dots,N_+\}}^k}} \hskip-7pt \Phi_{\gamma^+_{\underline{d}}}\big(\! -\underline{\vartheta}^+, -\underline{\eta}^+ \big) \, X_{\gamma^+_{\underline{d}}}   \hfill (5.5)  }  $$
where the  $ \Phi_{\gamma^\pm_{\underline{d}}} $'s  are suitable monomials (in the  $ \vartheta_i $'s  and  $ \eta_j $'s)  of degree  $ k $  with a coefficient  $ \pm 1 \,$,  and  $ \, \Phi_{\gamma^+_{\underline{d}}} = \pm \Phi_{\gamma^-_{\underline{d}}} \, $.
                                                    \par
   Note that, letting  $ V $  be the  $ \fg $--module  used in  Definition \ref{def_Che-sgroup_funct}  to define  $ G(A) \, $,  all the identities above actually hold inside  $ \End\big(V_\bk(A)\big) \, $.

\smallskip

   We proceed now to prove the following, intermediate result:

\bigskip

\noindent
 {\it  $ \underline{\text{Claim}} $:}  {\sl Let  $ \, g_\pm \, , \, f_\pm \in G_1^{\pm,<}(A) \, $  be such that  $ \, g_- \, g_+ = f_- \, f_+ \, $.  Then  $ \, g_\pm = f_\pm \; $}.

\medskip

   Indeed, let  $ \, V = \oplus_\mu V_\mu \, $  be the splitting of  $ V $  as direct sum of weight spaces.  Root vectors map weight spaces into weight spaces, via  $ \, X_\delta.V_\mu \subseteq V_{\mu + \delta} \; $  (for each root  $ \delta $  and every weight  $ \mu \, $).  An immediate consequence of this and of the expansions in (5.4--5) is that, for all weights  $ \mu $  and  $ \, v_\mu \in V_\mu \setminus \{0\} \, $,
  $$   \big( f_-^{\,-1} \, g_- \big).v_\mu \, \in \,
{\textstyle \bigoplus_{\gamma^- \in \N \Delta_1^-}} V_{\mu + \gamma^-}  \;\; ,  \qquad  \big( f_+ \, g_+^{\,-1} \big).v_\mu \, \in \, {\textstyle \bigoplus_{\gamma^+ \in \N \Delta_1^+}} V_{\mu + \gamma^+}  $$
where  $ \, \N \Delta_1^\pm $  is the  $ \N $--span  of  $ \, \Delta_1^\pm \, $.  In particular, this means that the only weight space in which both  $ \big( f_-^{\,-1} \, g_- \big).v_\mu $  and  $ \big( f_+ \, g_+^{\,-1} \big).v_\mu $  may have a non-trivial weight component is  $ V_\mu $  itself, as  $ \, \N \Delta_1^- \! \cap \N \Delta_1^+ = \{0\} \, $.  Moreover, let us denote by  $ \, {\big( \big( f_-^{\,-1} \, g_- \big).v_\mu \big)}_{\mu + \gamma_d^-} \, $  the weight component of  $ \, \big( f_-^{\,-1} \, g_- \big).v_\mu \, $  inside  $ V_{\mu + \gamma_d^-} \, $,  and similarly let  $ \, {\big( \big( f_+ \, g_+^{\,-1} \big).v_\mu \big)}_{\mu + \gamma_d^+} \, $  be the weight component of  $ \, \big( f_+ \, g_+^{\,-1} \big).v_\mu \, $  inside  $ V_{\mu + \gamma_d^+} \, $.  Then, looking in detail at (5.4--5), we find
  $$  \displaylines{
   \hfill   {\big( \big( f_-^{\,-1} \, g_- \big).v_\mu \big)}_{\mu + \gamma_d^-} \, = \; \big( \vartheta^-_d \! - \eta^-_d \big) \, X_{\gamma^-_d}.v_\mu   \hfill \forall \;\, d=1,\dots,N_-  \quad  \cr
   \hfill  {\big( \big( f_+ \, g_+^{\,-1} \big).v_\mu \big)}_{\mu + \gamma_d^+} \, = \; \big( \eta^+_d \! - \vartheta^+_d \big) \, X_{\gamma^+_d}.v_\mu   \hfill \forall \;\, d=1,\dots,N_+  \quad  }  $$
In fact, this is certainly true for  $ \gamma_d^\pm $  simple, and one can check directly case by case that any odd root  $ \gamma_d^\pm $  can never be the sum of three or more odd roots all positive or negative like  $ \gamma_d^\pm $  itself.  Now, since by hypothesis we have  $ \, g_- \, g_+ = \, f_- \, f_+ \; $,  \, so that  $ \, f_-^{-1} \, g_- = \, f_+ \, g_+^{-1} \; $,  \, comparing the weight components of the action of both sides of this equation on weight spaces  $ V_\mu $   --- and recalling that  $ \fg $  acts faithfully on  $ V \, $,  so  $ \, X_{\gamma_d^\pm}.V_\mu \not= 0 \, $  for some  $ \mu $  ---   we get right away  $ \, \vartheta_d^\pm = \eta_d^\pm \, $  for all  $ d \, $,  hence  $ \, g_\pm = f_\pm \; $,  \, q.e.d.   \hfill  $ \diamondsuit $

\vskip13pt

   Let now go on with the proof of the theorem.

\medskip

   By definition of  $ G_0(A) \, $,  both  $ g_0 $  and  $ f_0 $  are products of finitely many factors of type  $ x_\alpha(t_\alpha) $  and  $ h_i(s_i) $  for some  $ \, t_\alpha \in A_0 \, $, $ \, s_i \in U(A_0) \, $   --- with  $ \, \alpha \in \Delta_0 \, $,  $ \, i = 1, \dots, \ell \, $.  We call  $ B $  the superalgebra of  $ A $  generated by all the  $ \vartheta^\pm_d $'s,  the  $ \eta^\pm_d $'s,  the  $ \, t_\alpha $'s  and the  $ \, s_i $'s.  Then by construction  $ B $  is finitely generated (as a superalgebra), and  $ B_1 $  is finitely generated as a  $ B_0 $--module.
                                                    \par
   By  Lemma \ref{red-to-classical},  $ G(B) $  embeds injectively as a subgroup into  $ G(A) \, $;  so the identity  $ \; g_0 \, g_- \, g_+ \, = \, f_0 \, f_- \, f_+ \; $  also holds inside  $ G(B) \, $.  Thus we can switch from  $ A $  to  $ B \, $,  i.e.~we can assume from scratch that  $ \, A = B \, $.  In particular then,  $ A $  is finitely generated, hence  $ A_1 $  is finitely generated as an  $ A_0 $--module.

\medskip

  Consider in  $ A $  the ideal  $ A_1 \, $,  the submodules  $ A_1^{\,n} $  (cf.~\S \ref{first_preliminaries}),  for each  $ \, n \in \N \, $,  and the ideal  $ \big( A_1^{\,n} \big) $  of  $ A $  generated by  $ A_1^{\,n} \, $:  as  $ A_1^{\,n} $  is  {\sl homogeneous}, we have also  $ \; A \big/ \! \big( A_1^{\,n} \big) \in \salg \; $.
 Moreover, as  $ A_1 $  is finitely generated (over  $ A_0 $),  by assumption, we have  $ \, A_1^{\,n} = \{0\} = \big(A_1^{\,n}\big) \, $  for  $ \, n \gg 0 \, $.  So it is enough to prove
  $$  g_0 \equiv f_0  \mod \big(A_1^{\,n}\big) \; ,  \qquad  g_\pm \equiv f_\pm  \mod \big(A_1^{\,n}\big)   \eqno \forall \; n \in \N   \qquad (5.6)  $$
where, for any  $ \, A' \in \salg \, $,  any  $ I $  ideal of  $ A' $  with  $ \, \pi_I : A' \relbar\joinrel\twoheadrightarrow A' \big/ I \, $  the canonical projection, by  $ \; x \equiv y \mod I \; $  we mean that two elements  $ x $  and  $ y $  in  $ G(A') $  have the same image in  $ G \big( A' \big/ I \big) \, $  via the map  $ G(\pi_I) \, $.

\smallskip

   We prove (5.6) by induction.  The case  $ \, n = 0 \, $  is clear (there is no odd part).  We divide the induction step in two cases:  $ n $  even and  $ n $  odd.
                                                                   \par
   Let (5.6) be true for  $ n $  even.  In particular,  $ \, g_\pm \equiv f_\pm  \mod \big(A_1^{\,n}\big) \, $: then (see the proof of the  {\it Claim\/}  above) we have  $ \, \vartheta_d^\pm \equiv \eta_d^\pm \mod \big(A_1^{\,n}\big) \, $  for all  $ d \, $,  hence
  $$  (\vartheta_d^\pm - \eta_d^\pm)  \, \in \,  \big(A_1^{\,n}\big) \cap A_1 \subseteq \big(A_1^{n+1}\big) \qquad  \text{for all  $ d \, $}  $$
by an obvious parity argument.  Thus  $ \, g_\pm \equiv f_\pm  \mod \big(A_1^{n+1}\big) \, $  too, hence   --- since we have  $ \, g_0 \, g_- \, g_+ = f_0 \, f_- \, f_+ \, $  too ---   we get  $ \, g_0 \equiv f_0  \mod \big(A_1^{n+1}\big) \, $  as well: therefore, (5.6) holds for  $ \, n\!+\!1 \, $,  q.e.d.
                                                                   \par
   Let now (5.6) hold for  $ n $  odd.  Then  $ \, g_0 \equiv f_0  \mod \big(A_1^{\,n}\big) \, $;  but  $ \, g_0, f_0 \in G_0(A) = G_0(A_0) \, $  by definition, hence  $ \, g_0 \equiv f_0 \mod \big(A_1^{\,n}\big) \cap A_0 \, $.  Therefore  $ \, g_0 \equiv f_0 \mod \big(A_1^{n+1}\big) \, $,  because  $ \, \big(A_1^{\,n}\big) \cap A_0 \subseteq \big(A_1^{n+1}\big) \, $  by an obvious parity argument again.  Thus from  $ \, g_0 \, g_- \, g_+ = f_0 \, f_- \, f_+ \, $  we get also  $ \; g_- \, g_+ \equiv f_- \, f_+  \mod \big(A_1^{n+1}\big) \, $.  Then the  {\it Claim\/}  above   --- applied to  $ G\big(A\big/(A_1^{n+1})\big) $  ---   eventually gives  $ \, g_\pm \equiv f_\pm  \mod \big(A_1^{n+1}\big) \, $,  so that (5.6) holds for  $ \, n\!+\!1 \, $.
\end{proof}

\medskip

\begin{corollary} \label{nat-transf}
   The group product yields functor isomorphisms
  $$  \displaylines{
   G_0 \times G_1^{-,<} \times G_1^{+,<} \;{\buildrel \cong \over {\relbar\joinrel\lra}}\; G \quad ,  \qquad  \bG_0 \times \bG_1^{-,<} \times \bG_1^{+,<} \;{\buildrel \cong \over {\relbar\joinrel\lra}}\; \bG  }  $$
as well as those obtained by permuting the  $ (-) $-factor  and the  $ (+) $-factor  and/or moving the  $ (0) $-factor  to the right.  Moreover, all these induce similar functor isomorphisms with the left-hand side obtained by permuting the factors above, like  $ \, G_1^{+,<} \times G_0 \times G_1^{-,<} \;{\buildrel \cong \over \lra}\; G \, $,  $ \, \bG_1^{-,<} \times \bG_0 \times \bG_1^{+,<} \;{\buildrel \cong \over \lra}\; G \, $,  etc.
\end{corollary}

\begin{proof}
The first isomorphism arises from  Theorem \ref{crucial}.  The second one then is an easy consequence of the first one and of  Theorem \ref{sheaf-iso}  of Appendix A, because  $ \bG $
is the sheafification of  $ G \, $.  Similarly for the other functors.
\end{proof}

\medskip

\begin{remark}
 The four functor isomorphisms
  $$  \displaylines{
   G_1^{-,<} \times G_0 \times G_1^{+,<} \,\;{\buildrel \cong \over \lra}\;\, G  \quad ,  \qquad  G_1^{+,<} \times G_0 \times G_1^{-,<} \,\;{\buildrel \cong \over \lra}\;\, G  \cr
  \bG_1^{-,<} \times \bG_0 \times \bG_1^{+,<} \,\;{\buildrel \cong \over \lra}\;\, \bG  \quad ,  \qquad \bG_1^{-,<} \times \bG_0 \times \bG_1^{+,<} \,\;{\buildrel \cong \over \lra}\;\, \bG  }  $$
can be thought of as sort of a super-analogue of the classical  {\sl big cell decomposition\/}  for reductive algebraic groups.
\end{remark}

\medskip

\begin{proposition} \label{G1pm-repres}
  The functors  $ \, G_1^{\pm,<} : \salg \lra \sets \; $  are representable: they are the functor of points of the superscheme  $ \, \mathbb{A}_\bk^{0|N_\pm} $,  with  $ \, N_\pm = \big| \Delta_1^\pm \big| \, $.  In particular they are sheaves, hence  $ \; G_1^{\pm,<} = \bG_1^{\pm,<} \; $.
                                                      \par
   Similarly, we have also  $ \, G_1^< = \bG_1^< \cong \mathbb{A}^{0|N} $  as super-schemes,
where  $ \, N := \big| \Delta_1 \big| = N_+ + N_- \, $  (=the total number of odd roots).
\end{proposition}

\begin{proof}  Clearly, by the very definitions, there exists a natural transformation\break
 $ \, \Psi^\pm : \mathbb{A}_\bk^{0|N_\pm} \! \lra G_1^{\pm,<} \; $  given on objects by
 \vskip3pt
   \centerline{ $  \Psi^\pm(A) \, : \, \mathbb{A}_\bk^{0|N_\pm}(A) \! \lra G_1^{\pm,<}(A)  \quad ,  \qquad  (\vartheta_1,\dots,\vartheta_{N_\pm}) \, \mapsto \, {\textstyle \prod_{i=1}^{N_\pm}} \, x_{\gamma_i}(t_i) $ }
 \vskip5pt
\noindent
 Now, given elements
  $$  g_1^\pm = {\textstyle \prod_{i=1}^{N_\pm}} x_{\gamma_i}\!(\vartheta'_i)  \, \in \,  G_1^{\pm,<}(A)  \quad ,  \qquad h_1^\pm = {\textstyle \prod_{i=1}^{N_\pm}} x_{\gamma_i}\!(\vartheta''_i)  \, \in \, G_1^{\pm,<}(A)  $$
assume that  $ \; g^\pm_1 = h^\pm_1 \, $,  hence  $ \; h^-_1 \, {(g^-_1)}^{-1} = 1 \, $.  Then we get
  $$  \big( \vartheta'_1, \dots, \vartheta'_{N_\pm} \big)  \, = \,  \big( \vartheta''_1, \dots, \vartheta''_{N_\pm} \big)  $$
just as showed in the proof of  Theorem \ref{crucial}.  This means that  $ \Psi^\pm $  is an isomorphism of functors, which proves the claim.
\end{proof}

\bigskip

   Finally, we can prove that the Chevalley supergroups are algebraic:

\bigskip

\begin{theorem}  \label{representability}
   Every Chevalley supergroup  $ \bG $  is an algebraic supergroup.
\end{theorem}

\begin{proof}
 We only need to show that the functor  $ G $  is representable.  Now,  Corollary \ref{nat-transf}  and  Proposition \ref{G1pm-repres}  give  $ \; \bG \, \cong \, \bG_0 \times \bG_1^{-,<} \times \bG_1^{+,<} \; $,  \, with  $ \bG_1^{-,<} $  and  $ \bG_1^{+,<} $  being representable; but  $ \bG_0 $  also is representable   --- as it is a classical algebraic group, see  \S \ref{const-che-sgroup}.  But any direct product of representable functors is representable too (see  \cite{ccfd},  Ch.~5), so we are done.
\end{proof}

\bigskip

\begin{remark}
 This theorem asserts that Chevalley supergroup functors actually
provide algebraic supergroups.  This is quite remarkable, as some
of these supergroups had not yet been explicitly constructed
before. In fact, giving the functor of points of a supergroup it
is by no means sufficient to define the supergroup: proving the
representability   --- i.e., showing that there is a superscheme
whose functor of points is the given one ---   can be very hard.
                                                    \par
   For example using the procedure described above, it is possible to
construct the algebraic supergroups corresponding to all of the
exceptional classical Lie superalgebras  $ F(4) $,  $ G(3) $  and
$ D(2,1;a) $   --- for  $ \, a \in \Z \, $  ---   
 \hbox{and to the strange
Lie superalgebras.}

\smallskip

   The existence of such groups in the differential and analytic
categories is granted through the theory of Harish-Chandra
pairs, in which the category of supergroups is identified with
pairs consisting of a Lie group and a super Lie algebra, (see  \cite{koszul},  \cite{bcc}  and  \cite{vis}  for more details on this subject).  Our theory allows to realize such supergroups explicitly and over arbitrary fields.
\end{remark}

\medskip

   Another immediate consequence of  Corollary \ref{nat-transf}  and  Proposition \ref{G1pm-repres}  is the following, which improves, for Chevalley supergroups, a more general result proved by Masuoka  (cf.~\cite{ms},  Theorem 4.5) in the algebraic-supergeometry setting (see also  \cite{vis}, and references therein,  for the complex-analytic case).

\bigskip

\begin{proposition}
 For any Chevalley supergroup  $ \bG \, $,  there are isomorphisms of commutative superalgebras
  $$  \begin{array}{rl}
   \cO(\bG) \hskip-3pt  &  \cong \;  \cO(\bG_0) \otimes \cO\big(\bG_1^{-,<}\big) \otimes \cO\big(\bG_1^{+,<}\big)  \; \cong  \phantom{\Big|}  \\
                 &  \qquad \qquad  \cong \;  \cO(\bG_0) \otimes \bk \big[ \xi_1, \dots, \xi_{N_-} \big] \otimes \bk \big[ \chi_1, \dots, \chi_{N_+} \big]
      \end{array}  $$
where  $ \, N_\pm = \big| \Delta_1^\pm \big| \, $,  the subalgebra
$ \cO(\bG_0) $  is totally even, and  $ \, \xi_1, \dots, \xi_{N_-}
\, $  and  $ \, \chi_1, \dots, \chi_{N_+} \, $  are odd elements.
\end{proposition}

\medskip

   We conclude this section with the analysis of a special case, that of commutative superalgebras  $ A $ for which  $ \, A_1^2 = \{0\} \, $.  This is a typical situation in commutative algebra theory: indeed,  {\sl any such  $ A $  is nothing but the central extension of the commutative algebra  $ A_0 $  by the  $ A_0 $--module  $ A_1 \, $}.

\medskip

\begin{proposition} \label{semidirect}
  Let  $ G $  be a Chevalley supergroup functor, and let  $ \bG $  be its associated Chevalley supergroup.  Assume  $ \, A \in \salg \, $  is such that  $ \, A_1^2 = \{0\} \, $.
                                                                    \par
   Then  $ G_1^+(A) \, $,  $ G_1^-(A) $  and  $ G_1(A) $  are  {\sl normal subgroups}  of\/  $ G(A) \, $, with
  $$  \displaylines{
   \hfill \hskip43pt   G_1^\pm(A)  \; = \;  G_1^{\pm,<}(A)  \; \cong \;  \mathbb{A}_\bk^{0|N_\pm}\!(A)   \hfill  \text{with\ } \; N_\pm = \big| \Delta_1^\pm \big|  \cr
   G_1(A)  \; = \;  G_1^-(A) \cdot G_1^+(A)  \; = \;   G_1^+(A) \cdot G_1^-(A)   \qquad  \cr
   G_1(A)  \; \cong \;  G_1^-(A) \times G_1^+(A)  \; \cong \;  G_1^+(A) \times G_1^-(A)  \; \cong \; \mathbb{A}_\bk^{0|N_-}\!(A) \times \mathbb{A}_\bk^{0|N_+}\!(A)  }  $$
(where  ``$ \; \cong $''  means isomorphic as groups), the group
structure on  $ \mathbb{A}_\bk^{0|N_\pm}\!(A) $  being the obvious
one.  In particular,  $ G(A) $  is the semidirect product, namely
 $ \; G(A) \, \cong \,  G_0(A) \ltimes G_1(A)  \, \cong \,  G_0(A_0) \ltimes \Big( \mathbb{A}_\bk^{0|N_-}\!(A) \times \mathbb{A}_\bk^{0|N_+}\!(A) \Big) \, $,  of the  {\sl classical}  group  $ G_0(A_0) $  with the  {\sl totally odd}  affine superspace  $ \, \mathbb{A}_\bk^{0|N_-}\!(A) \times \mathbb{A}_\bk^{0|N_+}\!(A) \, $.
 \vskip4pt
   Similar results hold with a symbol  ``$ \, \bG $''  replacing  ``$ \, G $''  everywhere.
\end{proposition}

\begin{proof}
 The assumptions on  $ A $  and the commutation formulas in  Lemma \ref{comm_1-pssg}{\it (b)\/}  ensure that all the 1-parameter subgroups associated to odd roots do commute with each other.  This implies that  $ \, G_1^-(A) \, G_1^+(A) = G_1^+(A) \, G_1^-(A) \, $,  that the latter coincides with  $ G_1(A) \, $,  that  $ \, G_1^\pm(A) = G_1^{\pm,<}(A) \, $,  and also that  $ G_1^\pm(A) $  and  $ G_1(A) $  are subgroups of  $ G(A) \, $.  Moreover,  Lemma \ref{comm_1-pssg}{\it (a)\/}  implies that  $ G_1^\pm(A) $  and  $ G_1(A) $  are also normalized by  $ G_0(A) \, $.  By  Theorem \ref{g0g1}  we conclude that  $ G_1^\pm(A) $  and  $ G_1(A) $  are normal in  $ G(A) \, $.
                                                                       \par
   All remaining details follow from  Proposition \ref{G1pm-repres}  and  Theorem \ref{representability}.
The statement for  $ \bG $  clearly follows as well.
\end{proof}

\bigskip
\bigskip
\bigskip

 \section{Independence of Chevalley and Kostant superalgebras}  \label{ind_che-kost_alg}

\medskip

   Next question is the following: what is the role played by the representation  $ V \, $?
Moreover, we would like to show that our construction is
independent of the choice of an admissible  $ \Z $--lattice  $ M $
in a fixed  $ \fg $--module  (over  $ \KK \, $).

\medskip

   Let  $ \bG' $  and  $ \bG $  be two Chevalley supergroups obtained by the same  $ \fg \, $,  possibly with a different choice of the representation.  We denote with  $ X_\alpha $  and with  $ X'_\alpha $ respectively the elements of the Chevalley basis in  $ \fg $  identified (as usual) with their images under the two representations of  $ \fg \, $.

\medskip

\begin{lemma}
 Let  $ \, \phi : \bG \lra \bG' \, $  be a morphism of Chevalley supergroups such that on local superalgebras  $ A $  we have
  $$  \displaylines{
   \indent  \text{\it (1)}  \hfill \qquad   \phi_A\big(\bG_0(A)\big) = \bG'_0(A)   \hskip115pt \hfill \quad  \cr
   \indent  \text{\it (2)}  \hfill \qquad   \phi_A\big(1 + \vartheta \, X_\beta\big) = 1 + \vartheta \, X'_\beta  \quad \hfill  \forall \;\; \beta \in \Delta_1 \; ,  \,\; \vartheta \in A_1  \quad  }  $$
   \indent   Then  $ \; \text{\it Ker}(\phi_A) \subseteq \bT \; $,  \; where  $ \bT $  is the maximal torus in the ordinary algebraic group  $ \, \bG_0 \subseteq \bG \, $  (see  \S \ref{const-che-sgroup}).
\end{lemma}

\begin{proof}
 For any  {\sl local\/}  superalgebra  $ A $  we have  $ \, \bG(A) = G(A) \, $.  Now let  $ \, g \in \bG(A) = G(A) \, $,  $ \, g \in \text{\it Ker}(\phi_A) \, $.  By  Theorem \ref{crucial}  we have  $ \; g = g_1^- \, g_0 \, g_1^+ \; $  with  $ \, g_0 \! \in \! G_0(A) \, $,  $ \, g_1^\pm \! \in \! G_1^{\pm,<}(A) \, $;  \, but then
  $$  \phi_A\big(g_1^-\big) \, \phi_A(g_0) \, \phi_A\big(g_1^+\big)  \, = \,  \phi_A(g)  \, = \,  e_{{}_{\bG'}}  $$
By the assumption  {\it (2)\/}  and the uniqueness of expression of  $ g \, $,  we have that  $ \; \phi_A\big(g_1^-\big) = e_{{}_{\bG'}} = \phi_A\big(g_1^+\big) \; $  and  $ \; g_0 \in \text{\it Ker}\,(\phi_{0,A}) \subseteq \bT(A) \, $,  \; where  $ \phi_{0,A} $  is the restriction of  $ \phi_A $  to  $ \bG_0(A) \, $.  The claim follows.
\end{proof}

\bigskip

   Let now  $ L_0 $  be the root lattice of  $ \fg \, $;  also, we let  $ L_1 $  be the weight lattice of  $ \fg \, $,  defined to be the lattice of weights of all rational  $ \fg $--modules.
                                                 \par
   For any lattice $ L $  with  $ \, L_0 \subseteq L \subseteq L_1 \, $,  there is a corresponding Chevalley supergroup.  The relation between Chevalley supergroups corresponding to
different lattices is the same as in the classical setting.

\bigskip

\begin{theorem} \label{mainthm}
 Let  $ \bG $  and  $ \bG' $  be two Chevalley supergroups constructed using two representations  $ V $  and  $ V' $  of the same  $ \fg $  over the same field\/  $ \KK $  (as in  \S \ref{adm-lat}),  and let $ L_V $,  $ L_{V'} $  be the corresponding lattices of weights.
                                                 \par
   If  $ \, L_V \supseteq L_{V'} \, $,  then there exists a unique morphism  $ \; \phi : \bG \longrightarrow \bG' \; $  such that  $ \,\; \phi_A\big(1 + \vartheta \, X_\alpha\big) = 1 + \vartheta \, X'_\alpha \; $,  \, and  $ \,\; \text{\it Ker}\,(\phi_A) \subseteq Z\big(\bG(A)\big) \, $,  \; for every local algebra  $ A \, $.  Moreover,  $ \phi $  is an isomorphism if and only if  $ \, L_V = L_{V'} \, $.
\end{theorem}

\begin{proof}
  As the same theorem is true for the classical part  $ \bG_0 \, $,  we can certainly set up a map  $ \; \phi_0 : G_0 \lra G'_0 \, $  and the corresponding one on the sheafification.  Now we define  $ \; \phi : \bG \lra \bG' \; $  in the following way.  For  $ \, A \in \salg \, $,  we set  $ \; \phi_A\big(1 + \vartheta \, X_\alpha\big) := 1 + \vartheta \, X'_\alpha \, $, $ \; \phi_A(g_0) :=
\phi_{0,A}(g_0) \;\, $;  \, then
  $$  \displaylines{
   \phantom{\Big|} \;   \phi_A\big( \big( 1 + \vartheta_1 \, X_{\alpha_1} \big) \cdots \big( 1 + \vartheta_r \, X_{\alpha_r} \big) \, g_0 \, \big( 1 + \eta_1 \, X_{\beta_1} \big) \cdots \big( 1 + \eta_s \, X_{\beta_s} \big) \big)  \; =   \hfill  \cr
   \hfill   = \;  \big( 1 + \vartheta_1 \, X'_{\alpha_1} \big) \cdots \big( 1 + \vartheta_r \, X_{\alpha_r} \big) \, \phi_{0,A}(g_0) \, \big( 1 + \eta_1 \, X'_{\beta_1} \big) \cdots \big( 1 + \eta_s \, X'_{\beta_s} \big)   \;\;  }  $$
This gives a well-defined $ \phi_A $  which in fact is also a
morphism (i.e., natural transformation): indeed,  $ \;
\phi_A(g\,h) = \phi_A(g) \, \phi_A(h) \; $  because all the
relations used to commute elements in  $ G_1^-(A) $,  $ G_0(A) $
and  $ G_1^+(A) $   --- so to write a given element in  $ G(A) $
in the ``normal form'' as in  Corollary \ref{nat-transf}  ---   do
not depend on the chosen representation, where now $A$ is taken to be local
(by Proposition
\ref{sheaf-morph} in the Appendix A, we have that the natural
transformation $\phi$ is uniquely determined by its behaviour on local
superalgebras).
\end{proof}

\bigskip

   As a direct consequence, we have the following ``independence result'':

\bigskip

\begin{corollary}
 Every Chevalley supergroup  $ \, \bG_V $  is independent   --- up to isomorphism ---   of the choice of an admissible lattice  $ M $  of  $ V $  considered in the very construction of  $ \, \bG_V $  itself.
\end{corollary}

\begin{proof}
 Let  $ M $  and  $ M' $  be two admissible lattices of  $ V \, $.  Then consider  $ \, V' := V \, $,  and consider  $ \bG_V $  and  $ \bG_{V'} $  constructed using respectively the two lattices  $ M $  and  $ M' $.  By construction we have  $ \, L_V = L_{V'} \, $,  hence  Theorem \ref{mainthm}  give  $ \, \bG_V \cong \bG_{V'} \, $,  which proves the claim.
\end{proof}

\bigskip
\bigskip
\bigskip

 \section{Lie's Third Theorem for Chevalley supergroups}  \label{che-sgroup_LieTT}

\medskip

   Let now  $ \bk $  be a  {\sl field},  with  $ \, \text{char}(\bk) \neq 2, 3 \, $.

\medskip

   Let  $ \bG $  be a Chevalley supergroup scheme over  $ \bk $,  built  out of a classical Lie superalgebra  $ \fg $  over  $ \KK $  as in  \S \ref{const-che-sgroup}.
%
%
 In  \S \ref{adm-lat},  we have constructed the Lie superalgebra  $ \, \fg_\bk := \bk \otimes_\Z \fg_V \, $  over  $ \bk $  starting from the  $ \Z $--lattice  $ \fg_V \, $.  We now show that the algebraic supergroup  $ \bG $  has  $ \fg_\bk $  as its tangent Lie superalgebra.

\medskip

   We start recalling how to associate a Lie superalgebra to a supergroup scheme.  For more details see \cite{ccfd}.

\medskip

   Let  $ \, A \in \salg \, $  and let  $ \, A[\epsilon] := A[x]\big/\big(x^2\big) \, $  be the  {\sl super\/}algebra  of dual numbers, in which  $ \; \epsilon := x \! \mod \! \big(x^2\big) \; $  is taken to be  \textit{even}.  We have that  $ \, A[\epsilon] = A \oplus \epsilon \, A \, $,  and there are two natural morphisms
 $ \; i : A \longrightarrow A[\epsilon] \, , \, a \;{\buildrel i \over \mapsto}\; a \, $,  and  $ \; p : A[\epsilon] \longrightarrow A \, $,  $ \, (a + \epsilon \, a' \big) \;{\buildrel p \over \mapsto}\; a \; $,  \; such that  $ \; p \circ i= {\mathrm{id}}_A \; $.

\medskip

\begin{definition} \label{supergroup}
   For each supergroup scheme  $ G \, $,  consider the homomorphism  $ \; G(p): G (A(\epsilon)) \longrightarrow G(A) \; $.  Then there is a supergroup functor
  $$  \Lie(G) : \salg \lra \sets \; ,  \qquad  \Lie(G)(A) \, := \, \text{\it Ker}\,(G(p))  $$
\end{definition}

\medskip

\begin{proposition}
   (cf.~\cite{ccfd}, \S 6.3.) Let  $ G $  be a supergroup scheme.  The functor  $ \Lie(G) $  is representable and can be identified with (the functor of points of) the tangent space at the identity of  $ G \, $,  namely
  $$  \Lie(G)(A)  \, = \,  {(A \otimes T_{1_G})}_0  $$
where  $ T_{1_G} $  is the super vector space  $ \, \mathfrak{m}_{G,1_G} \big/ \mathfrak{m}_{G,1_G}^2 \, $,  with  $ \, \mathfrak{m}_{G,1_G} \, $  being the maximal ideal of the local algebra  $ \, \cO_{G,1_G} \, $.
\end{proposition}

\medskip

   With an abuse of notation we will use the same symbol  $ \Lie(G) $  to denote both the functor and the underlying super vector space.

\smallskip

   Next we show that  $ \Lie(G) $  has a Lie superalgebra structure: this is equivalent to asking the functor  $ \, \Lie(G) : \salg \ra \sets \, $  to be Lie algebra valued, by which we just mean that it is actually a functor from  $ \salg $  to  $ \lie $,  the category of Lie algebras over  $ \bk \, $.

\medskip

\begin{definition}
 Define the  {\it adjoint action\/}  of  $ G $  on  $ \Lie(G) $  as
  $$  \Ad : G  \lra  \rGL(\Lie( G )) \quad ,  \qquad  \Ad(g)(x) \, := \, G (i)(g) \cdot x \cdot {\big(G (i)(g)\big)}^{-1}  $$
 \vskip5pt
\noindent
 for all  $ \, g \in G(A) \, $,  $ \, x \in \Lie(G)(A) \, $.  Define also the  {\it adjoint morphism\/}  $ \ad $  as
  $$  \ad  \, := \,  \Lie(\Ad) : \Lie(G) \lra \Lie(\rGL(\Lie(G))) := \End(\Lie(G))  $$
 \vskip5pt
\noindent
 where  $ \rGL $  and  $ \End $  are the functors defined as follows:  $ \, \rGL(V)(A) \, $  and  $ \, \End(V)(A) \, $,  for a supervector space  $ V $,  are respectively the automorphisms and the endomorphisms of  $ \, V(A) := {(A \otimes V)}_0 \; $.
 \vskip4pt

   Finally, we define  $ \; [x,y] := \ad(x)(y) \, $,  \; for all  $ \, x,y \in \Lie(G)(A) \, $.
\end{definition}

\smallskip

\begin{proposition}
 (cf.~\cite{ccfd}, \S 6.3.)
The functor  $ \, \Lie(G) : \salg \lra \sets \, $  is  {\sl Lie algebra valued},  via the bracket  $ \, [\,\ ,\ ] \, $  defined above.  In other words, it yields a functor
  $$  \Lie(G) : \salg \lra \lie  $$
where  $ \, \lie \, $  stands for the category of Lie algebras over  $ \bk \, $.
\end{proposition}

\bigskip

%
%
%
%
%

   Let us now see an important example.

\bigskip

\begin{example}
 We compute the functor  $ \Lie(\rGL_{m|n}) \, $.  Consider the map
%
%
  $$  \displaylines{
   \rGL_{m|n}(p) :  \phantom{\Big|}  \rGL_{m|n}\big(A(\epsilon)\big) \relbar\joinrel\lra \rGL_{m|n}(A)  \cr
   \hskip7pt
 \begin{pmatrix} p + \epsilon \, p'  &  q + \epsilon \, q'  \\
                   r + \epsilon \, r'  &  s + \epsilon \, s'  \end{pmatrix}
 \hskip9pt  \mapsto  \hskip9pt
 \begin{pmatrix}  p  &  q  \\
                    r  &  s  \end{pmatrix}  }  $$
with  $ p $,  $ p' $,  $ s $  and  $ s' $  having entries in  $ A_0 \, $,  and  $ q $,  $ q' $,  $ r $  and  $ r' $  having entries in  $ A_1 \, $;  moreover,  $ p $  and  $ s $  are invertible matrices.  One can see immediately that
  $$  \Lie(\rGL_{m|n})(A)  \; = \;  \text{\it Ker}\,\big(\rGL_{m|n}(p)\big)  \; = \;
\left\{ \begin{pmatrix} I_m + \epsilon \, p'  &  \epsilon \, q'  \\
                        \epsilon \, r'  &  I_n + \epsilon \, s'  \end{pmatrix} \right\}  $$
where  $ I_\ell $  is an  $ \ell \times \ell $  identity matrix.
The functor  $ \Lie(\rGL_{m|n}) $  is clearly group valued and can
be identified with the (additive) group functor  $ M_{m|n} $
  $$  M_{m|n}(A)  \; = \; \Hom\big(M(m|n)^*,A\big)  \; = \;  \Hom_\salg\big({\mathrm{Sym}}\big(M(m|n)^*\big),A\big)  $$
where
  $ \; M(m|n) \, := \, \left\{ \begin{pmatrix}  P  &  Q  \\
                                                 R  &  S  \end{pmatrix} \right\}
\cong \bk^{m^2+n^2|2mn} \, $
--- with  $ P $,  $ Q $,  $ R $  and  $ S $  being  $ m \times m $,  $ m \times n $,  $ n \times m $  and  $ n \times n $  matrices with entries in  $ \bk \, $  ---   is a supervector space.  An  $ \, X \in M(m|n) \, $  is even iff  $ \, Q = R = 0 \, $,  it is odd iff  $ \, P = S = 0 \, $.
 \vskip4pt
   Notice that  $ M(m|n) $  is a Lie superalgebra, whose Lie superbracket is given by  $ \; [X,Y] \, := \, X Y - {(-1)}^{p(X)p(Y)} \, Y X \; $,  \; so  $ \Lie(\rGL_{m|n}) $  is a Lie superalgebra.

\medskip

   Let us compute explicitly for this special case the morphisms  $ \Ad $  and  $ \ad \, $.

\medskip

   Since  $ \, G (i) : \rGL_{m|n}(A) \lra \rGL_{m|n}\big(A(\epsilon)\big) \, $  is an inclusion, if we identify  $ \rGL_{m|n}(A) $  with its image we can write
  $$  \Ad(g)(x) \, = \, g \, x \, g^{-1} \; ,  \eqno \forall \;\; g \in \rGL_{m|n}(A) \, , \; x \in {\mathrm{M}}_{m|n}(A)  \quad  $$
By definition we have  $ \, \Lie\big(\rGL({\mathrm{M}}_{m|n})\big)(A) = \big\{\, 1 + \epsilon
\, \beta \;\big|\; \beta \in \rGL({\mathrm{M}}_{m|n})(A) \,\big\} \phantom{\Big|} $.  So we have, for  $ \; a, b \in {\mathrm{M}}_{m|n}(A) \cong \Lie(\rGL_{m|n})(A) = \big\{ 1 + \epsilon \, a \,\big|\, a \in
{\mathrm{M}}_{m|n}(A) \big\} \, $,
  $$  \ad(1 + \epsilon \, a)(b)  \, = \,  (1 + \epsilon \, a) \, b \, (1 - \epsilon \, a)  \, = \,  b + (ab - ba) \, \epsilon \, = \,  b + \epsilon \, [a,b]  $$
The outcome is  $ \; \ad(1 + \epsilon \, a) = {\mathrm{id}} +
\epsilon \, \beta(a) \; $,  \, with  $ \, \beta(a) = [a,\text{\bf
--}\;] \; $.
\end{example}

\bigskip

We are ready for the main theorem of this section.

\bigskip

\begin{theorem}
 If  $ \, \bG = \bG_V \, $  is a Chevalley supergroup built upon\/  $ \fg $  and  $ V \, $,  then  $ \; {\Lie}(\bG_V) = \fg \; $  as functors with values in  $ \lie \, $.
\end{theorem}

\begin{proof} The first remark is that all our arguments take place inside  $ \rGL(V) $,  hence we can argue using the formulas of the previous example.  Certainly we know that the two spaces under exam have the same (super)dimension, in fact by  Theorem \ref{representability}  we know that  $ \, \bG = \bG_0 \times \bG_1^< \, $,  hence its tangent space at the origin has dimension  $ \, \text{\it dim}(\fg_0)\,\big|\text{\it dim}(\fg_1) \, $.  It is also clear by classical considerations that  $ \, {\Lie}(\bG_V)_0 = \fg_0 \, $.  An easy calculation shows that  $ {\Lie}(\bG_V) $  contains all the generators of  $ \fg_1 \, $,  hence we have the result.
\end{proof}

%

 %
 %

\chapter{The cases  $ A(1,1) \, $,  $ P(3) $  and  $ Q(n) $}  \label{cases A-P-Q}

\bigskip

 {\it
   This chapter treats the ``singular'' cases, namely those when  $ \fg $  is of type  $ A(1,1) $,  $ P(3) $  or  $ Q(n) \, $.  Indeed, in these cases the roots happen to have non-trivial multiplicity (i.e., root spaces have multiplicity greater than one), so additional care is needed: therefore, they are dropped from the discussion in the main body of the text   --- chapters 3 through 5 ---   and dealt with here.
                                                             \par
   The main point is that we have to re-visit our definition of Chevalley basis, which our whole construction is based upon.  In short, the outcome is that one still has a similar notion, but one may have more than one root vector attached to a unique root: indeed, they will be just as many as the root multiplicity.

\smallskip

   Actually, in cases  $ A(1,1) $  and  $ P(3) $  essentially it is enough to adopt a suitable notation: then almost nothing changes, that is one can replicate, step by step, all analysis and arguments carried on before, essentially with no change.

\smallskip

   On the other hand, in case  $ Q(n) $  one must undertake from scratch a careful revision of the entire construction.  Indeed, the very definition of Chevalley basis must be carefully (re)thought: once this is done, we prove the existence again by direct inspection   --- i.e., we do provide an explicit Chevalley basis ---   but the ``uniform argument'' mentioned in Chapter 3,  Remark
 3.8,
%
%
 also applies again.
                                                                \par
   The changes in the definition of a Chevalley basis for  $ Q(n) $  induce some changes in the construction and description of  $ \kzg $  too.  But then, all the rest of our construction of Chevalley supergroups goes through essentially unchanged.
%
%
                                                                             \par
   The main fact to point out is a special feature of these three cases which make them different from all other ones: namely, some (odd) roots have multiplicity greater than one, so the set of (odd) roots itself is no longer fit to index (odd) root vectors in a basis.  This leads us to introduce a different index set for root vectors, still close to the root set but definitely different.

\smallskip

   As to types  $ A(1,1) $  and  $ P(3) \, $,  according to  \S \ref{root-syst}  in both cases there exist linear dependence relations that identify some (odd) positive roots with some (odd) negative ones.  Now, for the common value of such a root we can find a root vector when considering the root as a positive one, and another, linearly independent root vector, when looking at the same root as a negative one.  In the end, any such root has exactly multiplicity 2, and we just have to find out a neat way to index all root vectors in a consistent manner.
                                                                             \par
   The solution is immediate: the same way of indexing root vectors that we adopted respectively for  $ A(n,n) $   ---  $ n > 1 $  ---   and for  $ P(m) $   ---  $ m \not= 2 $  ---  still works in the present context.  Similarly, the description of a Chevalley basis is (up to using a suitably adapted notation) essentially the same   --- both for  $ A(1,1) $  and  $ P(3) $  ---   as in the general case.  Then all our results   --- about Kostant algebras, Chevalley groups and their properties ---   follow: statements and proofs are the same, just minimal notational changes occur.
%
%

\smallskip

   For type  $ Q(n) $  instead, the new, ``exotic'' feature is that the root set includes also  $ 0 \, $,  as an {\sl odd\/}  root with multiplicity  $ n \, $,  while all other roots are both even and odd and have multiplicity  $ 2 \, $.  Root vectors then must be indexed by a greater set then the root set.  Moreover, all root vectors relative to the root  $ 0 $  are odd, while for any root  $ \, \alpha \not= 0 \, $  there exist an  {\sl even\/}  as well as an  {\sl odd\/}  root vector attached to  $ \alpha \, $.  With such root vectors we can build up a (suitably defined) Chevalley basis.
                                                                             \par
  All our results again hold true in case  $ Q(n) $  too, with the same statements; however, in this case some proofs need additional arguments, due to the special behavior of the root 0 and of the root vectors in a Chevalley basis.
}
\bigskip

 \section{Chevalley bases and Chevalley superalgebras}  \label{che_bas-alg_A-P-Q}

\medskip

   Throughout the section  $ \fg $  denotes any Lie superalgebra of type  $ A(1,1) \, $,  $ P(3) $  or  $ Q(n) \, $.  We begin with the following definition   --- sort of a generalization of root system ---   which eventually will be our tool to index root vectors in a (suitably defined) Chevalley basis.

\medskip

\begin{definition}  \label{Deltatilde}
 We define a set  $ \; \wiDelta := \wiDelta_0 \coprod \wiDelta_1 \; $,  a map  $ \; \pi : \wiDelta \longrightarrow \Delta \cup \{0\} \; $  and two partial operations on  $ \wiDelta $  (``partial'' in the sense that  $ \wiDelta $  is not closed with respect to them)  $ \; + : \wiDelta \times \wiDelta \dashrightarrow \wiDelta \, $,  $ \; - : \wiDelta \dashrightarrow \wiDelta \, $,  \, as follows.
 \vskip4pt
   {\it (a)} \,  if  $ \, \fg = A(1,1) \, $,  we set  $ \, \wiDelta_0^\pm \! := \Delta_0^\pm \, $  and  $ \; \wiDelta_0 := \wiDelta_0^+ \cup \wiDelta_0^- = \Delta_0 \, $.  The latter can be described by subsets of pairs indexing the (even) positive or negative roots   --- in the (classical) root system of type  $ A(1) \times A(1) $  ---   namely  $ \; \wiDelta_0 := \wiDelta_0^+ \coprod \wiDelta_0^- \; $  with
  $$  \wiDelta_0^+  \, := \,  \Delta_0^+  \, = \,  \big\{(1,2),(3,4)\big\} \;\; ,  \quad  \wiDelta_0^- \, := \, \Delta_0^- \, = \, \big\{(2,1),(4,3)\big\} \;\, ;  $$
with these identifications, one has  $ \, -(i,j) = (j,i) \, $  for the opposite of a root.
                                                        \par
   Similarly, we define
  $$  \wiDelta_1^+  \, := \,  \big\{(1,3),(1,4),(2,3),(2,4)\big\} \;\; ,  \quad  \wiDelta_1^-  \, := \, \big\{(3,1),(4,1),(3,2),(4,2,)\big\}  $$
and eventually we set  $ \; \wiDelta_1 := \wiDelta_1^+ \coprod \wiDelta_1^- \; $.

\smallskip

   The partial operations  $ \, - \, $  and  $ \, + \, $  on  $ \wiDelta $  are given by
  $$  -(r,s) := (s,r) \;\; ,  \qquad  (i,j) + (h,k) := \delta_{j,h} (i,k) - {(-1)}^{\varepsilon(i,j) \, \varepsilon(h,k)} \delta_{k,i} (h,j)  $$
where  $ \, \varepsilon(t,l) := 1 \, $  if  $ \, t, l \geq 2 \, $
or  $ \, t, l > 2 \, $,  \, and  $ \, \varepsilon(t,l) := -1 \, $
otherwise.
                                                                      \par
   Now, the set  $ \Delta_1 $  of odd roots of  $ \fg $  identifies with the quotient space  $ \, \Delta_1 = \wiDelta_1 \big/ \!\! \sim \, $,  \, where the equivalence relation  $ \, \sim \, $  between pairs given by
  $$  (1,4) \sim (3,2) \;\; ,  \qquad  (1,3) \sim (4,2) \;\; ,  \qquad  (2,3) \sim (4,1) \;\; ,  \qquad  (2,4) \sim (3,1) $$
(note that the partition into  $ \sim $--equivalence  classes is ``transversal'' to the partition  $ \, \wiDelta_1 := \wiDelta_1^+ \coprod \wiDelta_1^- \, $).  We define  $ \; \pi : \wiDelta = \wiDelta_0 \coprod \wiDelta_1 \longrightarrow \Delta \cup \{0\} \; $  as the identity map on  $ \, \wiDelta_0 = \Delta_0 \, $  and as the quotient map on  $ \wiDelta_1 \, $.
 \vskip6pt
   {\it (b)} \,  if  $ \, \fg = P(3) \, $,  \, we set  $ \, \wiDelta_0^\pm \! := \Delta_0^\pm \, $  and  $ \; \wiDelta_0 := \wiDelta_0^+ \cup \wiDelta_0^- = \Delta_0 \, $.  The latter can be described by subsets of pairs indexing the (even) positive or negative roots   --- in the (classical) root system of type  $ A(3) $  ---   namely  $ \; \wiDelta_0 := \wiDelta_0^+ \coprod \wiDelta_0^- \; $  with
  $$  \wiDelta_0^+  \, := \,  \Delta_0^+ = \big\{ (i,j) \big\}_{1 \leq i < j \leq 4} \;\; ,  \quad  \wiDelta_0^- \, := \, \Delta_0^- = \big\{ (j,i) \big\}_{1 \leq i < j \leq 4}  $$
with  $ \, -(r,s) = (s,r) \, $.  Then we define  $ \; \wiDelta_1 := \wiDelta_1^+ \coprod \wiDelta_1^- \; $  with
  $$  \wiDelta_1^+ \, := \, \big\{ [h,k] \big\}_{1 \leq h \leq k \leq 4}  \;\; ,  \quad  \wiDelta_1^-  \, := \,  \big\{ [q,p] \big\}_{1 \leq p < q \leq 4} \;\; .  $$
                                                                      \par
   The partial operations  $ \, - \, $  and  $ \, + \, $  on  $ \wiDelta $  are defined as follows. Let us consider in the free  $ \Z $--module  $ \Z^4 $  the canonical basis, whose elements are
  $$  \varepsilon_1 := (1,0,0,0) \; ,  \quad  \varepsilon_2 := (0,1,0,0) \; ,  \quad  \varepsilon_3 := (0,0,1,0) \; ,  \quad  \varepsilon_4 := (0,0,0,1)  $$
and let us consider the embedding of  $ \wiDelta $  into  $ \Z^4 $  given by
  $$  (r,s) \mapsto (\varepsilon_r - \varepsilon_s)  \quad ,  \qquad  [h,k] \mapsto (\varepsilon_h + \varepsilon_k)  \quad ,  \qquad  [q,p] \mapsto -(\varepsilon_q + \varepsilon_p)  $$
for all  $ \, r \not= s \, $,  $ \, h \leq k \, $,  $ \, p \lneqq q \, $.  Then we define the partial operations  $ \, - \, $  and  $ \, + \, $  on  $ \wiDelta $  as being the restrictions of the same name operations in  $ \Z^4 $,  taking the former as defined (on  $ \wiDelta \, $)  whenever the result belongs to  $ \wiDelta $  itself.
                                                                      \par
   Let  $ \, \sim \, $  be the equivalence relation in  $ \wiDelta_1 $  given by
  $$  \displaylines{
   \qquad   [1,2] \sim [4,3]  \quad ,  \qquad  [1,3] \sim [4,2]  \quad ,  \qquad  [1,4] \sim [3,2]  \quad ,   \hfill \cr
   \hfill   [2,3] \sim [4,1]  \quad ,  \qquad  [2,4] \sim [3,1]  \quad ,  \qquad  [3,4] \sim [2,1]  \quad ; \qquad  }  $$
then the set  $ \Delta_1 $  of odd roots of  $ \fg $  identifies with the quotient space  $ \, \Delta_1 = \wiDelta_1 \big/ \!\! \sim \; $.  We define  $ \; \pi : \wiDelta = \wiDelta_0 \coprod \wiDelta_1 \longrightarrow \Delta \cup \{0\} \; $  as the identity map on  $ \, \wiDelta_0 = \Delta_0 \, $  and as the quotient map on  $ \wiDelta_1 \; $.
 \vskip6pt
   {\it (c)} \,  if  $ \, \fg = Q(n) \, $,  \, then  $ \, \Delta = \Delta_1 = \Delta_0 \coprod \, \{0\} \, $.  We set  $ \, \wiDelta_0^\pm \! := \Delta_0^\pm \times \{(0,1)\} \, $  and  $ \; \wiDelta_0 :=
\wiDelta_0^+ \cup \wiDelta_0^- \, $.  Then we fix any partition  $ \, I_n^+ \! \coprod I_n^- \! = \{1,\dots,n\} \, $  of  $ \, \{1,\dots,n\} \, $,  \, and we define
  $$  \wiDelta^\pm_1  \, := \,  \big( (\Delta^\pm_1 \setminus \{0\}) \! \times \! \{(1,1)\} \big) \,{\textstyle \coprod}\, \big( \{0\} \times {\{(1,i)\}}_{i \in I_n^\pm} \big)  $$
and  $ \; \wiDelta_1 := \wiDelta_1^+ \coprod \wiDelta_1^- \; $.  Note that definitions imply  $ \, \wiDelta_0 \cap \wiDelta_1 = \emptyset \, $,  while on the other hand  $ \, \Delta_0 \cap \Delta_1 = \Delta_0 \not= \emptyset \, $  instead.
                                                     \par
   As  $ \, \wiDelta \subseteq \big( \Delta \times \{0,1\} \times I_n \big) \, $,  \, the map  $ \; \pi : \wiDelta \longrightarrow \Delta \cup \{0\} = \Delta \; $  is nothing but the restriction to  $ \wiDelta $  of the projection onto the first factor: then the operator  $ \, - : \wiDelta \longrightarrow \wiDelta \, $  is given by taking the opposite on the left-hand factor.  Finally, the partial operation  $ \, \big(\, \walpha, \wbeta \,\big) \mapsto \walpha + \wbeta \, $  is given by
  $$  \big( \alpha \, , (p \, , i) \big) \, + \, \big( \beta \, , (q \, , j) \big)  \; := \;  \big( \alpha + \beta \, , (\, p + q \, (\hskip-10pt \mod \! 2) \, , i \wedge j \,) \big)  $$
Note that this operation has a neutral element  $ \widetilde{0} \, $,  namely  $ \, \widetilde{0} = (0,(0,1)) \, $.
\end{definition}

\medskip

   We are now ready to give the definition of Chevalley basis.  We invite the reader to look and compare with  Definition \ref{def_che-bas}  that holds for classical Lie superalgebras different from  $ A(1,1) \, $,  $ P(3) $  and  $ Q(n) $  to notice the differences.

\bigskip

\begin{definition} \label{def_che-bas_A-P-Q}
 Let  $ \fg $  be as above.  We call  {\it Chevalley basis\/}  of  $ \fg $  any homogeneous  $ \KK $--basis  $ \; B = {\big\{ H_i \big\}}_{1,\dots,t} \coprod {\big\{ X_{\widetilde{\alpha}} \big\}}_{\widetilde{\alpha} \in \widetilde{\Delta}} \; $  with the following properties.
 \vskip9pt
   \textit{(a)}  \quad   $ \, \big\{ H_1 , \dots , H_\ell \big\} \, $  is a  $ \KK $--basis  of  $ \fh \, $;  moreover, with  $ \, H_\alpha \! \in \fh \, $  as in  \S \ref{root-syst}:
 \vskip4pt
   \centerline{ if  $ \, \fg \not= Q(n) \, $,  \quad  $ \, \fh_\Z  \, := \,  \text{\it Span}_{\,\Z}\big( H_1 , \dots , H_\ell \big)  \, = \,  \text{\it Span}_{\,\Z}\big( \big\{ H_\alpha \,\big|\, \alpha \! \in \! \Delta \! \cap \! (-\Delta) \big\}\big) \; $; }
 \vskip3pt
   \centerline{ if  $ \, \fg = Q(n) \, $,  \quad  $ \, \fh_\Z  :=  \text{\it Span}_{\,\Z}\big( H_1 , \dots , H_\ell \big)  \, = \,  \big\{ h \in \fh \,\big|\, (h \, , h_\alpha) \in \Z \;\; \forall \; \alpha \! \in \! \Delta \big\} \; $; }
 \vskip8pt
   \textit{(b)}  \hskip4pt   $ \big[ H_i \, , H_j \big] = 0 \, ,   \hskip5pt
 \big[ H_i \, , X_{\widetilde{\alpha}} \big] \! = \pi\big(\widetilde{\alpha}\big)\!(H_i) \, X_{\widetilde{\alpha}} \, ,   \hskip11pt  \forall \; i, j \! \in \! \{ 1, \dots, \ell \} \, ,  \; \widetilde{\alpha} \! \in \! \widetilde{\Delta} \, $;
 \vskip11pt
   \textit{(c.$\not{\hskip-3pt{Q}}$)}  \quad   if  $ \, \fg \not= Q(n) \, $,  \, then
 \vskip6pt
   \centerline{ \hskip7pt   $ \big[ X_{\widetilde{\alpha}} \, , \, X_{-\widetilde{\alpha}} \big]  \, = \,  \sigma_{\widetilde{\alpha}} \, H_{\pi(\widetilde{\alpha})}  \hskip25pt  \forall \;\; {\widetilde{\alpha}} \in \widetilde{\Delta} \cap \big( - \widetilde{\Delta} \,\big) \, $, }
 \vskip4pt
\noindent
 with  $ H_{\pi(\widetilde{\alpha})} $  as in  {\it (a)},  and  $ \; \sigma_{\widetilde{\alpha}} := -1 \; $  if  $ \, \widetilde{\alpha} \in \widetilde{\Delta}_1^- \, $,  $ \; \sigma_{\widetilde{\alpha}} := 1 \; $  otherwise;
 \vskip11pt
   \textit{(c.Q)}  \quad   if  $ \, \fg = Q(n) \, $,  \, then
 \vskip6pt
   \centerline{ \hskip7pt   $ \; \big[ X_{(\alpha,(0,1))} \, , X_{(-\alpha,(0,1))} \big]  = \, H_{\alpha}  \hskip25pt  \forall \; \alpha \! \in \! \Delta \!\! \setminus \! \{0\} \, $,  \; with  $ H_\alpha $  as in  {\it (a)}\;; }
 \vskip5pt
   \centerline{ \hskip5pt   $ \; \big[ X_{(\alpha,(1,1))} \, , X_{(-\alpha,(1,1))} \big]  = \,  H_{\alpha'}  \hskip25pt  \forall \; \alpha \! \in \! \Delta \! \setminus \{0\} \, $,  \; with  $ \, H_{\alpha'} \in \fh_\Z \; $; }
 \vskip5pt
   \centerline{ \hskip3pt   $ \big[ X_{(\alpha,(0,1))} \, , X_{(-\alpha,(1,1))} \big]  = \, X_{(0,(1,\alpha))}   \hskip29pt  \forall \; \alpha \! \in \! \Delta \! \setminus \{0\} \, $,   \hskip51pt }
                                                                           \par
   \hfill   with  $ \; X_{(0,(1,\alpha))} := \sum_{k=1}^n e_{\alpha;k} \, X_{(0,(1,k))} \; $  if  $ \; H_\alpha = \sum_{k=1}^n e_{\alpha;k} \, H_k \phantom{\Big|} \, $;
 \vskip5pt
   \centerline{ \hskip3pt   $ \big[ X_{(0,(1,i))} \, , X_{(0,(1,j))} \big] = \, 2 \, H_{i,j}   \hskip15pt  \forall \; i, j = 1, \dots, n \, $,  \, with  $ \; H_{i,j} \in \fh_\Z \; $; }
 \vskip13pt
   \textit{(d)}  \quad  $ \, \big[ X_{\widetilde{\alpha}} \, , \, X_{\widetilde{\beta}} \big]  \, = \, c_{\widetilde{\alpha},\widetilde{\beta}} \; X_{\widetilde{\alpha} + \widetilde{\beta}}  \hskip17pt   \forall \;\, \widetilde{\alpha} , \widetilde{\beta} \in \widetilde{\Delta} \, : \, \widetilde{\alpha} \not= -\widetilde{\beta} \, , \; \widetilde{\beta} \not= -\widetilde{\alpha} \, $,  \, with
 \vskip9pt
   \hskip11pt   \textit{(d.1)} \,  if  $ \;\; \big( \widetilde{\alpha} + \widetilde{\beta} \,\big) \not\in \widetilde{\Delta} \, $,  \; then  $ \;\; c_{\widetilde{\alpha},\widetilde{\beta}} = 0 \; $,  \; and  $ \; X_{\widetilde{\alpha} + \widetilde{\beta}} := 0 \; $,
 \vskip7pt
   \hskip11pt   \textit{(d.2)} \,  if  $ \, \big( \pi(\widetilde{\alpha}), \pi(\widetilde{\alpha}) \big) \not= 0 \, $  or  $ \, \big( \pi(\widetilde{\beta}), \pi(\widetilde{\beta}) \big) \not= 0 \, $,  and (cf.~Definition \ref{alpha-string})  if  $ \; \Sigma^{\pi(\widetilde{\alpha})}_{\pi(\widetilde{\beta})} := \big\{ \pi \! \big( \widetilde{\beta} \,\big) \! - r \, \pi \! \big( \widetilde{\alpha} \big) \, , \, \dots \, , \, \pi \! \big( \widetilde{\beta} \,\big) \! + q \, \pi \! \big( \widetilde{\alpha} \big) \big\} \; $  is the  $ \pi \! \big( \widetilde{\alpha} \big) $--string  through  $ \pi \! \big( \widetilde{\beta} \,\big) \, $,  \, then  $ \; c_{\widetilde{\alpha}, \widetilde{\beta}} = \pm (r+1) \, $,  \, with the following exceptions:
 \vskip3pt
   \textit{(d.2-I)} \,  if  $ \, \fg = P(3) \, $,  and  $ \; \widetilde{\alpha} = [i,j] \, $,  $ \, \widetilde{\beta} = (i,j) \; $   --- notation of  \ref{Deltatilde}{\it (3)}  ---   then  $ \; c_{\widetilde{\alpha},\widetilde{\beta}} = \pm (r+2) \; $;
 \vskip3pt
   \textit{(d.2-II)} \,  if  $ \fg = Q(n) \, $,  and  $ \, \widetilde{\alpha} = \big(0,(1,k)\big) \, $,  $ \, \widetilde{\beta} = \big(\epsilon_i - \epsilon_j,(1,1)\big) \, $   --- using the standard notation for the classical root system of type  $ A_n \, $  ---   then  $ \;\; c_{\widetilde{\alpha}, \widetilde{\beta}} = (\epsilon_i + \epsilon_j)(\alpha_k) \; $;
 \vskip5pt
   \hskip11pt   \textit{(d.3)} \,  if  $ \, \big( \pi(\widetilde{\alpha}), \pi(\widetilde{\alpha}) \!\big) \! = 0 = \big( \pi \! \big(\widetilde{\beta}\,\big), \pi \! \big(\widetilde{\beta}\,\big) \!\big) $,  \, then  $ \; c_{\widetilde{\alpha},\widetilde{\beta}} = \pm \pi \! \big(\widetilde{\beta}\,\big) \big(H_{\pi(\widetilde{\alpha})}\big) \, $.
\end{definition}

\vskip9pt

   Here again, for notational convenience, we shall write  $ \, X_\wdelta := 0 \, $  whenever  $ \, \wdelta \, $  belongs to the  $ \Z $--span  of  $ \widetilde{\Delta} $  but either  $ \, \wdelta \not\in \wiDelta \, $,  or  $ \, \wdelta \in \wiDelta \, $  and  $ \, \pi\big(\,\wdelta\,\big) = 0 \, $.

\medskip

\begin{remarks}   \label{rem-Chev_A-P-Q}  {\ }
 \vskip3pt
   {\it (1)} \,  The  {\it Chevalley superalgebra\/}  (of  $ \fg $)  is defined again just like in the other cases, namely as  $ \, \fg^\Z := \Z\text{\sl --span of \ } B \, $,  where  $ B $  is any Chevalley basis of  $ \fg \, $.  Again, it is a Lie superalgebra over  $ \Z \, $,  independent of the choice of  $ B \, $.
 \vskip3pt
   {\it (2)} \,  If  $ \; \big( \pi(\widetilde{\alpha}), \pi(\widetilde{\alpha}) \big) = 0 = \big( \pi(\widetilde{\beta}), \pi(\widetilde{\beta}) \big) \; $  then  $ \; \pi \big(\widetilde{\beta}\,\big) \! \big(H_{\pi(\widetilde{\alpha})}\big) = \pm (r\!+\!1) \, $.  Therefore,  {\sl condition  {\it (d.3)\/}  in  Definition \ref{def_che-bas_A-P-Q}  reads just like  {\it (d.2)}.}
\end{remarks}

\bigskip

   Now we show how to prove the  {\bf existence\/}  of Chevalley bases   --- as claimed in  Theorem \ref{exist_che-bas}  ---   in the present cases.  We proceed by direct construction of explicit bases; on the other hand, it is worth stressing that the ``uniform argument'' sketched in  Remark \ref{unif_exist_Chev-bases}  does apply again to case  $ A(1,1) $,  {\sl and even to case}  $ Q(n) $,  up to a few, obvious changes.

\medskip

   $ \underline{A(1,1) \, , \, P(3)} \phantom{\bigg|} \; $:  \, In case  $ \, A(1,1) \, $,  the description of an explicit Chevalley basis given in the proof of  Theorem \ref{exist_che-bas}  for  $ A(n,n) $   ---  $ n \not= 1 $  ---  works again,  {\it verbatim}.  Similarly the description of a Chevalley basis for  $ P(n) $   ---  $ n \not= 3 $  ---  applies again to  $ \, P(3) \, $,  up to reading (notation as in  Definition \ref{Deltatilde}{\it (3)\/})
  $$  \alpha_{i,j} := (i,j)  \quad ,  \qquad  \beta_{h,k} := [h,k]  \quad ,
\qquad  \beta_{q,p} := [q,p]  $$
for  $ \, 1 \leq i \not= j \leq 4 \, $,  $ \, 1 \leq h < k \leq 4 \, $,
$ \, 1 \leq p < q \leq 4 \, $.
%
%

\medskip

   $ \underline{Q(n)} \phantom{\bigg|} \; $:  \, In this case a Chevalley basis is a variation of the basis given in  \cite{fss},  \S 2.49 (we change the Cartan generators); we now describe it explicitly.
                                                                               \par
   First consider the Lie superalgebra (sub-superalgebra of  $ \, \rgl(n\!+\!1|\,n\!+\!1) \, $)
  $$  \widetilde{Q}(n)  \; := \,\;  \bigg\{\! \begin{pmatrix}
                         A  &  B  \\
                         B  &  A
\end{pmatrix}
 \in \rgl(n\!+\!1|\,n\!+\!1) \;\bigg|\;\, A \in \rgl(n\!+\!1) \, , \, B \in \rsl(n\!+\!1) \;\bigg\}  $$
such that, by definition,  $ \; Q(n) := \widetilde{Q}(n) \Big/ \KK I_{2(n+1)} \; $.  Take in  $ \widetilde{Q}(n) $ the elements
  $$  \displaylines{
   L_i := \e_{i,i} + \e_{i+n+1,i+n+1} \, ,  \;\;  K_t := \e_{t,t+n+1} - \e_{t+1,t+n+2} + \e_{t+n+1,t} - \e_{t+n+2,t+1}  \cr
   E_{i,j} := \e_{i,j} + \e_{i+n+1,j+n+1} \; ,  \hskip57pt  F_{i,j} := \e_{i,j+n+1} + \e_{i+n+1,j}  }  $$
for all  $ \, i, j = 1, \dots, n+1 \, $,  $ \, t = 1, \dots, n \, $,  $ \, i \not= j \, $.  Then
  $$  {\big\{ L_i \big\}}_{i=1,\dots,n+1}  \;{\textstyle \bigcup}\;  {\big\{ E_{i,j} \big\}}_{i,j=1,\dots,n+1}^{i \not= j}  \;{\textstyle \bigcup}\;  {\big\{ K_t \big\}}_{t=1,\dots,n}  \;{\textstyle \bigcup}\;  {\big\{ F_{i,j} \big\}}_{i,j=1,\dots,n+1}^{i \not= j}  $$
is a homogeneous  $ \KK $--basis  of  $ \widetilde{Q}(n) \, $,  the
$ L_i $'s  and the  $ E_{i,j} $'s  being even, the $ K_t $'s  and
the  $ F_{i,j} $'s  being odd.  As  $ \, I_{2(n+1)} = L_1 + \cdots
+ L_{n+1} \, $,  the quotient Lie superalgebra  $ \; Q(n) :=
\widetilde{Q}(n) \Big/ \KK I_{2(n+1)} \; $  has homogeneous  $ \KK $--basis
  $$  B  \; := \;  {\big\{ L_i \big\}}_{i=1,\dots,n}  \;{\textstyle \bigcup}\;  {\big\{ E_{i,j} \big\}}_{i,j=1,\dots,n+1}^{i \not= j}  \;{\textstyle \bigcup}\;  {\big\{ K_t \big\}}_{t=1,\dots,n}  \;{\textstyle \bigcup}\;  {\big\{ F_{i,j} \big\}}_{i,j=1,\dots,n+1}^{i \not= j}  $$
where we use again the same symbols to denote the images of
elements of  $ \widetilde{Q}(n) $  inside  $ Q(n) \, $.  In terms
of these, the multiplication table of  $ Q(n) $  reads
  $$  \displaylines{
   \big[ L_i , L_j \big]  \; = \;  0 \;\; ,  \quad  \big[ L_i , K_t \big]  \; = \;  0  \cr
   \big[ K_r \, , K_s \big]  \; = \;  2 \, \big( \delta_{r,s} - \delta_{r,s+1} \big) \, L_r \, + \, 2 \, \big( \delta_{r,s} - \delta_{r+1,s} \big) \, L_{r+1}  \cr
   \big[ L_k \, , E_{i,j} \big]  \; = \;  \alpha_{i,j}(L_k) \, E_{i,j} \;\; ,
 \qquad  \big[ L_k \, , F_{i,j} \big]  \; = \;  \alpha_{i,j}(L_k) \, F_{i,j}  \cr
   \big[ K_t \, , E_{i,j} \big]  \; = \;  \alpha_{i,j}(L_t \! - \! L_{t+1}) \, F_{i,j} \;\; ,
 \qquad  \big[ K_t \, , F_{i,j} \big]  \; = \;  \widetilde{\alpha}_{i,j}(L_t \! - \! L_{t+1}) \, E_{i,j}  \cr
   \hfill   \qquad   \big[ E_{i,j} \, , E_{k,\ell} \big]  \; = \;  \delta_{j,k} \, E_{i,\ell} \, - \, \delta_{\ell,i} \, E_{k,j}   \hfill   \forall \;\; (i,j) \not= (\ell,k)  \cr
   \hfill   \qquad   \big[ E_{i,j} \, , F_{k,\ell} \big]  \; = \;  \delta_{j,k} \, F_{i,\ell} \, - \, \delta_{\ell,i} \, F_{k,j}   \hfill   \forall \;\; (i,j) \not= (\ell,k)  \cr
   \hfill   \qquad   \big[ F_{i,j} \, , F_{k,\ell} \big]  \; = \;  \delta_{j,k} \, E_{i,\ell} \, + \, \delta_{\ell,i} \, E_{k,j}   \hfill   \forall \;\; (i,j) \not= (\ell,k)  \cr
   \big[ E_{i,j} \, , E_{j,i} \big]  \, = \,  L_i - L_j \;\, ,
 \quad  \big[ E_{i,j} \, , F_{j,i} \big]  \, = \,  {\textstyle \sum_{t=1}^{j-1}} \, K_t \;\, ,
 \quad  \big[ F_{i,j} \, , F_{j,i} \big]  \, = \,  L_i + L_j  }  $$
where the  $ \alpha_{i,j} $'s  are the non-zero roots of  $ Q(n)
\, $,  forming the classical root system of type  $ A_n $  (namely
$ \, \alpha_{i,j} = \epsilon_i - \epsilon_j \, $,  where the  $
\epsilon_\ell $'s  form the dual basis to the canonical basis of
the diagonal matrices in  $ \rgl(n+1) \, $),  while  $ \,
\widetilde{\alpha}_{i,j} := \epsilon_i + \epsilon_j \, $.
 \vskip5pt
   In particular, this shows that the  $ L_i $'s  ($ \, i = 1, \dots, n \, $)  form a  $ \KK $--basis  of the Cartan subalgebra of  $ Q(n) $   --- which is the image in  $ Q(n) $  of the subspace of diagonal matrices in  $ \widetilde{Q}(n) $  ---   each  $ E_{i,j} \, $,  resp.~each  $ F_{i,j} \, $,  is a root vector (the former being even, the latter odd) for the root  $ \alpha_{i,j} \, $,  and the  $ K_t $'s  form a  $ \KK $--basis of the (totally odd) zero root space, namely  $ \, \fg_{\alpha=0} \cap \fg_1 \, $.  Now set
  $$  H_k \, := \, L_k \;\; ,   \hskip11pt   X_{(0,(1,k))} \, := \, K_k \;\; ,
\hskip11pt   X_{(\alpha_{i,j},(0,1))} \, := \, E_{i,j} \;\; ,   \hskip11pt
X_{(\alpha_{i,j},(1,1))} \, := \, F_{i,j}  $$
for all  $ \; k = 1, \dots, n \; $  and  $ \; i, j = 1, \dots, n\!+\!1 \; $  with  $ \; i \not= j \; $.  Then the above formulas eventually show that  $ \; B := {\big\{ H_k \big\}}_{k = 1, \dots, n} \coprod {\big\{ X_{\widetilde{\alpha}} \big\}}_{\widetilde{\alpha} \in \widetilde{\Delta}} \; $  is a Chevalley basis of  $ \, \fg = Q(n) \, $  in the sense of  Definition \ref{def_che-bas},  by direct check.

\bigskip

\begin{remark}
   For  $ \fg $  of type  $ Q(n) \, $   --- to be precise, of type  $ \widetilde{Q}(n) $  ---   a Chevalley basis was given also in  \cite{bkl},  Lemma 4.3.
\end{remark}

\bigskip
\bigskip
\bigskip

 \section{Kostant superalgebras}  \label{kost-superalgebra_A-P-Q}

\medskip

   In  \S \ref{kost-form},  the Kostant's superalgebra  $ K_\Z(\fg) $  was defined as the subalgebra of  $ U(\fg) $  generated by divided powers of the root vectors attached to even roots, root vectors attached to odd roots, and binomial coefficients in the elements of the Chevalley basis which belong to  $ \fh \, $.
%
%
 We perform exactly the same construction for  $ \, A(1,1) \, $,  $ \, P(3) \, $  and  $ \, Q(n) \, $:  thus the definition of  $ K_\Z(\fg) $  for the present cases reads essentially like  Definition \ref{def-kost-superalgebra},  the only difference is that root vectors are indexed by elements of  $ \widetilde{\Delta} $  instead of  $ \Delta \, $.

\medskip

   The  {\sl commutation rules\/}  among generators of  $ K_\Z(\fg) $  are very close to those in  \S \ref{comm-rul_kost}.  Nevertheless, some differences occur, which we now point out.

 \vskip13pt
\noindent
 {\bf (1) Even generators only:}
%
%
 These relations involve only the  $ H_i $'s  and their binomial coefficients, and root vectors relative to  {\sl even\/}  roots.  Then they are just the same as in  \S \ref{comm-rul_kost},  just reading  $ {\widetilde{\Delta}}_0 $  instead of  $ \Delta_0 \, $,  $ \, X_{\pm\widetilde{\alpha}} \, $  instead of  $ \, X_{\pm\alpha} \, $,  $ \, X_{\widetilde{\beta}} \, $  instead of  $ \, X_\beta \, $,  $ \, \pi\big(\widetilde{\alpha}\big)(H) \, $  instead of  $ \, \alpha(H) \, $  and  $ \, H_{\pi(\widetilde{\alpha})} \, $  instead of  $ \, H_\alpha \, $.

\vskip13pt

\noindent
 {\bf (2) Odd and even generators  {\rm  (also involving the  $ X_{\widetilde{\gamma}} $'s,  $ \, \widetilde{\gamma} \in {\widetilde{\Delta}}_1 \, $)}:}
 \vskip-11pt
  $$  \displaylines{
   \hfill   X_{\widetilde{\gamma}} \, f(H)  \; = \;  f\big(H - \pi\big(\widetilde{\gamma}\big)\!(H)\,\big) \, X_{\widetilde{\gamma}}   \qquad   \hfill \phantom{(4.6)}  \cr
   \hfill   \phantom{{}_{\big|}} \forall \;\; \widetilde{\gamma} \in {\widetilde{\Delta}}_1 \, , \; h \in \fh \, , \; f(T) \in \KK[T]   \qquad \qquad  \cr
%
%
   \hfill   X_{-\widetilde{\gamma}} \, X_{\widetilde{\gamma}}  \; = \;  - X_{\widetilde{\gamma}} \, X_{-\widetilde{\gamma}} \, + \, H_{\widetilde{\gamma}}   \hskip29pt  \forall \;\, \widetilde{\gamma} \in {\widetilde{\Delta}}_1 \cap \big(\! -{\widetilde{\Delta}}_1 \big)   \hfill \phantom{(4.10)}  \cr
   \text{with}  \quad  H_{\widetilde{\gamma}} := \big[ X_{\widetilde{\gamma}} \, , \, X_{-\widetilde{\gamma}} \big] \, \in \, \fh_\Z \; ,  \phantom{{}_{\big|}}  \cr
   \hfill   X_{\widetilde{\gamma}} \, X_{\widetilde{\delta}}  \; = \;  - X_{\widetilde{\delta}} \, X_{\widetilde{\gamma}} \, + \, c_{\widetilde{\gamma},\widetilde{\delta}} \, X_{\widetilde{\gamma} + \widetilde{\delta}} \;\, ,   \hfill   \;\; \phantom{{}_{\big|}} \forall \;\, \widetilde{\gamma} , \widetilde{\delta} \in {\widetilde{\Delta}}_1 \, , \; \pi \! \big( \widetilde{\gamma} \,\big) \! + \pi \! \big(\, \widetilde{\delta} \,\big) \not= 0   \hskip11pt \phantom{(4.11)}  \cr
   \text{with  $ \, c_{\widetilde{\gamma},\widetilde{\delta}} \, $  as in  Definition \ref{def_che-bas_A-P-Q},}  \phantom{{}_{\big|}}  \cr
   \hfill   X_{(\alpha,(0,1))} \, X_{(-\alpha,(1,1))}  \, = \,  X_{(-\alpha,(1,1))} \, X_{(\alpha,(0,1))} \, + \, X_{(0,(1,\alpha))} \;\, ,   \hfill   \forall \; \alpha \in \Delta  \hskip9pt  \phantom{\big|}   \phantom{(4.12)}  \cr
   \text{with  $ \, X_{(0,(1,\alpha))} := {\textstyle \sum_{k=1}^n} e_{\alpha;k} \, X_{(0,(1,k))} \; $  as in  Definition \ref{def_che-bas_A-P-Q}{\it (c.Q)},}  \phantom{{}_{\big|}}  \cr
   \hfill   X_{(0,(1,i))} \, X_{(0,(1,j))}  \; = \;  - X_{(0,(1,j))} \, X_{(0,(1,i))} \, + \, 2 \, H_{i,j} \;\, ,   \hfill   \forall \; i \, , j  \hskip9pt  \phantom{\big|}  \phantom{(4.13)}  \cr
   \text{with  $ \; H_{i,j} := \big[ X_{(0,(1,i))} \, , \, X_{(0,(1,j))} \big] \, \in \, \fh_\Z \; $  as in  Definition \ref{def_che-bas}{\it (c.Q)},}  \phantom{{}_{\big|}}  \cr
   \hfill   X_{\widetilde{\alpha}}^{(n)} \, X_{\widetilde{\gamma}}  \; = \;  X_{\widetilde{\gamma}} \, X_{\widetilde{\alpha}}^{(n)} \, + \, {\textstyle \sum_{k=1}^n} \, \Big( {\textstyle \prod_{s=1}^k} \, \varepsilon_s \Big) \, {\textstyle \Big(\! {{r \, + \, k} \atop k} \!\Big)} \, X_{\widetilde{\gamma} + k \, \widetilde{\alpha}} \, X_{\widetilde{\alpha}}^{(n-k)}   \hfill \phantom{(4.14)}  \cr
   \hfill   \forall \;\; n \in \N \, ,  \;\;\; \forall \;\; \widetilde{\alpha} \in {\widetilde{\Delta}}_0 \, , \; \widetilde{\gamma} \in {\widetilde{\Delta}}_1  \, : \;  \widetilde{\alpha} \not= \pm 2 \, \widetilde{\gamma} \, ,  \;  \widetilde{\alpha} \not= (0,(1,i)) \, , \, \text{for any  $ i \, $},   \phantom{{}_{\big|}}  \quad  \cr
   \text{with} \hskip9pt  \sigma^{\widetilde{\alpha}}_{\widetilde{\gamma}} = \big\{ \widetilde{\gamma} - r \, \widetilde{\alpha} \, , \dots , \widetilde{\gamma} \, , \dots , \widetilde{\gamma} + q \, \widetilde{\alpha} \,\big\} \; ,  \hskip7pt  X_{\widetilde{\gamma} + k \, \widetilde{\alpha}} := 0  \,\;\text{\ if\ }\;  \big(\, \widetilde{\gamma} \! + \! k \, \widetilde{\alpha} \,\big) \not\in \widetilde{\Delta} \, ,   \hfill  \cr
   \text{and}  \hskip11pt   \varepsilon_s = \pm 1  \text{\;\;\ such that \ }  \big[ X_\alpha \, , X_{\gamma + (s-1) \, \alpha} \big] = \varepsilon_s \, (r+s) \, X_{\gamma + s \, \alpha} \;\; ,  \phantom{{}_{\big|}}  \hfill  \cr
   \;   X_{(\alpha,(0,1))}^{(n)} \, X_{(0,(1,k))}  \, = \,  X_{(0,(1,k))} \, X_{(\alpha,(0,1))}^{(n)} - \, \alpha(H_k) \, X_{(\alpha,(0,1))}^{(n-1)} \, X_{(\alpha,(1,1))}   \phantom{{}_{\big|}}   \hfill \phantom{(4.15)}  \cr
   \hfill   \forall \;\; n \in \N \, ,  \;\;\; \forall \;\; \alpha \in \Delta_0 \, ,  \;\;\; \forall \;\; k \, ,   \phantom{{}_{\big|}}  \quad  \cr
   \hfill   X_{\widetilde{\gamma}} \, X_{\widetilde{\alpha}}^{(n)}  \; = \;  X_{\widetilde{\alpha}}^{(n)} \, X_{\widetilde{\gamma}} \;\; ,  \qquad  X_{-\widetilde{\gamma}} \, X_{-\widetilde{\alpha}}^{\,(n)}  \; = \;  X_{-\widetilde{\alpha}}^{\,(n)} \, X_{-\widetilde{\gamma}}   \hskip51pt   \hfill \phantom{(4.16)}  \cr
   \hfill   X_{-\widetilde{\gamma}} \, X_{\widetilde{\alpha}}^{(n)}  \; = \;  X_{\widetilde{\alpha}}^{(n)} \, X_{-\widetilde{\gamma}} \, + \, z_{\widetilde{\gamma}} \; \pi\!\big(\widetilde{\gamma}\big)\!(H_{\widetilde{\gamma}}) \; X_{\widetilde{\alpha}}^{(n-1)} \, X_{\widetilde{\gamma}}   \hskip51pt   \hfill \phantom{(4.17)}  \cr
   \hfill   X_{\widetilde{\gamma}} \, X_{-\widetilde{\alpha}}^{\,(n)}  \; = \;  X_{-\widetilde{\alpha}}^{\,(n)} \, X_{\widetilde{\gamma}} \, - \, z_{\widetilde{\gamma}} \; \pi\!\big(\widetilde{\gamma}\big)\!(H_{\widetilde{\gamma}}) \; X_{-\widetilde{\alpha}}^{\,(n-1)} \, X_{-\widetilde{\gamma}}   \hskip51pt \phantom{{}_{|}}   \hfill \phantom{(4.18)}  \cr
   \hfill   \forall \;\; n \in \N \; ,  \;\;\; \forall \;\; \widetilde{\gamma} \in {\widetilde{\Delta}}_1 \; , \;\; \widetilde{\alpha} = 2 \, \widetilde{\gamma} \in {\widetilde{\Delta}}_0 \; , \;\; z_{\widetilde{\gamma}} := c_{\widetilde{\gamma},\widetilde{\gamma}} / 2 = \pm 2   \qquad  }  $$

\medskip

   Using these relations, one proves the  {\sl PBW-like theorem for  $ K_\Z(\fg) $},  {i.e.}  Theorem \ref{PBW-Kost},  with the same arguments as in the other cases.  What changes is only the statement, as root vectors are now indexed by elements of  $ \widetilde{\Delta} \, $.
                                                                           \par
   A similar comment applies to the  Corollary \ref{kost_tens-splitting}  and the Remarks after it.

\bigskip
\bigskip
\bigskip

 \section{Chevalley supergroups and their properties}  \label{const-che-sgroup_A-P-Q}

\medskip

   The construction of Chevalley supergroups of types  $ A(1,1) \, $,  $ P(3) $  and  $ Q(n) $  follows step by step that of other cases in  \S \ref{che-sgroup}.  Like for the previous steps, one essentially has only to change root vectors indexed by elements of  $ \Delta $  with root vectors indexed by  $ \widetilde{\Delta} \, $.
                                                                      \par
   The ingredients and the strategy are exactly the same, in particular one fixes a Chevalley basis to start with.  The second ingredient is the notion of admissible lattices: their definition, existence and description of their stabilizers are dealt with just like in  \S \ref{adm-lat}.

\medskip

   Using an admissible lattice, we define supergroup functors  $ \, x_{\widetilde{\delta}} \, $,  $ \, h_H \, $  and  $ h_i \, $,  associated to each  $ \, \widetilde{\delta} \in \widetilde{\Delta} \, $,  $ \, H \in \fh_\Z \, $  and  $ \, i = 1, \dots, \ell \, $,  just like in  Definition \ref{chevalley}.  The analysis carried on about such objects in  \S \ref{const-che-sgroup}   --- in particular,  Proposition \ref{hopf-alg}  ---   extends to the present context too.  Then we introduce the direct analogue of  Definition \ref{def_Che-sgroup_funct},  where the  $ x_{\widetilde{\delta}} \, $'s  replace the  $ x_\delta $'s,  thus getting the notions of Chevalley supergroup functor  $ G $  and Chevalley supergroup  $ \bG \, $.  Similarly, all definitions and considerations about  $ G_0 $  and  $ \bG_0 $   --- the latter being (almost) a classical Chevalley-like algebraic group ---   also extend to the present case:
 in particular, their relations with  $ G $  and  $ \bG $  are the same as before.

\medskip

   To extend the construction and analysis carried on in  \S \ref{che-sgroup_alg},  we can repeat the same definitions, up to replacing  $ \Delta $  with  $ \widetilde{\Delta} \, $,  $ \Delta^\pm $  with  $ \widetilde{\Delta}^\pm $,  $ \alpha $  with  $ \widetilde{\alpha} \, $,  etc.  Thus we have subsets  $ G_1^{\pm,<}(A) $  and subgroups  $ G_1(A) $,  $ G_0^\pm(A) $,  $ G_1^\pm(A) $  and  $ G^\pm(A) $  of  $ G(A) \, $,  and similarly with  ``$ \, \bG \, $''  instead of  ``$ \, G \, $''.

\medskip

   The first modification we have to do is in  Lemma \ref{comm_1-pssg},  which now reads

\bigskip

\begin{lemma}  \label{comm_1-pssg_A-P-Q}  {\ }
 \vskip7pt
\noindent
 (a) \,  Let  $ \; \widetilde{\alpha} = \big(\alpha,(0,1)\big) \in {\widetilde{\Delta}}_0 \, $,  $ \, \widetilde{\gamma} \in {\widetilde{\Delta}}_1 \, $,  $ \, A \in \salg \, $  and  $ \, t \in A_0 \, $,  $ \, \vartheta \in A_1 \, $.
 \vskip4pt
   If  $ \; \widetilde{\gamma} \not\in \big\{ \big(-\!\alpha,(1,1)\big) \, , \, \big(0,(1,i)\big) \big\} \, $,
%
%
 \, there exist  $ \, c_s \! \in \! \Z \, $  such that
  $$  \big( x_{\widetilde{\gamma}}(\vartheta) \, , \, x_{\widetilde{\alpha}}(t) \big)  \, = \,
{\textstyle \prod_{s>0}} \, x_{\widetilde{\gamma} + s \, \widetilde{\alpha}}\big(c_s \, t^s \vartheta\big)  \;\; \in \;\;  G_1(A)  $$
(the product being finite).  More precisely  (cf.~\S \ref{kost-superalgebra_A-P-Q}  for the notation),
  $$  \big( 1 + \vartheta \, X_{\widetilde{\gamma}} \, , \, x_{\widetilde{\alpha}}(t) \big)  \, = \,
{\textstyle \prod_{s>0}} \, \Big( 1 + {\textstyle \prod_{k=1}^s \varepsilon_k \cdot {{s+r} \choose r}} \cdot t^s \vartheta \, X_{\widetilde{\gamma} + s \, \widetilde{\alpha}} \Big)  $$
where the factors in the product are taken in any order (as they
do commute).
 \vskip4pt
   If  $ \,\; \widetilde{\gamma} = -\big(\alpha,(1,1)\big) \, $,  then   --- with notation of  Definition \ref{def_che-bas_A-P-Q}(c.Q)  ---
  $$  \displaylines{
   \qquad  \big( x_{\widetilde{\gamma}}(\vartheta) \, , \, x_{\widetilde{\alpha}}(t) \big)  \, = \,
x_{(0,(1,\alpha)}(- t \, \vartheta)  \, := \,  \big(\, 1 - t \, \vartheta \, X_{(0,(1,\alpha)}\,\big)  \, =   \hfill  \cr
   \hfill   = \,  {\textstyle \prod_k} \big( 1 - e_{\alpha;k} \, t \, \vartheta \, X_{(0,(1,k))} \big)  \, = \,  {\textstyle \prod_k} \, x_{(0,(1,k))} (-e_{\alpha;k} \, t \, \vartheta)  \;\; \in \;\;  G_1(A)  }  $$
where the factors in the product are taken in any order (as they
do commute).
 \vskip4pt
   If instead  $ \; \widetilde{\gamma} = \big(0,(1,i)\big) \, $  for some  $ i \, $,  then
  $$  \big( x_{\widetilde{\gamma}}(\vartheta) \, , \, x_{\widetilde{\alpha}}(t) \big)  \, = \;
x_{\widetilde{\gamma} + \widetilde{\alpha}}\big( \alpha(H_i) \, t \, \vartheta \big)  \, = \,
\big(\, 1 + \alpha(H_i) \, t \, \vartheta \, X_{(\alpha,(1,1))} \big)  \;\; \in \;\;  G_1(A)  $$
 \vskip5pt
\noindent
 (b)  Let  $ \; \widetilde{\gamma} , \widetilde{\delta} \! \in \! {\widetilde{\Delta}}_1 \, $,  $ \, A \! \in \! \salg \, $,  $ \, \vartheta, \eta \! \in \! A_1 \, $.  Then (notation of  Definition \ref{def_che-bas_A-P-Q})
  $$  \big( x_{\widetilde{\gamma}}(\vartheta) \, , \, x_{\widetilde{\delta}}(\eta) \big)  \; = \;\,  x_{\widetilde{\gamma} + \widetilde{\delta}}\big(\! - \! c_{\widetilde{\gamma},\widetilde{\delta}} \; \vartheta \, \eta \big)  \; = \;  \big(\, 1 \! - \! c_{\widetilde{\gamma},\widetilde{\delta}} \; \vartheta \, \eta \, X_{\widetilde{\gamma} + \widetilde{\delta}} \,\big)  \;\; \in \;\;  G_0(A)  $$
if  $ \; \pi \!\big( \widetilde{\gamma} \big) \! + \pi \!\big(\,
\widetilde{\delta} \,\big) \not= 0 \; $;  otherwise, for  $ \;
\widetilde{\gamma} \! = \! \big(\gamma,\!(1,\!1)\big) \, $,  $ \,
\widetilde{\delta} \! = \! \big(\!-\!\gamma,\!(1,\!1)\big) \! =:
\! -\widetilde{\gamma} \, $,
  $$  \big( x_{\widetilde{\gamma}}(\vartheta) \, , \, x_{-\widetilde{\gamma}}(\eta) \big)  \; = \;  \big(\, 1 \! - \vartheta \, \eta \, H_{\widetilde{\gamma}} \,\big)  \; = \;  h_{H_{\widetilde{\gamma}}}\big( 1 \! - \vartheta \, \eta \big)  \;\; \in \;\;  G_0(A)  $$
and eventually, for  $ \; \widetilde{\gamma} = \big(0,(1,i)\big)
\, $,  $ \; \widetilde{\delta} = \big(0,(1,j)\big) \, $,
  $$  \big( x_{(0,(1,i))}(\vartheta) \, , \, x_{(0,(1,j))}(\eta) \big)  \; = \;  \big(\, 1 \! - \! 2 \; \vartheta \, \eta \, H_{i,j} \,\big)  \; =: \;  h_{\alpha_{{}_{H_{i,j}}}}\!\!\big( 1 \! - \! 2 \; \vartheta \, \eta \big)  \,\; \in \;\,  G_0(A)  $$
with  $ \, \alpha_{{}_{H_{i,j}}} \!\! \in \fh^* \,$  corresponding
to  $ \, H_{i,j} \in \fh \, $   --- notation of  \S \ref{kost-superalgebra_A-P-Q}.

\vskip9pt

\noindent
 (c)  Let  $ \; \widetilde{\alpha} , \widetilde{\beta} \in \widetilde{\Delta} \, $,  $ \, A \in \salg \, $,  $ \, t \in U(A_0) \, $,  $ \, \bu \in A_0 \! \times \! A_1 = A \, $.  Then
  $$  \hskip31pt   h_{\widetilde{\alpha}}(t) \; x_{\widetilde{\beta}}(\bu) \; {h_{\widetilde{\alpha}}(t)}^{-1}  \; = \;  x_{\widetilde{\beta}}\big( t^{\pi(\widetilde{\beta})(H_{\pi(\widetilde{\alpha})})} \, \bu \big)  \;\; \in \;\;  G_{p(\widetilde{\beta})}(A)  $$
where  $ \; p(\widetilde{\beta}) := r \, $,  by definition, if and only if
$ \, \widetilde{\beta} \in {\widetilde{\Delta}}_r \; $.
\end{lemma}

\bigskip

   The proof of the above follows right the same arguments as before.

%
%
%

\medskip

   Then all results from  Theorem \ref{g0g1}  to  Proposition \ref{semidirect}  extend to the present case, both for statements and proofs   --- more or less  {\it verbatim\/}  indeed.  In particular, we have factorizations  $ \; G \, \cong \, G_0 \times G_1^< \; $  and  $ \; \bG \, \cong \, \bG_0 \times \bG_1^< \, $ --- as well as the ``big cell''-type ones ---   the latter implying that the group functor  $ \bG $  is  {\sl representable},  thus it is an algebraic supergroup.

\medskip

   Finally,  {\sl all the content of  \S \ref{ind_che-kost_alg}  and  \S \ref{che-sgroup_LieTT}  extends to the present context},  without any change.  This means that all our construction still are independent of specific choices, and that the algebraic supergroups thus obtained do have the original Lie superalgebras as their tangent Lie superalgebras.

%

\appendix
%
%

 %
 %

\chapter{Sheafification}

   In this Appendix we discuss the concept of sheafification of a functor in supergeometry.  Most of this material is known or easily derived from known facts.  We include it here for completeness and lack of an appropriate reference.

\medskip

   Hereafter we shall make a distinction between a superscheme  $ X $  and its functor of points, that we shall denote by  $ h_X $  or, if  $ \, X = \uspec(A) \, $,  by  $ h_A \, $.

\medskip

   We start by defining {\sl local} and {\sl sheaf} functors.
For their definitions in the classical setting see for example
\cite{dg},  pg.~16, or  \cite{eh},  ch.~VI.

\begin{definition} \label{local}
 Let  $ \, F : \salg \lra \sets \, $  be a functor.  Fix  $ \, A \in \salg \, $.
Let  $ \, \{f_i\}_{i \in I}
%
%
 \subseteq \!
A_0 \, $,  $ \, \big({\{f_i\}}_{i \in I}\big) = A_0 \, $  and let
$ \, \phi_i \! : \! A \! \ra \! A_{f_i} \, $,  $ \, \phi_{ij} \! :
\! A_{f_i} \! \ra \! A_{f_i{}f_j} \, $  be the natural morphism,
where  $ \, A_{f_i} := A\big[f_i^{-1}\big] \, $.  We say that  $ F
$  is  {\it local\/}  if for any  $ \, A \in \salg \, $  for any
family  $ \, \{\alpha_i\}_{i \in I} \, $,  $ \, \alpha_i \in
F(A_{f_i}) \, $,  such that  $ \, F(\phi_{ij})(\alpha_i) =
F(\phi_{ji})(\alpha_j) \, $  for all  $ i $  and  $ j $,  there
exists a unique  $ \, \alpha \in F(A) \, $  such that  $ \,
F(\phi_i)(\alpha) = \alpha_i \, $  for all possible families  $
\{f_i\}_{i \in I} $  described above.
\end{definition}

   We want to rewrite this definition in more geometric terms in order to show that this is essentially the gluing property appearing in the usual definition of sheaf on a topological space.
                                                            \par
   We first observe that there is a contravariant equivalence of categories between the category of commutative superalgebras  $ \salg $  and the category of affine superschemes  $ \aschemes $,  i.e.~those superschemes that are the spectrum of some superalgebra (see Section \ref{preliminaries}  for more details).  The equivalence is realized by  $ \, A \mapsto \uspec(A) \, $  and it is explained in full details in  \cite{cf},  Observation 5.1.6. Hence a functor  $ \, F : \salg \lra \sets \, $  can also be equivalenty regarded as a functor  $ \, F : \aschemes^\circ \lra \sets \, $.  With an abuse of notation we shall use the same letter to denote both functors.

\medskip

   Let  $ F $  be a local functor, regarded as  $ \, F : \aschemes^\circ \lra \sets \, $,  and let  $ F_A $  be its restriction to the affine open subschemes of  $ \uspec(A) \, $.  Then  $ F_A $  is a sheaf in the usual sense; we must just forget the subscheme structure of the affine subschemes of  $ \uspec(A) $  and treat them as open sets in the topological space  $ \spec(A) $,  their morphisms being the inclusions.  Then  $ F_A $  being a functor means that it is a presheaf in the Zariski topology, while the property detailed in  Definition \ref{local}  ensures the gluing of any family of local sections which agree on the intersection of any two parts of an open covering.

\medskip

   The most interesting   --- for us ---   example of local functor is the following:

\smallskip

\begin{proposition}
   (see  \cite{cf},  Proposition 5.3.5)  If  $ X $  is a superscheme, its functor of points  $ \; \salg \,{\buildrel {h_X} \over {\lra}}\, \sets \, $,  $ \, A \mapsto h_X(A) := \Hom \big(\, \uspec(A) \, , X \,\big) \, $,  \, is local.
\end{proposition}

%
%
%
%

\medskip

   We now turn to the following problem.  If we have a presheaf  $ \cF $  on a topological space in the ordinary sense, we can always build its sheafification, which is a sheaf  $ \widetilde{\cF} $  together
with a sheaf morphism  $ \, \alpha : \cF \lra \widetilde{\cF} \,
$.  This is the (unique) sheaf, which is locally isomorphic to the
given presheaf and has the following universal property: any
presheaf morphism $ \, \phi : \cF \lra {\mathcal{G}} \, $,  with
$ {\mathcal{G}} $  a sheaf, factors via  $ \alpha $  (for more
details on this construction,  see \cite{ha},  Ch.~II).  We now
want to give the same construction in our more general setting.

\smallskip

   The existence of sheafification of a functor from the category of algebras to the category of sets is granted in the ordinary case by  \cite{dg},  ch.~I, \S 1, no.~4, which is also nicely summarized in  \cite{dg},  ch.~III, \S 1, no.~3.  The proof is quite formal and one can carry it to the supergeometric setting.  We however prefer to introduce Grothendieck topologies and the concept of  {\sl site\/}  and to construct the sheafification of a functor from  $ \salg $  to  $ \sets $  through them.  In fact, as we shall see, very remarkably Grothendieck's treatment is general enough to comprehend supergeometry.  For more details one can refer to  \cite{gr}  and  \cite{vst}.

\medskip

\begin{definition}
   We call a category  $ \cC $  a  {\it site\/}  if it has a  {\it Grothendieck topology},  i.e.~to every object  $ \, U \in \cC \, $  we associate a collection of so-called  {\it coverings\/}  of  $ U $,  i.e.~sets of arrows  $ \{U_i \lra U\} $,  such that:
                                             \par
 \noindent   \; {\it i)} \,  If  $ \, V \lra U \, $  is an isomorphism, then the set  $ \, \{V \lra U\} \, $  is a covering;
                                             \par
 \noindent   \; {\it ii)} \,  If  $ \, \{U_i \lra U\} \, $  is a covering and  $ \, V \lra U \, $  is any arrow, then the fibered products  $ \, \{U_i \times_U V\} \, $  exist and the collection of projections  $ \, \{U_i \times_U V\} \lra V \, $  is a covering;
                                             \par
 \noindent   \; {\it iii)} \,  If  $ \, \{U_i \lra V\} \, $  is a covering and for each index  $ i $  we have a covering  $ \, \{V_{ij} \lra U_i\} \, $,  then the collection  $ \, \{V_{ij} \lra U_i \lra U\} \, $  is a covering of  $ U \, $.
\end{definition}

\smallskip

   One can check that  $ \salg $,  $ \aschemes $  and their ordinary correspondents are sites (for the existence of fibered products in such categories see  \cite{cf}, ch.~5).

\smallskip

\begin{definition}
   Let  $ \cC $  be a site.  A functor  $ \, F : \cC^\circ \lra \sets \, $  is called  {\it sheaf\/}  if for all objects  $ \, U \in \cC \, $,  coverings  $ \, \{U_i \lra U\} \, $  and families  $ \, a_i \in F(U_i) \, $  we have the following.  Let  $ \, p_{ij}^{\,1} : U_i \times_U U_j \lra U_i \, $,  $ \, p_{ij}^{\,2} : U_i \times_U U_j \lra U_j \, $  denote the natural projections and assume $ \, F(p_{ij}^1)(a_i) = F(p_{ij}^2)(a_j) \in F(U_i \times_U U_j) \, $  for all  $ i $,  $ j \, $.  Then  $ \; \exists\,! \; a \in F(U) \; $  whose pull-back to  $ F(U_i) $  is  $ a_i \, $, for every  $ i \, $.
\end{definition}

\smallskip

   We are ready for the sheafification of a functor in this very general setting.

\smallskip

\begin{definition}
   Let  $ \cC $  be a site and let  $ \, F: \cC^\circ \lra \sets \, $  be a functor (in a word, it is a set-valued ``presheaf'' on  $ \cC^\circ \, $).  A  {\it sheafification\/}  of  $ F $  is a sheaf  $ \, \tF : \cC^\circ \lra \sets \, $  with a natural transformation $ \, \alpha : F \lra \tF \, $,  such that:
                                             \par
 \noindent   \; {\it i)} \,  for any  $ \, U \in \cC \, $  and  $ \, \xi, \eta \in F(U) \, $  such
that  $ \, \alpha_U(\xi) = \alpha_U(\eta) \, $  in  $ \tF(U) $,
there is a covering  $ \, \{ \sigma_i : U_i \lra U \} \, $  such
that  $ \, F(\sigma_i)(\xi) = F(\sigma_i)(\eta) \, $  in  $ F(U_i)
\, $;
                                             \par
 \noindent   \; {\it ii)} \,  for any  $ \, U \in \cC \, $  and any  $ \, \xi \in \tF(U) \, $,
there is a covering  $ \, \{ \sigma_i : U_i \lra U \} \, $  and
elements  $ \, \xi_i \in F(U_i) \, $  such that  $ \,
\alpha_{U_i}(\xi_i) = \tF(\sigma_i)(\xi) \, $  in  $ \tF(U_i) \,
$.
\end{definition}

\smallskip

   The next theorem states the fundamental properties of the sheafification.

\smallskip

\begin{theorem} \label{sheafif-property}
  (cf.~\cite{vst})
%
%
 \, Let  $ \cC $  be a site,  $ \, F : \cC^\circ \! \lra \sets \, $  a functor.
                                             \par
 \noindent   \; {\it i)} \,  If  $ \tF $  is a sheafification of  $ F $  with  $ \, \alpha : F \lra \tF \, $,  then any morphism  $ \, \psi : F \lra G \, $,  with  $ G $  a sheaf, factors uniquely through $ \tF $.
%
%
                                             \par
 \noindent   \; {\it ii)} \,  $ F $  admits a sheafification  $ \tF $,  unique up to a canonical isomorphism.
\end{theorem}

\medskip

\noindent
 We shall use this construction for  $ \, \cC \! = \! \aschemes $,  or equivalently  $ \, \cC^\circ \! = \! \salg \, $.

\smallskip

\begin{observation} \label{localiso}
   Let  $ \, F : \aschemes^\circ \lra \sets \, $  be a functor,  $ \tF $  its sheafification.  Then  $ \tF_A $  is the sheafification of  $ F_A $  in the usual sense, that is the sheafification of  $ F $  as sheaf defined on the topological space  $ \spec(A) $.  In particular, since a sheaf and its sheafification are
locally isomorphic, we have that  $ \, F_{A,p} \cong \tF_{A,p} \,
$,  i.e.~they have isomorphic stalks (via the natural map  $ \,
\alpha : F \lra \tF \, $)  at any  $ \, p \in \spec(A) \, $,  for
all superalgebras  $ A \, $.  To ease the notation we shall drop
the suffix  $ A $  and write just  $ F_p $  instead of  $ F_{A,p}
\, $.
\end{observation}

\smallskip

   The rest of this section is devoted to prove the following result:

\smallskip

\begin{theorem} \label{sheaf-iso}
   Let  $ \, F, G : \salg\lra \sets \, $  be two functors, with  $ G $  sheaf.  Assume we have a natural transformation  $ \, F \lra G \, $,  which is an isomorphism on local superalgebras, i.e.~$ \, F(R) \cong G(R) \, $  (via this map) for all local superalgebras  $ R \, $.  Then  $ \, \tF \cong G \, $.  In particular,  $ \, F \cong G \, $  if also  $ F \, $  is a sheaf.
\end{theorem}

\smallskip

\begin{lemma} \label{localsheaf}
  Let  $ \, F : \salg \lra \sets \, $  be a functor; for  $ \, p \in \spec(A) \, $,  let $ \, F_p = \varinjlim F(R) \, $,  where the direct limit is taken for the rings  $ R $  corresponding to the open affine subschemes of  $ \uspec(A) $  containing  $ p \, $.  Then  $ \, F_{p} = F(A_p) \, $.
\end{lemma}

\begin{proof}
 By Yoneda's Lemma, we have
%
%
 \vskip-15pt
  $$  F_{p} = \varinjlim F(R)=\varinjlim \Hom(h_R, \! F)=
\Hom(h_{\varinjlim R}, \! F) = \Hom(h_{A_p}, \! F)=F(A_p)  $$
 \vskip-3pt
\noindent
 as  $ \varinjlim $  and  $ \Hom $  commute (cf.~\cite{la}, p.~141)  and  $ \, A_p = \varinjlim R \, $ (cf.~\cite{am}, p.~47).
\end{proof}

\smallskip

\begin{lemma} \label{localalgebra}
   Let  $ \, A \in \salg \, $,  $ \, p \in \spec(A_0) \, $.  Then  $ A_p $  (= the localization  at  $ p $  of  $ A $  as an  $ A_0 $--module)  is a local superalgebra, whose maximal ideal is  $ \, \mathfrak{m} \! = \! \big( \mathfrak{m}_0 , {(A_1)}_p \big) \, $,  where  $ \mathfrak{m}_0 $  is the maximal ideal in the algebra  $ \, {(A_0)}_p = {(A_p)}_0 \, $.
\end{lemma}

\begin{proof}
   From  $ \, A = A_0 \oplus A_1 \, $  we get  $ \, A_p = {(A_0)}_p \oplus {(A_1)}_p \, $,  and clearly this is a superalgebra with  $ \, {(A_p)}_0 = {(A_0)}_p \, $,  $ \, {(A_p)}_1 = {(A_1)}_p \; $.  Now let us consider  $ \, \mathfrak{m} := \big( \mathfrak{m}_0 , {(A_1)}_p \big) = \mathfrak{m}_0 + {(A_1)}_p \; $.  By the above,  $ \, \mathfrak{m} \neq A_p = {(A_0)}_p \oplus (A_1)_p \; $.  Now take  $ \, x \not\in \mathfrak{m} \, $:  then  $ \, x = x_0 + x_1 \, $  with  $ \, x_0 \in {(A_0)}_p \, $,  $ \, x_1 \in {(A_1)}_p \, $,  so  $ x_0 $  is invertible in  $ \, {(A_0)}_p \subseteq {(A_1)}_p \, $  and  $ x_1 $ is nilpotent, hence  $ x $  is invertible.
\end{proof}

\smallskip

\begin{proposition} \label{sheafiso}
   Let  $ \, F, G : \salg \lra \sets \, $  be local functors and let  $ \, \alpha : F \lra G \, $  be a natural transformation.  Assume that  $ \, F_A \cong G_A \, $  via  $ \alpha $,  where  $ F_A $  and  $ G_A $  denote the ordinary sheaves corresponding to the restrictions of  $ F $  and  $ G $  to the category of open affine subschemes in  $ \uspec(A) $  (morphisms given by the inclusions).  Then  $ \alpha $  is an isomorphism, hence  $ \, F \cong G \, $.
\end{proposition}

\begin{proof}
   We can certainly write an inverse for  $ \alpha_A $  for every object  $ A $,  the problem is to see if it is well behaved on the arrows.  However, this is true because  $ \alpha $  is a natural transformation.
\end{proof}

\smallskip

   We are ready for the proof of  Theorem \ref{sheaf-iso}:

\bigskip

\noindent
 {\it Proof of  Theorem \ref{sheaf-iso}.}  Assume first  $ F $  and  $ G $  are sheaves. Since  $ \, F(R) \cong G(R) \, $  for all local algebras  $ R \, $,  by  Lemma \ref{localsheaf}  this implies that  $ \, F_{p} \cong G_{p} \, $  for all  $ \, p \in \spec(A) \, $,  for all superalgebras  $ A \, $.  Hence  $ \, F_A \cong G_A \, $  by  \cite{ha}, ch.~II, \S 1.1.  By  Proposition \ref{sheafiso},  we have that  $ \, F \cong G \, $  (all isomorphisms have to be intended via the natural transformation  $ \, \alpha : F \lra \tF \, $).
                                                                        \par
   Now assume  $ F $  is not a sheaf.  We have  $ \, \alpha : F \lra \tF \lra G \, $  by  Theorem \ref{sheafif-property}.  If  $ \, A \in \salg \, $,  restricting our functors to the open affine sets in $ \uspec(A) $  we get  $ \, F_A \ra \tF_A \ra G_A \, $.  By  Observation \ref{localiso},  $ F_A $  and $ \tF_A $  are locally isomorphic via  $ \alpha $,  so  $ \, F_{p} \cong \tF_{A,p} \, $.  By hypothesis  $ \, F(R) \cong G(R) \, $,  so  $ \, F_{p} \cong G_{p} \, $  by  Proposition \ref{localsheaf},  hence  $ \, \tF_{p} \cong G_{p} \, $.  Arguing as before, we get the result.   \hfill \qed\break

\medskip

   Along the same lines, the reader can prove the following proposition:

\medskip

\begin{proposition} \label{sheaf-morph}
   Let  $ \, \phi : F \lra G \, $  be a natural transformation between two local functors from  $ \salg $ to  $ \sets $.  Assume we know  $ \phi_R $  for all local superalgebras  $ R \, $.  Then  $ \phi $  is uniquely determined.
\end{proposition}

%

%
%


\bibliographystyle{amsalpha}

  %
  %

%

%
%


\end{document}